\documentclass [12pt] {article}

\usepackage{titling}
\usepackage{blindtext}
\usepackage {amssymb,latexsym}
\usepackage {amsmath}
\usepackage{graphicx}
\usepackage{amsthm}
\usepackage{relsize}
\usepackage[margin=1in]{geometry}

\usepackage{a4wide}
\usepackage[latin1]{inputenc}
\usepackage{fancyhdr}
\usepackage{pgfplots}
\pgfplotsset{compat=1.14}
\usepackage{amscd}
\usepackage{hyperref}
\usepackage{graphicx}
\usepackage{relsize}
\usepackage{newlfont}
\usepackage{amssymb}
\usepackage{youngtab}
\usepackage{young}
\usepackage{mathtools}
\usepackage{tikz}
\usetikzlibrary{shapes, backgrounds} 
\usepackage{tikz-cd}
\usepackage{amsthm}
\usepackage{enumerate,enumitem}

\DeclareFontFamily{U}{wncy}{}
\DeclareFontShape{U}{wncy}{m}{n}{<->wncyr10}{}
\DeclareSymbolFont{mcy}{U}{wncy}{m}{n}
\DeclareMathSymbol{\Sh}{\mathord}{mcy}{"58} 
\DeclareMathSymbol{\RuK}{\mathord}{mcy}{"4B}
\DeclareMathSymbol{\Run}{\mathord}{mcy}{"6E}
\DeclareMathSymbol{\RuZ}{\mathord}{mcy}{"5A}
\DeclareMathSymbol{\Rua}{\mathord}{mcy}{"61}
\DeclareMathSymbol{\Rud}{\mathord}{mcy}{"64}
\DeclareMathSymbol{\RuD}{\mathord}{mcy}{"44}

\newcommand{\RoM}[1]
{\MakeUppercase{\romannumeral #1}}

\newtheorem {lemma} {Lemma}
\newtheorem {proposition} [lemma] {Proposition}
\newtheorem {theorem} [lemma] {Theorem}
\newtheorem {corollary} [lemma] {Corollary}
\newtheorem {definition}[lemma] {Definition}
\theoremstyle{definition}

\makeatletter
\@addtoreset{equation}{section}

\makeatother

\begin{document}
\setlength{\parindent}{0pt} 
\title {\textbf{A gentle introduction to Drinfel'd associators}}

\author{Martin Bordemann, Andrea Rivezzi and Thomas Weigel}
  
\begin{titlingpage}

\date{}
\maketitle
\vspace{-1cm}

\begin{abstract}
\noindent
In this paper we  give an introduction to Drinfel'd's associator coming from the Knizhnik-Zamolodchikov connections and a self-contained proof of the hexagon and pentagon equations by means of minimal amounts of analysis or differential geometry: we rather use limits of concrete parallel transports.
\end{abstract}

\tableofcontents
\end{titlingpage}

\newpage

%
%
\section*{Introduction}
 \addcontentsline{toc}{section}{Introduction}

This work will not contain any new result since the 
mid-1990s, but is meant to be a pedagogical approach to
the celebrated work by V.G.Drinfel'd about the associator $\Phi$ (constructed by means of the Knizhnik-Zamolodchikov connections) satisfying the hexagon and pentagon equations, see \cite{Dri89} and \cite{Dri90}. 

Since its introduction, Drinfel'd's associator has seen many important applications such as the solution of the problem of
the quantization of Lie bialgebras by P.I.Etingof and 
D.A.Kazhdan \cite{EK96} (1996) (see also
P.~\v{S}evera's work \cite{Sev16} (2016)), D.E.Tamarkin's approach \cite{Tam99} (1999) to 
M.L.Kontsevich's formality theorem in deformation quantization 
\cite{Kon03} (1997),
and the solution of the problem of
the quantization of Lie quasibialgebras by B.Enriquez and G.Halbout \cite{EH10I} and \cite{EH10II} in 2010, see also
\cite{SS15}. Moreover, regarding the associator as a formal power series
in the free algebra generated by two distinct elements,
its coefficients are directly related to multiple zeta values,
see e.g.~\cite[p.209-213]{ES02}, thus
establishing an important link to number theory.

The aim of this work is to give a detailed and --as we hope-- self-contained account of
the definition of the particular associator coming from the Knizhnik-Zamolodchikov connections and the proof of
the hexagon and pentagon identities, which 
--in a nutshell-- means the following: let 
$\mathcal{A}$ be an arbitrary complex unital associative algebra. For any given elements $A,B\in\mathcal{A}$
a Drinfel'd associator $\Phi(A,B)$ is an invertible formal 
power series of non-commutative polynomials in $A,B$ (with formal parameter $\lambda$) for which we assume $\Phi(A,B)^{-1}=\Phi(B,A)$ and which satisfies two identities in the following context: for a positive integer $n\geqslant 2$ a finite family
$\left(A_{ij}\right)_{1\leqslant i\neq j\leqslant n}$
of elements of $\mathcal{A}$
satisfies the \emph{infinitesimal braid relations} iff
\begin{subequations}
	\begin{align}
	A_{ij}-A_{ji}&= 0 \qquad \forall ~1\leqslant i\neq j\leqslant n,
	\label{EqIntroIB-definition-one}\\
	\left[A_{ij}+A_{ik}, A_{jk}\right] &= 0 \qquad \forall i,j,k \in \{1, \ldots, n \} \ \text{ such that } \# \{i,j,k \} = 3, \label{EqIntroIB-definition-two} \\
	\left[A_{ij},A_{kl}\right] &= 0 
	\qquad \forall i,j,k,l \in \{1, \ldots, n \} \ \text{ such that } \# \{i,j,k,l \} = 4.
	\label{EqIntroIB-definition-three}
	\end{align}
\end{subequations}
The \emph{Hexagon Equation} for $\Phi$ is the following
identity for $n=3$,
\begin{equation}
\label{eq:hexagon-equation-introduction}
e^{\lambda\mathbf{i}\pi\left( A_{13}+A_{23}\right)} 
~= ~
\Phi\left(A_{13},A_{12}\right)~
e^{\lambda\mathbf{i}\pi A_{13}}~
\Phi\left(A_{23},A_{13}\right)~
e^{\lambda\mathbf{i}\pi A_{23}}~
\Phi\left(A_{12},A_{23}\right),
\end{equation}
and the \emph{Pentagon Equation} for $\Phi$ is the following identity
for $n=4$,
\begin{equation}\label{EqPentDefPentagonEquation}
\Phi\left(A_{12},A_{23}+A_{24}\right)
~\Phi\left(A_{13}+A_{23},A_{34}\right)
=  \Phi\left(A_{23},A_{34}\right)~
\Phi\left(A_{12}+A_{13},A_{24}+A_{34}\right)~
\Phi\left(A_{12},A_{23}\right).
\end{equation}
The problem of finding solutions is quite non-trivial:
for instance, the naive choice $\Phi(A,B)=1$ would solve the pentagon equation, but for non-commuting $A_{13},A_{23}$ it would clearly not
solve the hexagon equation. \\
We shall only mention, but not treat at all the following nonexhaustive list of important results: there exist
rational associators, see \cite{Dri89}, \cite{Dri90}, and
the work by Bar-Natan, \cite{BNa98} (1998). Other
associators have been constructed linked to the Kashiwara-Vergne conjecture by A.Y.Alekseev, B.Enriquez, C.Torossian \cite{AET10} (2010). Moreover, non-trivial solutions of the pentagon equations automatically satisfy
a certain hexagon equation, see H.Furusho's work \cite{Fur10}.\\
Drinfel'd's original method, see \cite{Dri89}, \cite{Dri90}, and also 
e.g.~\cite{Kas95}, \cite{ES02}, consists in  the comparison of different global solutions (in certain simply connected regions of
$\mathbb{R}^N$) of the linear system (a first order linear partial differential equation) defined by the
Knizhnik-Zamolodchikov connection (see \cite{KZ84}) with respect to their pole structure at certain complex or real
hyperplanes, referred to as `monodromy', see also T.Kohno's work \cite{Koh85}. 
This approach is partially motivated by the theory of complex differential equations with singularities,
compare for instance \cite{CL55} or \cite{Del70}. From a point of view of differential geometry this amounts to the computation
of covariantly constant sections of a trivial vector bundle with respcect to a flat connection which are uniquely determined by their 
value at a given point. \\
As far as we know most of the treatments of the Drinfel'd associator and its identities in the literature are
somewhat sketchy, and in the beginning it had not been so clear to us how much analysis is really needed to understand the details: for instance,
we had been confused by symbols like
`$z^{\lambda B}$' --where $z$ is a complex coordinate in the
some open domain of the complex plane, and $B$ is a formal series in a given associative algebra-- which appear in
some texts without explanation and seem to require a proper treatment of complex logarithms. They resemble `multivalued functions' which --according to Deligne, see \cite[p.37, D\'{e}f.~6.2.]{Del70}-- should be defined as functions on the universal cover of the domain which in turn does not seem to be easily accessible to
computations already in the important case of the doubly punctured plane.
Another difficulty for us has been the understanding of the details of the
limiting procedure encoded in Drinfel'd's `zones' used in particular to prove the pentagon equation.

We have chosen a slightly more elementary method
requiring only rudimentary analysis which may --as we hope-- be interesting even to the most misoanalytic
algebraist: since the value of a covariantly constant section (the solutions of the linear system) at some 
point $x$ can be defined by the parallel transport along a continuous piecewise smooth path joining a reference point to $x$, we find it reasonable to focus on formal parallel transports (in the algebra of all formal power series in a given associative unital complex algebra, $\mathcal{A}[[\lambda]]$), i.e.~first order ordinary linear differential equations along concrete continuous piecewise smooth paths in explicitly
given contractible regions of $\mathbb{R}$ and 
$\mathbb{R}^2$ with respect to some flat formal connection deriving from the Knizhnik-Zamolodchikov connection.
However, in Drinfel'd's approach the `reference point'
is --in some sense-- `at a singularity', and in order to capture that we use the well-known regularization procedure: 
first, we make paths within the domain depend on a strictly
positive `small' parameter $\delta$ such that
in the limit $\delta\to 0$ these
$\delta$-dependent paths $c_\delta$ would be pushed
to the boundary of these regions where the connection becomes singular. Then we compute the $\delta$-dependent
parallel transports within the domain where as usual a
 composition of
paths corresponds to multiplication of the corresponding
parallel transports in $\mathcal{A}[[\lambda]]$. It will turn out that each such parallel transport 
$W^{(c_\delta)}$ factorizes in a product of invertible formal power series
as
\begin{equation}\label{EqIntroFactorizationOfParTrans}
    W^{(c_\delta)} =S^{(c_\delta)} G^{(c_\delta)}
     H^{(c_\delta)}  
\end{equation}
where $S^{(c_\delta)}$ is `singular', i.e.~diverging 
--in powers of $|\ln(\delta)|$-- for
$\delta\to 0$, $G^{(c_\delta)}$ is `good',  i.e.~converging to
a wanted term for $\delta\to 0$, and $H^{(c_\delta)}$ is `harmless', i.e.~converging to $1$ for $\delta\to 0$
where the terms proportional to $\lambda^n$, $n>0$,
tend to zero dominated by a `power law' $\delta^\beta$, $\beta>0$. Harmless terms will turn out to be stable by conjugation with
singular terms. 
Parallel transports along different (composed) paths
having the same initial and final points will be equal due
to the flatness of the used connection --for instance the famous Knizhnik-Zamolodchikov connection-- thereby inducing
algebraic identities: in all the important identities
the singular terms cancel out for all strictly positive 
$\delta$, and the remaining terms
give the wanted identities in the limit $\delta\to 0$.

The paper is organised as follows: $\S$\ref{SecSomeAnalysisETC} is an introduction to (formal) linear parallel transports in
open subsets of $\mathbb{R}^N$ to make explicit the 
methods used in $\S$\ref{SecDrinfeldAssociator} where Drinfel'd's associator and its Hexagon and Pentagon identities are treated.
\\
 We recall in $\S$\ref{SubSecElementaryAnalysis}
the analysis of \emph{piecewise $\mathcal{C}^\infty$ functions}
(in the particular sense of a finite number of discontinuities where
left-hand and right-hand limits exist): it turns out to be quite useful to view them as functions only defined on the open subset of the defining interval given by the complement of the finite set of singularities and not to define values at the singularities except when they are continuous. We believe that
mentioning some details about these functions may be useful when later smooth paths are composed in the
sense of algebraic topology thereby remaining continuous
but being in general only piecewise smooth at the glueing
points, see for instance \cite[p.71, Prop.3.3]{KN63}. It comes with no surprise that
there is a Theorem completely analogous to the Fundamental
Theorem of Analysis for the piecewise case, see 
Theorem \ref{TPiecewiseFundTheoCalc}.\\
In $\S$\ref{SubSecFormalLinearODEs} we briefly 
review certain \emph{formal linear ordinary differential equations}
 (which is a particular case of K.-T.Chen's classical work \cite{Che61})
defined by a left multiplication
in the space of all formal power series with coefficients
in an arbitrary associative unital algebra $\mathcal{A}$
\[
    \dfrac{\mathsf{d}W_{\cdot\alpha}}{\mathsf{d}s}
    =\lambda YW_{\cdot\alpha} 
    ~~~\mathrm{and}~~~W_{\alpha\alpha} =1
\] 
where $Y$ is a suitable formal power series of
piecewise $\mathcal{C}^\infty$-functions taking values in 
$\mathcal{A}$: here each component $Y_r$, $r\in\mathbb{N}$, is in the algebraic tensor product of the algebra of piecewise smooth complex valued functions on an interval and $\mathcal{A}$. Thanks to the presence of the formal parameter, the existence and uniqueness theory of the solutions
$s\mapsto W_{s\alpha}$ (which are required to be continuous)
is quite elementary and algebraic
and does not require any additional topological structures on $\mathcal{A}$. The formulation by the well-known \emph{iterated integrals} (see (\ref{EqFormalLinODEIteratedIntegrals})) is automatical,
and the relations with the `vector fields' $Y$ and their
`propagators' $W_{\beta\alpha}$ with respect to reparametrization is
resumed in Proposition \ref{PFormalLinODEPropagatorPropert}.\\
$\S$\ref{SubSec Norms and limits} is devoted to
a simple framework to consider \emph{limits of propagators}
when the vector fields $Y$ depend on one or more nonformal
`small' parameters $\boldsymbol{\delta}=(\delta,\epsilon,\ldots)$: for these limits, an arbitrary norm is chosen on the complex vector space $\mathcal{A}$
(which will in general not be a normed algebra), and there can be many non equivalent ones. However, for each power of the formal parameter the limits will be done separately,
and it is crucial from the choice of the restricted frame in which
$Y$ and hence $W$ are considered that for each power of the formal
parameter the subspace of $\mathcal{A}$ spanned by all the relevant coefficients is always finite-dimensional, hence the
restriction of the chosen norm to these subspaces is
well-known to be equivalent to the restriction of any other norm on 
$\mathcal{A}$ to the same subspaces. 
This allows to define limits $\delta,\epsilon\to 0$ compatible with the algebra structure independently on the norm
chosen on $\mathcal{A}$, see Proposition \ref{PNormsLimLimitRules}. We also include a definition and
some properties of the subspace of parameter dependent
elements whose norm (order by order) is bounded by powers
of the parameter or by powers of its logarithm, see
Proposition \ref{PFilNormsRBHProperties} and Lemma
\ref{LNormsLimIntegralBounds}. This will become important to make precise the terms appearing in the factorization equation (\ref{EqIntroFactorizationOfParTrans}).
\\
The last two Sections $\S$\ref{SubSecFormalConnectionsAndParallelTransports}
and $\S$\ref{SubSecFlatFormalConnections} deal with \emph{formal
linear connections} $\Gamma$ (in trivial bundles) linking
given \emph{continuous piecewise smooth paths} with formal linear
ODEs: their propagators are the well-known parallel transports of differential geometry, 
see e.g.~\cite[p.68]{KN63}. We recall their  invariance by reparametrization and their multiplicativity
with respect to composition of paths and the simple relations with pull-backs, see Theorem \ref{TFormConnOperationsForParTransp}. In case the connection is
flat (see e.g.~\cite[p.92]{KN63}) parallel transports
are independent of paths in star-shaped open sets of
$\mathbb{R}^N$. We provide the proofs of these theorems,
see Theorem \ref{TFormFlatPathIndep}, Lemma \ref{LFlatFormConnSmoothing}, and Corollary
\ref{CFormFlatContractibleLoops} in detail as well as
the relation to complex coordinates in the Appendix.

In the second part we recall the \emph{Drinfel'd-Kohno (Lie) algebras}, 
\emph{complex configuration spaces} $Y_n$, and the 
\emph{Knizhnik-Zamolodchikov connections}
in $\S$\ref{SubSecDKAlgebrasKnZhConnections}. In $\S$\ref{SubSecDrinfeldAssociatorDefinitionAndProperties}
we give a definition of the \emph{Drinfel'd associator} as 
the non singular part of a parallel transport in the
open unit interval from $\delta>0$ to $1-\epsilon<1$
with $0<\delta,\epsilon\leqslant \frac{1}{4}$: this
has been inspired by the quantity
$G_a(1-a)$ (for $0<a<1$) in 
\cite[p.465]{Kas95}, and we prove the factorization equation (see \cite[p.465, Lemma XIX.6.3]{Kas95}) in
Theorem \ref{TDrinfeldAssDefAssociator}: we are using the trick to split the path $\delta\to 1-\epsilon$ in two
halves (using the midpoint $1/2$) parametrized in an
`exponential way' which give two `square roots' of the parallel
transport each having only one singular factor, see
Lemma \ref{LDrinfeldAssFactorizationHalfAss}, and are 
surprisingly easy to estimate. By parametrization invariance of parallel transports this leads to Drinfel'd's original definition
\cite[p.833, equation (2.1)]{Dri90}.  Having two parameters 
proves to be useful later  when proving the Hexagon and Pentagon identities: this will be done in Sections
$\S$\ref{TheHexagonEquation} and
$\S$\ref{SubSecPentagonEquation}. \\
For the proof of the \emph{Hexagon equation}, see \eqref{eq:hexagon-equation-introduction} and (\ref{EqHexDefHexagonEqn}), in $\S$\ref{TheHexagonEquation}
one chooses a pull-back
of the Knizhnik-Zamolodchikov connection from $Y_3$ to the doubly punctured plane $\mathbb{C}^{\times\times}$, and we consider
the parallel transport along a loop (dependent on the small parameter $\delta$) in the complement of the upper half plane
composed of three line segments
and three lower half circles. Using a certain cyclic symmetry of the doubly punctured plane (given by three simple explicit
complex rational functions) 
we just have to compute the parallel transport along one
of the half circles which gives an exponential function
times a `harmless' term going to $1$ for $\delta\to 0$
whereas the parallel transports along the line segments
give associators and singular terms: the latter are shown to cancel out of the equation.\\
Finally, $\S$\ref{SubSecPentagonEquation} is devoted
to the proof of the \emph{Pentagon Equation}, see
(\ref{EqPentDefPentagonEquation}): here, as in Drinfel'd's work \cite[p.834]{Dri90} we first use the wedge $x_1<x_2<x_3<x_4$
in $\mathbb{R}^4$, set $x_1=0$ and $x_4=1$, and choose five line segments in the remaining $x_2$-$x_3$-plane depending on two scales $\delta$ and $\delta^2$ according to an interpretation of Drinfel'd's `zones', cf.~\cite[p.1454]{Dri89} or
\cite[p.834]{Dri90}. The identities of the parallel transports along these five paths (w.r.t~the pull-back of the Knizhnik-Zamolodchikov connection on $Y_4$)
will give the 
Pentagon Equation
in the limit $\delta\to 0$ after having shown the factorization into singular,
 bounded, and harmless terms and the cancellation of the
 singular terms, see Theorems \ref{TPentEqFiveParTrDeltaPositive} and
  \ref{TPentEqFinalProofOfPentEq}.

\subsection*{Conventions and Notations}

 The symbol $\mathcal{A}$ will always denote a unital associative algebra over the field of all complex numbers $\mathbb{C}$. Unadorned tensor products will
 always refer to tensor products over $\mathbb{C}$,
 i.e.~$\otimes=\otimes_\mathbb{C}$. For a non-empty set $M$ and a complex vector space $E$
 the symbol $\mathrm{Fun}(M,E)$ will denote the complex vector space of all $E$-valued maps $M\to E$.
 Next, our formal parameter
 $\lambda$ is equal to Drinfel'd's $\frac{h}{2\pi\mathbf{i}}$, the latter being denoted by $\overline{h}$ in \cite{Kas95}. 

\subsection*{Acknowledgements}

The authors would like to thank A.Appel, P.Clavier,
F.Corniquel, S.Goette, B.Hurle, D.Manchon, D.Panazzolo and S.Waldmann for useful discussions.


\section{Some analysis for formal parallel transports}
\label{SecSomeAnalysisETC}
      
Let $\mathcal{A}$ be a fixed associative unital complex algebra. We write $\mathcal{A}[[\lambda]]$ for the
$\mathbb{C}[[\lambda]]$-module of all formal power series with coefficients in $\mathcal{A}$ which itself
is an associative unital algebra over the ring
$\mathbb{C}[[\lambda]]$.

\subsection{Elementary analysis of piecewise $\mathcal{C}^\infty$-functions} 
\label{SubSecElementaryAnalysis}

This Section recalls piecewise smooth functions because
later on we shall need parallel transports along composed paths
which are in general not smooth at the `gluing point'.
The generalization to piecewise $\mathcal{C}^k$-functions
is immediate, but not necessary for the sequel.
      
For two given real numbers $a<b$ let
$]a,b[$ denote the open interval of all real numbers $s$ such that $a<s<b$. 
Recall the complex vector space $\mathcal{C}^k\big(]a,b[,\mathbb{C}\big)$ of all complex-valued functions on $]a,b[$ whose $r$th derivatives all exist and are continuous
for all orders $r$ if $k=\infty$, respectively.
By means of pointwise multiplication it is a unital 
commutative and associative complex algebra. Write
$\mathcal{C}^\infty\big(U,\mathbb{C}\big)$ for the obvious generalization whenever $U\subset \mathbb{R}$ is a
finite union of open intervals.\\
Next, we have to speak about piecewise such functions
in the more restricted sense of finitely many `nice' discontinuities, see e.g.~\cite[p.158, Def.~IV.4.2]{AF88}:
we have chosen a formulation avoiding the choice of values at the `singular points'. More precisely: let
$[a,b]$ be the closed interval of all real
numbers $s$ such that $a\leqslant s\leqslant b$. Choose
a finite subset $D\subset [a,b]$ which contains $a,b$, hence $D$ is of the general form
$D=\{a=a_0<a_1<\cdots<a_m
<a_{m+1}:=b\}$ where $m$ is a non-negative integer.
We shall refer to $D$ as the (potential) \emph{singular set}
and to its open dense complement $[a,b]\setminus D$ as the \emph{regular set}. Define the space of all \textbf{piecewise $\mathcal{C}^\infty$-functions on $[a,b]$ with singular set $D\subset [a,b]$}
in the following slightly unusual way:
\begin{eqnarray}
  \mathcal{C}_D^\infty([a,b],\mathbb{C}) & := &
   \Big\{ f\in \mathcal{C}^\infty([a,b]\setminus D,\mathbb{C})~\Big|~
       \forall~r,i\in\mathbb{N}~\mathrm{with}~ 
       1\leqslant i\leqslant m+1: 
       \nonumber \\
       & &  
          ~~  \lim_{\epsilon\downarrow 0}
       f^{(r)}(a_i-\epsilon) ~\mathrm{exists \ and}~~
            \forall~0\leqslant i\leqslant m:
            \lim_{\epsilon\downarrow 0}
            f^{(r)}(a_i+\epsilon)~
            \mathrm{exists}\Big\} \nonumber \\
            \label{EqPiecewiseDefCKD}
\end{eqnarray}
where $\epsilon \downarrow 0$ means that only strictly positive real numbers $\epsilon$ are considered in the
limit. In other words, each element $f$ of
$\mathcal{C}_D^\infty([a,b],\mathbb{C})$ is a 
$\mathcal{C}^\infty$-function
outside the singular set $D$, and left-side and
right-side limits of all the higher derivatives at the singular points both have to exist, but need not be equal. It follows that the restriction of each
$f^{(r)}$, $r\in\mathbb{N}$, to each open interval $]a_i,a_{i+1}[$, $0\leqslant i\leqslant m$, 
uniquely extends to a continuous function $f^{(r)}_i$ defined on the
closed interval $[a_i,a_{i+1}]$ by means of the
right-side limit at $a_i$ and the left-side limit at
$a_{i+1}$.
It is clear that for each finite set $D$ such that $\{a,b\}\subset D\subset [a,b]$   the
complex vector space $\mathcal{C}_D^\infty([a,b],\mathbb{C})$
is a unital commutative associative algebra (with respect to pointwise multiplication on the regular subset) thanks to the Leibniz rule of higher order derivatives, see
e.g.~\cite[p.~178]{AF88}, and we have the canonical maps
\begin{equation*}
  ~\forall~D\subset D':
     \mathcal{C}_D^\infty([a,b],\mathbb{C})
     \hookrightarrow \mathcal{C}_{D'}^\infty([a,b],\mathbb{C})
\end{equation*}
induced by the obvious restrictions which are injections since the regular sets are dense in
the closed interval $[a,b]$. These injections
are morphisms of unital complex algebras. We shall not denote them explicitly. Moreover,
the usual derivative (defined only on the regular set $[a,b]\setminus D$) clearly is compatible with
left- and right-sided limits and thus induces a derivation
of algebras
\begin{equation}\label{EqPiecewiseDerivative}
   \mathcal{C}_D^\infty([a,b],\mathbb{C})\to \mathcal{C}_D^\infty([a,b],\mathbb{C}):f\mapsto \frac{\mathsf{d}f}{\mathsf{d}s}.
\end{equation}
We shall very often need the following subalgebra
of $\mathcal{C}_D^\infty([a,b],\mathbb{C})$:
\begin{equation} \label{EqPiecewiseCkContinuous}
  \mathcal{C}_D^\infty([a,b],\mathbb{C})^0:=
  \big\{f\in\mathcal{C}_D^\infty([a,b],\mathbb{C})~\big|~
     f~\mathrm{extends~to~a~continuous~function~}
     [a,b]\to \mathbb{C} \big\}.
\end{equation}
Clearly, the existence of this continuous extension is equivalent to the fact that left-side and right-side limits at the singular points of $f$ (but not necessarily of its higher order derivatives) coincide, whence it is unique if it exists. Note that for these spaces the usual evaluation at $\alpha\in [a,b]$ of a function makes
sense as opposed to the general case where no value
at the singular points is defined.
We shall also need to compose these piecewise 
$\mathcal{C}^\infty$-functions: in addition to the closed interval $[a,b]$ and the singular subset 
$\{a,b\}\subset D\subset [a,b]$ choose another closed
interval $[a',b']$ (where $a'<b'$ are real numbers)
and a finite subset $D'=\{a'=a'_0<a'_1<\cdots<a'_{m'}<a'_{m'+1}=b'\}$ of $[a',b']$ such that
$\{a',b'\} \subset D'\subset [a',b']$.  We shall call 
a piecewise $\mathcal{C}^{\infty}$ function $\theta$ on $[a',b']$
with singular set $D'$ \emph{compatible with $[a,b]$ and $D$}
if the following condition is satisfied:
\begin{equation}\label{EqPiecewiseCompositionCompatible}
   \theta\big([a',b'] \setminus D \big)\subset [a,b]\subset \mathbb{R}~~~\mathrm{and}~~~
   \theta|^{-1}(D)~\mathrm{is~finite~subset~of~}[a',b']
\end{equation}
where $\theta|$ denotes the  
$\mathcal{C}^{\infty}$-function 
$[a',b']\setminus D'\to [a,b]$ outside its singular set.
It is immediate that the composition $f\circ \theta$
is a well-defined function on 
$[a',b']\setminus \big(D'\cup\theta|^{-1}(D) \big)$,
and the iterated chain rule (also called Faa di Bruno Theorem, see e.g.~\cite[p.291, equation (3)]{AF88})
shows that it is a
$\mathcal{C}^{\infty}$-function.
The left-side and right-side limits of the $r$th derivative $f\circ \theta$ at the singular points in 
$D'\cup\theta|^{-1}(D)$ exist which easily follows from
the continuity of the continuous extensions $\theta^{(r)}_j:[a'_j,a'_{j+1}]\to [a,b]$ of the restriction of $\theta^{(r)}$ to $]a'_j,a'_{j+1}[$
for all integers $r\geqslant 0$ and $0\leqslant j\leqslant m'$.
Hence, we can define the composition
\begin{equation}\label{EqPiecewiseDefComposition}
   f\circ \theta \in 
\mathcal{C}^{\infty}_{D'\cup \theta|^{-1}(D)}
     \big([a',b'],\mathbb{C}\big).
\end{equation}
It is not hard to see that the chain rule works for
this composition and differentiation (\ref{EqPiecewiseDerivative}).

Next, we need to use the well-known \emph{Riemann integral}: note that
for every element $f\in \mathcal{C}_D^\infty([a,b],\mathbb{C})$
and $\alpha,\beta\in [a,b]$ we can define the Riemann integral
\begin{equation}\label{EqPiecewiseIntegralSign}
   I_\alpha^\beta(f):=\left\{ 
     \begin{array}{cl}
       \int_\alpha^\beta \hat{f}(s)ds & \mathrm{if}~
                                         \alpha\leqslant \beta \\
      - \int_\beta^\alpha \hat{f}(s)ds & \mathrm{if}~
       \alpha\geqslant \beta        
     \end{array}\right.
\end{equation}
where $\hat{f}$ is any extension of $f$ from $[a,b]\setminus D$ to $[a,b]$ (for instance
$\hat{f}(a_i)=0$ for all $0\leqslant i\leqslant N+1$: it is well-known that any such extension  is Riemann integrable and that the integral does not depend on the extension, that is which values of $\hat{f}$
are chosen at the singular points contained in the domain of integration, see e.g.~\cite[p.273]{Lan97}. Recall Chasles's rule for any
$f\in \mathcal{C}_D^\infty([a,b],\mathbb{C})$:
\begin{equation*}
   \forall~\alpha,\beta,\gamma\in [a,b]:~~~
    I_\alpha^\gamma(f)=I_\alpha^\beta(f)
                       +I_\beta^\gamma(f).
\end{equation*}
Recall that the complex-valued function $[a,b]\to 
\mathbb{C}$ defined for every $f\in \mathcal{C}_D^\infty([a,b],\mathbb{C})$ by
\begin{equation*}
    I_\alpha(f)(s):= I_\alpha^s(f)
\end{equation*}
is the usual \emph{primitive of $f$}. We resume the following
properties of the primitive which are well-known variants of the fundamental theorem of calculus and
standard integration techniques:

\begin{theorem}\label{TPiecewiseFundTheoCalc}
	Let $a,b \in \mathbb{R}$ with $a<b$, and let $D \subset \mathbb{R}$ be a finite set such that $\{a,b\}\subset D\subset [a,b]$. Then for any 
	$f,g\in C^\infty_D\big([a,b],\mathbb{C}\big)$ and
	$h\in C^{\infty}_D\big([a,b],\mathbb{C}\big)^0$ the following holds:
	\begin{enumerate}
		\item For any $\alpha\in [a,b]$ the primitive
		  $I_\alpha$ defines a $\mathbb{C}$-linear map  $C^\infty_D\big([a,b],\mathbb{C}\big)
		  \to C^{\infty}_D\big([a,b],\mathbb{C}\big)^0$ 
		  whence $I_\alpha(f)$ is always continuous.  Moreover,
		  \begin{equation}
		  \label{EqPiecewisePrimitiveNormalization}
		  I_\alpha(f)(\alpha)=0.
		  \end{equation}
		\item Fundamental theorem of calculus: for the derivatives (in the sense of (\ref{EqPiecewiseDerivative})) we get 
		  \begin{equation}
		  \label{EqPiecewisePrimitiveFundaTheo}
		    \frac{\mathsf{d}I_\alpha(f)}
		         {\mathsf{d}s}=f~~\mathrm{and}~~
		      I_\alpha\left(\frac{\mathsf{d}h}
		      {\mathsf{d}s}\right)= h-h(\alpha).
		  \end{equation}
		  Moreover, any element $F\in C^{\infty}_D\big([a,b],\mathbb{C}\big)^0$ satisfying $\frac{\mathsf{d}F}
		  {\mathsf{d}s}=f$ and $F(\alpha)=0$ is equal
		  to $I_\alpha(f)$.
		\item Let $a',b' \in \mathbb{R}$ with 
		$a'<b'$, let $D'$ be a finite set with
		$\{a',b'\}\subset D' \subset [a',b']$, let
		$\theta\in \mathcal{C}^{\infty}_{D'}
		\big([a',b'],\mathbb{R}\big)^0$ such that
		$\theta$ is compatible with $[a,b]$ and $D$,
		see (\ref{EqPiecewiseCompositionCompatible}).\\
		Then the composition
		$f\circ \theta$ (see
		 (\ref{EqPiecewiseDefComposition})) is in 
		 $\mathcal{C}^{\infty}_{D'\cup \theta|^{-1}(D)}
		 \big([a',b'],\mathbb{C}\big)$. Moreover,
		for each $\alpha'\in[a',b']$ there is the
		usual `change-of-variables-rule'
		\begin{equation}\label{EqPiecewiseChangeOfVar}
		  I_{\theta(\alpha')}(f)\circ \theta
		   = I_{\alpha'}\left((f\circ\theta)
		   \frac{\mathsf{d}\theta}
		   {\mathsf{d}s}\right)
		\end{equation}
		where both sides of the preceding equation are elements of 
$\mathcal{C}^{\infty}_{D'\cup \theta|^{-1}(D)}
		\big([a',b'],\mathbb{C}\big)^0$.
		\end{enumerate}
\end{theorem}

\noindent \textbf{Remark for the proof}: see e.g.~Lang's book \cite[p.272-274]{Lan97} for the proof of all the statements.
 The fact that the primitive
of a piecewise continuous function is continuous
is standard and follows from Chasles's rule and the
fact that piecewise continuity of $f$ implies that 
any extension $\hat{f}$ of $f$ is bounded on $[a,b]$.
The other statements follow on the regular set from their well-known analogues for continuous functions.
Note that the continuity of $F$ in statement
$ii.)$ is crucial: the derivative of $F-I_\alpha(f)$ vanishes on the regular points which implies by continuity that $F-I_\alpha(f)$ is an overall constant
continuous function being zero thanks to $F(\alpha)=0=I_\alpha(f)(\alpha)$. The last  equation (\ref{EqPiecewiseChangeOfVar}) follows from the preceding consideration by derivation of both sides on the regular points, the fundamental theorem
(left equation in 
(\ref{EqPiecewisePrimitiveFundaTheo}) and the normalization condition
(\ref{EqPiecewisePrimitiveNormalization}).\\  

We shall call a triple $(\theta,[a',b'],D')$
consisting of a continuous map $\theta:[a',b']\to [a,b]$
satisfying the hypotheses of statement $iii.)$ of the preceding Theorem \ref{TPiecewiseFundTheoCalc}
a \textbf{continuous piecewise $\mathcal{C}^{\infty}$ reparametrization of $\big([a,b],D\big)$}.

\subsection{Formal linear ODE's} 
\label{SubSecFormalLinearODEs}

In this Section we review a particular case of the general theory
described in K.-T.Chen's classical work \cite[p.110-115]{Che61}.\\
Fix an arbitrary associative unital complex algebra $\mathcal{A}$, two arbitrary real numbers $a,b$ such that $a<b$, an arbitrary finite subset $D$ of the real numbers such that $\{a,b\}\subset D\subset [a,b]$.
Consider the complex commutative associative unital
algebra $\mathcal{C}_D^\infty\big([a,b],\mathbb{C}\big)$
of all piecewise complex-valued $\mathcal{C}^\infty$-functions
on $[a,b]$ with potential singular set $D$,
see $\S$\ref{SubSecElementaryAnalysis} for details. Form the algebraic
tensor product 
$\mathcal{C}_D^\infty\big([a,b],\mathbb{C}\big)\otimes
\mathcal{A}$  and consider the $\mathbb{C}[[\lambda]]$-module of all
formal power series with coefficients in
$\mathcal{C}_D^\infty\big([a,b],\mathbb{C}\big)\otimes
\mathcal{A}$,
\begin{equation}
 \label{EqFormalLinODEPiecewiseTensorAofLambda}
    \Big(\mathcal{C}_D^\infty\big([a,b],\mathbb{C}\big)
    \otimes
     \mathcal{A}\Big)[[\lambda]].
\end{equation}
Note that the preceding $\mathbb{C}[[\lambda]]$-module
is again an associative unital algebra over
$\mathbb{C}[[\lambda]]$ (with respect to the tensor product multiplication and the Cauchy product of formal power series). Note further that each element
$F$ in this algebra can canonically be considered as a function from
$[a,b]\setminus D$ to $\mathcal{A}[[\lambda]]$, and we shall sometimes use the notation $s\mapsto F(s)=\sum_{r=0}^\infty
F_r(s)\lambda^r$. However, the algebra (\ref{EqFormalLinODEPiecewiseTensorAofLambda})  is in general
much smaller than the space of all `suitable' functions $[a,b]\setminus D$ to $\mathcal{A}[[\lambda]]$ where
things like continuity are more delicate to deal with, in particular in the important case of an 
infinite-dimensional $\mathcal{A}$. Observe that every element $F$ in the
algebra (\ref{EqFormalLinODEPiecewiseTensorAofLambda})
is thus a formal power series $F=\sum_{r=0}^\infty F_r\lambda^r$ such that for each non-negative integer $r$
the component $F_r$ is a \textbf{finite sum of terms} 
of the form $f\otimes A$ with
$f\in \mathcal{C}_D^\infty\big([a,b],\mathbb{C}\big)$ and
$A\in \mathcal{A}$. We shall discuss this important point
in more detail in $\S$\ref{SubSec Norms and limits}.

Tensoring the usual derivative $\frac{\mathsf{d}}{\mathsf{d}s}$, see 
(\ref{EqPiecewiseDerivative}),  with the identity map on $\mathcal{A}$ and extending on formal
power series in the usual `componentwise' way we get a 
`derivative' of the algebra (\ref{EqFormalLinODEPiecewiseTensorAofLambda})
which is again a derivation of algebras.
We shall denote it by
the same symbol $\frac{\mathsf{d}}{\mathsf{d}s}$.
In a completely analogous way we can extend the Riemann integral $I_\alpha^\beta$ and the primitive $I_\alpha$ to the algebras 
(\ref{EqFormalLinODEPiecewiseTensorAofLambda}) where
we shall continue to use the same symbols. It is obvious that $I_\alpha^\beta$ takes its values in
$\mathcal{A}[[\lambda]]$ and that all the statements of Theorem \ref{TPiecewiseFundTheoCalc}
remain true when $f,g$ and $h$ are replaced by elements in the corresponding algebras (\ref{EqFormalLinODEPiecewiseTensorAofLambda}).

Fix $Y\in \Big(\mathcal{C}_D^\infty\big([a,b],\mathbb{C}\big)
\otimes
\mathcal{A}\Big)[[\lambda]]$. We consider the following formal
linear ordinary differential equation (formal linear ODE)
\begin{equation}\label{EqFormalLinODEDefinition}
   \left\{\begin{array}{ccc}
      \frac{\mathsf{d}\omega}{\mathsf{d}s} &=&
          \lambda Y\omega \\
        \omega(\alpha)& =&\omega_\alpha
   \end{array}\right.
\end{equation}
where $\alpha\in[a,b]$, the `initial value' $\omega_\alpha$ is an element of
$\mathcal{A}[[\lambda]]$, and we look for solutions
\begin{equation*}
   \omega\in \Big(\mathcal{C}_D^{\infty}\big([a,b],\mathbb{C}\big)^0
   \otimes
   \mathcal{A}\Big)[[\lambda]],
\end{equation*}
of the differential equation (\ref{EqFormalLinODEDefinition}) whence $\omega$ \textbf{is required to be continuous} and piecewise $\mathcal{C}^{\infty}$.
The theory of existence and uniqueness of these
formal linear ODEs is well-known to 
be much simpler than the one of the usual differential equations.
Indeed, 
first there is the usual reformulation in terms of
integral equations known from usual ODE theory: suppose first that $\omega$ is
a continuous piecewise $\mathcal{C}^{\infty}$ solution of
(\ref{EqFormalLinODEDefinition}). Taking primitives on
both sides gives --using in Theorem \ref{TPiecewiseFundTheoCalc} the second equation of (\ref{EqPiecewisePrimitiveFundaTheo})-- the `integral equation'
\begin{equation}\label{EqFormalLinODEIntEq}
     \omega =\omega_\alpha 
      +\lambda I_\alpha\big(Y\omega\big)
\end{equation}
where $\omega_\alpha$ is considered as the
constant function on $[a,b]$ with value $\omega_\alpha$.
On the other hand, if the continuous piecewise 
$\mathcal{C}^{\infty}$ element $\omega$ is a solution of the
formal integral equation (\ref{EqFormalLinODEIntEq}), then
$\omega(\alpha)=\omega_\alpha$ by (\ref{EqPiecewisePrimitiveNormalization}), and differentiation of the integral equation --using
the first equation of (\ref{EqPiecewisePrimitiveFundaTheo})-- gives the
formal linear ODE  (\ref{EqFormalLinODEDefinition}).

Next, solving the formal integral equation (\ref{EqFormalLinODEIntEq}) is quite simple due to the presence of the factor $\lambda$ in front of $Y$: consider the following $\mathbb{C}[[\lambda]]$-linear maps 
\begin{equation*}
\begin{split}
   L_Y:\Big(\mathcal{C}_D^{\infty}\big([a,b],
   \mathbb{C}\big)^0
   \otimes
   \mathcal{A}\Big)[[\lambda]]
   &\to\Big(\mathcal{C}_D^\infty\big([a,b],\mathbb{C}\big)
   \otimes
   \mathcal{A}\Big)[[\lambda]] \\
   F &\mapsto L_Y(F):=YF
   \end{split}
\end{equation*}
and
\[
    I_\alpha:
    \Big(\mathcal{C}_D^\infty\big([a,b],\mathbb{C}
    \big)
    \otimes
    \mathcal{A}\Big)[[\lambda]]
    \to
    \Big(\mathcal{C}_D^{\infty}
    \big([a,b],\mathbb{C}\big)^0
    \otimes
    \mathcal{A}\Big)[[\lambda]].
\]
Then the composition $I_\alpha\circ L_Y$ is a well-defined $\mathbb{C}[[\lambda]]$-linear endomorphism of the
$\mathbb{C}[[\lambda]]$-module 
$\Big(\mathcal{C}_D^{\infty}\big([a,b],\mathbb{C}\big)^0
\otimes
\mathcal{A}\Big)[[\lambda]]$. Hence, the formal integral
equation (\ref{EqFormalLinODEIntEq}) can be rewritten as
\begin{equation}\label{EqFormalLinODESolutionIntEq}
  \big(\mathrm{id}-\lambda I_\alpha\circ L_Y\big)(\omega)=\omega_\alpha,~~~
  \mathrm{hence}~~~\omega 
  = \big(\mathrm{id}
      -\lambda I_\alpha\circ L_Y\big)^{-1}(\omega_\alpha)
      =\sum_{r=0}^\infty \lambda^r 
      \big(I_\alpha\circ L_Y\big)^{\circ r}
      (\omega_\alpha)
\end{equation}
since it is obvious --thanks to the presence of the factor
$\lambda$-- that the formal series $\mathrm{id}-\lambda I_\alpha\circ L_Y$ is always invertible in the algebra of all $\mathbb{C}[[\lambda]]$-linear endomorphisms of 
$\Big(\mathcal{C}_D^{\infty}\big([a,b],\mathbb{C}\big)^0 \otimes \mathcal{A}\Big)[[\lambda]]$ seen as a $\mathbb{C}[[\lambda]]$-module
by the usual geometric series formula. Note that
the formula (\ref{EqFormalLinODESolutionIntEq}) is 
very often written out in terms of \textbf{iterated integrals}:
\begin{equation}\label{EqFormalLinODEIteratedIntegrals}
  \omega(s) = \omega_\alpha 
  + \sum_{r=1}^\infty\lambda^r\Bigg(
     \int_{\alpha}^{s}
     \left(\hat{Y}(s_1)\int_{\alpha}^{s_1}
     \left(\hat{Y}(s_2)\int_{\alpha}^{s_2}\left(\cdot \cdot
      \int_{\alpha}^{s_{r-1}}\hat{Y}(s_r)ds_r
      \right)
      \cdot \cdot ds_3\right)
     ds_2\right)ds_1\Bigg)\omega_\alpha
\end{equation}
where $\hat{Y}$ denotes any extension of $Y$ from 
$[a,b]\setminus D$ to $[a,b]$.
We shall write $W_{\cdot\alpha}:=(s\mapsto W_{s\alpha})$
for the particular solution $\omega$ of the formal ODE
(\ref{EqFormalLinODEDefinition}) with initial condition
$\omega_\alpha=1$, the unit element of the algebra
$\Big(\mathcal{C}_D^{\infty}\big([a,b],\mathbb{C}\big)^0
\otimes
\mathcal{A}\Big)[[\lambda]]$, hence
\begin{equation}\label{EqFormalLinODEDefinitionOfW}
\left\{\begin{array}{rcc}
\dfrac{\mathsf{d}W_{\cdot \alpha}}{\mathsf{d}s}& =&
\lambda Y W_{\cdot \alpha} \\
W_{\cdot \alpha}(\alpha) =W_{\alpha\alpha} &=&1
\end{array}\right. 
\end{equation}
and refer to it as the \emph{fundamental solution of
the formal ODE (\ref{EqFormalLinODEDefinition})
normalized at $\alpha$}, see e.g.~\cite[p.69]{CL55}.
Moreover, we shall refer to the value of the fundamental
solution $W_{\cdot \alpha}$ at $\beta\in [a,b]$,
\begin{equation}\label{EqFormalLinODEDefOfPropagatorW}
W_{\beta\alpha}:=W_{\cdot\alpha}(\beta)\in
\mathcal{A}[[\lambda]]
\end{equation}
 as the \emph{propagator} (from $\alpha$ to $\beta$).\\
We collect some properties of the above formal linear
ODE's in the following
\begin{proposition}
	\label{PFormalLinODEPropagatorPropert}
 With the above-mentioned hypotheses and notations we have the following:
 \begin{enumerate}
 	\item Every formal linear ODE (\ref{EqFormalLinODEDefinition}) has a unique (continuous!) solution $\omega$ given by the formulas (\ref{EqFormalLinODESolutionIntEq}) or
 	(\ref{EqFormalLinODEIteratedIntegrals}). It can be
 	expressed by the fundamental solution $W_{\cdot\alpha}$ normalized at $\alpha$ in the following way:
 	\begin{equation}\label{EqFormalLinODEGenInitCond}
 	    \omega=W_{\cdot\alpha}\omega_\alpha.
 	\end{equation}
 	\item \textbf{Groupoid properties}: Every fundamental solution $W_{\cdot\alpha}$
 	 has only invertible values in $\mathcal{A}[[\lambda]]$, and for all $\alpha,\beta,\gamma\in [a,b]$ we have the following identities for the propagators 
 	 \begin{equation}\label{EqFormalLinODEGroupoidProp}
 	    W_{\alpha\alpha}=1,
 	    ~~~W_{\gamma\beta}W_{\beta\alpha}
 	          =W_{\gamma\alpha},~~~
 	          W_{\alpha\beta}=W_{\beta\alpha}^{-1}.
 	 \end{equation}
 	 \item \textbf{Reparametrization}: Let $a',b' \in \mathbb{R}$ with 
 	 $a'<b'$, let $D'$ be a finite set with
 	 $\{a',b'\}\subset D' \subset [a',b']$, let
 	 $\theta\in \mathcal{C}^{\infty}_{D'}
 	 \big([a',b'],\mathbb{R}\big)^0$ be a continuous piecewise $\mathcal{C}^{\infty}$ reparametrization of $\big([a,b],D\big)$, see
 	 (\ref{EqPiecewiseCompositionCompatible}).
 	 Let $\alpha'\in [a',b']$, and
 	 let $W_{\cdot\theta(\alpha')}$ the fundamental
 	 solution of (\ref{EqFormalLinODEDefinitionOfW}) normalized at
 	 $\theta(\alpha')$.
 	 Then $W'_{\cdot\alpha'}:= W_{\cdot\theta(\alpha')}\circ\theta$
 	 is a fundamental solution normalized at $\alpha'$
 	 of the formal linear ODE
 	 \begin{equation}\label{EqFormalLinODEReparam}
 	    \frac{\mathsf{d}W'_{\cdot\alpha'}}{\mathsf{d}s'}
 	      =\big(Y\circ\theta\big)
 	      \frac{\mathsf{d}\theta}{\mathsf{d}s'}
 	      W'_{\cdot\alpha'},
 	      ~~\mathrm{hence}~~
 	      \forall~\alpha,\beta\in [a,b]:~~
 	      W'_{\beta'\alpha'}  
 	      = W_{\theta(\beta')\theta(\alpha')}
 	      \in \mathcal{A}[[\lambda]].
 	 \end{equation}
 	 Moreover, the propagator $W_{\beta\alpha}$ only depends on the values of $Y$ between $\alpha$ and $\beta$.
 	 \item \textbf{Factorization}: Let $Y=Y_0+Z$
 	 with $Y_0,Z\in \Big(\mathcal{C}_D^\infty\big([a,b],\mathbb{C}\big)
 	 \otimes
 	 \mathcal{A}\Big)[[\lambda]]$, and let $W_{\cdot\alpha}$ and $U_{\cdot\alpha}$ be the fundamental solutions normalized at $\alpha$ for the formal linear ODE's
 	 \[
 	     \frac{\mathsf{d}W_{\cdot \alpha}}{\mathsf{d}s} =
 	     \lambda Y W_{\cdot \alpha} ~~
 	     \mathrm{and}~~
 	     \frac{\mathsf{d}U_{\cdot \alpha}}{\mathsf{d}s} =
 	     \lambda Y_0 U_{\cdot \alpha}.
 	 \] 
 	 Then $W_{\cdot \alpha}$ factorizes in the following
 	 way:
 	 \begin{equation}\label{EqFormalLinODEFactorization}
 	   W_{\cdot \alpha}=U_{\cdot \alpha}
 	   \Xi_{\cdot \alpha}~~~\mathrm{where}~~~
 	   \frac{\mathsf{d}
 	   	\Xi_{\cdot \alpha}}{\mathsf{d}s} =
 	   \lambda \left(U_{\cdot \alpha}^{-1}Z 
 	       U_{\cdot \alpha}\right)\Xi_{\cdot \alpha}
 	       ~~\mathrm{and}~~\Xi_{\alpha\alpha}=1.
 	 \end{equation}
 	 \item Suppose that $YI_\alpha(Y)=I_\alpha(Y)Y$.
 	    Then the fundamental solution $W_{\cdot\alpha}$ of the formal linear ODE (\ref{EqFormalLinODEDefinitionOfW}) is explicitly given by
 	    \begin{equation}
 	    \label{EqFormalLinODEYCommutesWithIOfY}
 	       W_{\cdot\alpha} =e^{\lambda I_\alpha(Y)} .
 	    \end{equation}
 \end{enumerate}
\end{proposition}

\noindent \textbf{Sketch of the proof}: $i.)$ Existence and
uniqueness follow from the considerations in (\ref{EqFormalLinODESolutionIntEq}), and (\ref{EqFormalLinODEGenInitCond}) can be read off (\ref{EqFormalLinODEIteratedIntegrals}). \\
$ii.)$ Again by (\ref{EqFormalLinODESolutionIntEq})
and (\ref{EqFormalLinODEIteratedIntegrals}) it is immediate
that $W_{\cdot \alpha}$ is a formal series in the associative unital algebra $\Big(\mathcal{C}_D^\infty\big([a,b],\mathbb{C}\big)^0
\otimes
\mathcal{A}\Big)[[\lambda]]$
whose constant term is $1$, hence it is obviously invertible by a similar geometric series argument.
The first equation of (\ref{EqFormalLinODEGroupoidProp}) is part of the definition (\ref{EqFormalLinODEDefinitionOfW}). For the second note that $W_{\cdot \beta}$
and $W_{\cdot \alpha}$ satisfy the same formal linear ODE with initial condition (at $\beta$) $1$ and $W_{\beta\alpha}$, respectively. 
 Hence, by (\ref{EqFormalLinODEGenInitCond}) we get
$W_{\cdot \alpha}=W_{\cdot\beta}W_{\beta\alpha}$
which gives the second equation of (\ref{EqFormalLinODEGroupoidProp}) upon choosing $s=\gamma$. The third equation of (\ref{EqFormalLinODEGroupoidProp}) follows from the first and the second upon setting $\alpha=\gamma$.\\
$iii.)$ Equation (\ref{EqFormalLinODEReparam})  is an easy consequence of the chain rule
and equation (\ref{EqFormalLinODEDefinition}). The last statement follows either directly from the iterated integral form (\ref{EqFormalLinODEIteratedIntegrals})
or by choosing --assuming that $\alpha\leqslant \beta$ without loss of generality-- $[a',b']=[\alpha,\beta]$
and $\theta:[\alpha,\beta]\to [a,b]$ the canonical injection.\\
$iv.)$ Using the well-known formula 
$\frac{\mathsf{d}(U_{\cdot\alpha}^{-1})}{\mathsf{d}s}
=-U_{\cdot\alpha}^{-1}
\frac{\mathsf{d}U_{\cdot\alpha}}{\mathsf{d}s}
U_{\cdot\alpha}^{-1}$ gives the result upon differentiating $\Xi_{\cdot\alpha}=U_{\cdot\alpha}^{-1}W_{\cdot\alpha}$.\\
$v.)$ For each positive integer $r$, differentiating the $r$th power of $I_\alpha(Y)^r$ we get $rYI_\alpha(Y)^{r-1}$ thanks to the hypothesis $YI_\alpha(Y)=I_\alpha(Y)Y$ which shows the result when differentiating the exponential series
$e^{\lambda I_\alpha(Y)}=
\sum_{r=0}^\infty\frac{\lambda^r}{r!}I_\alpha(Y)^r$.

\subsection{Norms and limits}
 \label{SubSec Norms and limits}

We shall have to discuss limits of solutions of formal
linear ODE's given by elements $Y$ of
the algebra (\ref{EqFormalLinODEPiecewiseTensorAofLambda}) depending on a parameter $\boldsymbol\delta$ in some
subset $\mathbf{J}\subset \mathbb{R}^\ell$, 
and we are interested in `limits' when $\boldsymbol\delta\to \boldsymbol\delta_0$
where $\boldsymbol\delta_0$ is an accumulation point of
$\mathbf{J}$.
Since the complex associative unital algebra $\mathcal{A}$ is completely arbitrary, we have to include a discussion to
make sense of these limits.

Recall that a \emph{norm} on a complex vector-space $E$ is
a map $||~||:E\to \mathbb{R}$ taking only non-negative values, satisfying $||\xi||=0$ iff $\xi=0$ for all $\xi\in E$, satisfying $||z\xi||=|z|~||\xi||$ for all
$z\in\mathbb{C}$ and $\xi\in E$, and satisfying the triangular inequality $||\xi+\eta||\leqslant ||\xi||+||\eta||$
for all $\xi,\eta\in E$. Every vector subspace $V\subset E$
is automatically a normed space with respect to the restriction of the norm to $V$. It is easy to see that every complex vector space has at least one norm: in fact, let $\mathsf{B}\coloneqq\left(e_i\right)_{i\in\mathfrak{S}}$ be a basis for $E$ labeled by the set $\mathfrak{S}$: every vector $\xi\in E$ is a linear combination $\xi=\sum_{i\in\mathfrak{S}}x_ie_i$ where all the $x_i\in\mathbb{C}$ and the subset of $\mathfrak{S}$ for which $x_i\neq 0$ is finite. Define
\begin{equation}\label{EqNormsLimDefSupNorm}
  ||\xi||_\mathsf{B}=||\xi||\coloneqq
    \max\big\{|x_i|~\big|~i\in\mathfrak{S}\big\},
\end{equation}
and the norm properties are easy to check directly.
$E$ can also be considered as a subspace of the Banach
space of all bounded functions $\mathsf{B}\to \mathbb{C}$
equipped with the $\sup$-norm, but this remark is not necessary for the elementary treatment presented here.\\
Having a norm allows us to define \emph{limits}: more precisely,
for a given positive integer $\ell$
let  $\mathbf{J}\subset \mathbb{R}^\ell$ be  a non-empty set, and let
$\mathrm{Fun}(\mathbf{J},E)$ denote the complex vector space of all
maps $\mathbf{J}\to E$. Fix a norm $||~||$ on $E$, some norm
$|~|$ on
$\mathbb{R}^\ell$, and a function
$f\in \mathrm{Fun}(\mathbf{J},E)$. Let $\boldsymbol\delta_0$ be an accumulation point of $\mathbf{J}$. For any $\zeta\in E$ recall the following definition of a limit:
\begin{eqnarray}
\lim_{\boldsymbol\delta\to \boldsymbol\delta_0} f(\boldsymbol\delta)=\zeta ~\mathrm{w.r.t.~}||~||
  & ~~\mathrm{iff}~~ &
\forall~\epsilon\in \mathbb{R},~\epsilon>0~\exists~
\epsilon'\in\mathbb{R},~\epsilon'>0:
~\forall~\boldsymbol\delta\in \mathbf{J}:~~\nonumber \\ 
& &\mathrm{if}~|\boldsymbol\delta-\boldsymbol\delta_0|<\epsilon'~~\mathrm{then}~~
||f(\boldsymbol\delta)-\zeta||<\epsilon. \label{EqNormsLimitDefLimit}
\end{eqnarray}
As usual, if the limit exists, it is unique. However,
the existence of the limit a priori depends on the norms
$|~|$ and $||~||$ used. Recall that two norms $||~||$ and
$||~||'$ on $E$ are called
\emph{equivalent}  if: 
\begin{equation*}
\exists~C_1,C_2\in\mathbb{R},~C_1,C_2>0~~
\forall~\xi\in E:~~~C_1||\xi||\leqslant ||\xi||'
\leqslant C_2||\xi||.
\end{equation*}
Hence, if the norms $||~||$ and $||~||'$ on $E$ are equivalent and
if the norms $|~|$ and $|~|'$ on $\mathbb{R}^\ell$ are equivalent,
it is easy to see that in (\ref{EqNormsLimitDefLimit})
the statement using $||~||$ and $|~|$ is equivalent
to the one using $||~||'$ and $|~|'$: in this case the
limit does not depend on the norms used.
In general, two given norms on a complex vector space are not equivalent; however,
in the very important case of a \emph{finite-dimensional} vector space it is well-known that any two norms are equivalent, see e.g.~\cite[p.145, Thm.4.3.]{Lan97}. This always applies to the norms $|~|$ and $|~|'$ on $\mathbb{R}^\ell$ in statement (\ref{EqNormsLimitDefLimit}), but in general not
to the norms $||~||$ and $||~||'$ on $E$. In this manuscript the relevant limits will always `take place' in
finite-dimensional subspaces of the algebra $\mathcal{A}=E$
thereby insuring that the computation of limits will not
depend on the norms chosen.
More precisely, consider the following algebra
\begin{equation}\label{EqNormsLimAlgebraFunOtimesAOfLambda}
   \Big(\mathrm{Fun}\big(\mathbf{J},\mathbb{C}\big)\otimes \mathcal{A}\Big)[[\lambda]].
\end{equation}
Each element $F$ of this algebra is a formal power series
$F=\sum_{r=0}^\infty \lambda^rF_r$ where each component
$F_r$ is an element of $\mathrm{Fun}\big(\mathbf{J},\mathbb{C}\big)\otimes \mathcal{A}$,
hence can be considered as a map $\mathbf{J}\to \mathcal{A}$, and choosing a norm $||~||$ on $\mathcal{A}$ and a norm 
$|~|$ on $\mathbb{R}^\ell\supset \mathbf{J}$ --among all the equivalent ones-- we can consider limits $F_r\to \xi_r\in
\mathcal{A}$ for each non-negative integer $r$ seperately
in the sense of definition 
(\ref{EqNormsLimitDefLimit}). For any $\xi=\sum_{r=0}^\infty\lambda^r\xi_r\in
\mathcal{A}[[\lambda]]$ we thus define limits componentwise
for each $F$ in the algebra
(\ref{EqNormsLimAlgebraFunOtimesAOfLambda}):
\begin{equation}\label{EqNormsLimDefLimFormPowerSer}
   \lim_{\boldsymbol\delta\to \boldsymbol\delta_0} F(\boldsymbol\delta)=\xi ~\mathrm{w.r.t.~}||~||
   ~~~ \mathrm{iff} ~~~
   \forall~r\in\mathbb{N}:~~
   \lim_{\boldsymbol\delta\to \boldsymbol\delta_0} F_r(\boldsymbol\delta)=\xi_r~\mathrm{w.r.t.~}||~||.
\end{equation}
We enumerate some important properties of limits in the
algebra (\ref{EqNormsLimAlgebraFunOtimesAOfLambda}) in the
following
\begin{proposition}\label{PNormsLimLimitRules}
  Let $\mathbf{J}\subset \mathbb{R}^\ell$ as above, let $\boldsymbol\delta_0\in\mathbb{R}^\ell$ be an accumulation point of $\mathbf{J}$, and let $F=\sum_{r=0}^\infty \lambda^rF_r$ be
  an element of $\Big(\mathrm{Fun}\big(\mathbf{J},\mathbb{C}\big)\otimes \mathcal{A}\Big)[[\lambda]]$.
  \begin{enumerate}
  	\item For each $r\in\mathbb{N}$ there is a finite-dimensional subspace $V^{(F)}_r=V_r$ 
  	of $\mathcal{A}$ such that for each $r\in\mathbb{N}$
  	\begin{equation}\label{EqNormsLimExFiniteDimSubspace}
  	  F_r\in \mathrm{Fun}\big(\mathbf{J},\mathbb{C}\big)\otimes
  	    V_r.
  	\end{equation}
  	\item Let $||~||$ and $||~||'$ be two norms on the complex vector space $\mathcal{A}$, 
  	and let $\xi=\sum_{r=0}^\infty\lambda^r\xi_r\in
  	\mathcal{A}[[\lambda]]$. Then the statement
  	\[
  	   \lim_{\boldsymbol\delta\to \boldsymbol\delta_0}F(\boldsymbol\delta)=\xi ~\mathrm{w.r.t.~}||~||
  	   ~~~\mathrm{is~equivalent~to~}~~~
  	   \lim_{\boldsymbol\delta\to \boldsymbol\delta_0}F(\boldsymbol\delta)=\xi ~\mathrm{w.r.t.~}||~||',
  	\]
  	hence limits in the algebra (\ref{EqNormsLimAlgebraFunOtimesAOfLambda}) do not depend on the norms used.
  	\item Let $\tilde{F}=\sum_{r=0}^\infty \lambda^r\tilde{F}_r$ be
  	another element of $\Big(\mathrm{Fun}\big(\mathbf{J},\mathbb{C}\big)\otimes \mathcal{A}\Big)[[\lambda]]$, and let 
  	$\tilde{\xi}\in\mathcal{A}[[\lambda]]$ such that
  	$\lim_{\boldsymbol\delta\to \boldsymbol\delta_0}\tilde{F}(\boldsymbol\delta)=\tilde{\xi}$ with respect to any norm on $\mathcal{A}$. Then for all $\alpha,\beta\in\mathbb{C}$:
  	\begin{equation}
  	\label{EqNormsLimLimitSumsAreSumsOfLimits}
  	\lim_{\boldsymbol\delta\to \boldsymbol\delta_0}
  	\Big(\alpha F(\boldsymbol\delta)+\beta \tilde{F}(\boldsymbol\delta)\Big)= \alpha\lim_{\boldsymbol\delta\to \boldsymbol\delta_0}F(\boldsymbol\delta)
  	+\beta \lim_{\boldsymbol\delta\to \boldsymbol\delta_0}\tilde{F}(\boldsymbol\delta)
  	=\alpha\xi+\beta\tilde{\xi}
  	\end{equation}
and
  	\begin{equation}
  	\label{EqNormsLimLimitProdAreProdOfLimits}
  	 \lim_{\boldsymbol\delta\to \boldsymbol\delta_0}\Big(F(\boldsymbol\delta)\tilde{F}(\boldsymbol\delta)\Big)= 
  	  \left(\lim_{\boldsymbol\delta\to \boldsymbol\delta_0}F(\boldsymbol\delta)\right)
  	   \left(\lim_{\boldsymbol\delta\to \boldsymbol\delta_0}\tilde{F}(\boldsymbol\delta)\right)
  	   =\xi\tilde{\xi}.
  	\end{equation}
  \end{enumerate}
\end{proposition}
\begin{proof}
$i.)$ By definition of the algebraic tensor product each $F_r$ is a finite sum $F_{r1}\otimes A_{r1}+
\cdots +F_{rN_r}\otimes A_{rN_r}$ where $N_r$ is a non-negative integer, $F_{r1},\ldots,F_{rN_r}$ are functions $J\to \mathbb{C}$, and $A_{r1},\ldots,A_{rN_r}$
are elements of $\mathcal{A}$. Defining $V_r$ as the complex
linear hull of $A_{r1},\ldots,A_{rN_r}$ proves the statement.\\
$ii.)$ We shall prove a slightly more general statement: for each non-negative integer $r$ let $V'_r$ be another finite-dimensional subspace of $\mathcal{A}$
such that (\ref{EqNormsLimExFiniteDimSubspace})
is satisfied. Then each $F_r$ clearly is an element
of $\mathrm{Fun}\big(J,\mathbb{C}\big)\otimes
\big(V_r\cap V'_r)$. We can enlarge each $V_r,V'_r$ by at most one dimension to include $\xi_r$. From the definition of
the limit (\ref{EqNormsLimitDefLimit}), it is clear that it suffices to look at the restrictions of the norms
$||~||$ and $||~||'$ on $\mathcal{A}$ to the finite-dimensional subspaces
$V_r$, $V'_r$ and $V_r\cap V'_r$: the restriction of
the norm $||~||$ to $V_r\cap V'_r$ is equivalent to the restriction of the norm $||~||'$ thanks to the finite
dimension of $V_r\cap V'_r$ which shows that the limit
statements
w.r.t.~the norms $||~||$ and $||~||'$ are equivalent.\\
$iii.)$ The first equation (\ref{EqNormsLimLimitSumsAreSumsOfLimits}) is the usual statement that in any normed vector space addition and
scalar multiplication are continuous. All the limits do
not depend on the norms (take for instance for each
$r\in\mathbb{N}$ the finite-dimensional vector space $V_r+\tilde{V}_r$) thanks to the preceding statement
$ii.)$. The second statement 
(\ref{EqNormsLimLimitProdAreProdOfLimits}) is slightly more involved since \textbf{the normed vector space
$\big(\mathcal{A},||~||\big)$ is in general 
NOT a normed algebra} in the sense that $||AA'||\leqslant ||A||~||A'||$
for all $A,A'\in\mathcal{A}$. We shall first prove an intermediate estimate: for each $r\in\mathbb{N}$ pick
a finite-dimensional subspace $\tilde{V}_r$ such that
$\tilde{F}_r\in \mathrm{Fun}(J,\mathbb{C})\otimes \tilde{V}_r$ (which is possible thanks to statement
$i.)$). We have for each $r\in\mathbb{N}$
\begin{eqnarray}
   \left(F(\boldsymbol\delta)\tilde{F}(\boldsymbol\delta)\right)_r
   =\sum_{u=0}^rF_u(\boldsymbol\delta)\tilde{F}_{r-u}(\boldsymbol\delta) \in
     \sum_{u=0}^r V_u\tilde{V}_{r-u}
     \subset
     \big(V_0+\cdot \cdot+V_r\big)
     \big(\tilde{V}_0+\cdot \cdot +\tilde{V}_r\big)
     =: V_{(r)}\tilde{V}_{(r)}.\nonumber \\
   \label{EqNormsLimFTildeFRInFiniteDim}   
\end{eqnarray}
Clearly, the subspaces $V_{(r)}$ and $\tilde{V}_{(r)}$ 
of $\mathcal{A}$ are finite-dimensional. Consider the
restriction of the algebra multiplication $\mu:
\mathcal{A}\otimes \mathcal{A}\to \mathcal{A}$ to
the finite-dimensional vector space $V_{(r)}\otimes\tilde{V}_{(r)}$: the image
$\mu\left(V_{(r)}\otimes\tilde{V}_{(r)}\right)=
V_{(r)}\tilde{V}_{(r)}$ is again a finite-dimensional
subspace of $\mathcal{A}$. Choosing a basis $e_1,\ldots,e_M$ of the finite-dimensional subspace
$V_{(r)}+\tilde{V}_{(r)}+V_{(r)}\tilde{V}_{(r)}$ of 
$\mathcal{A}$ which is compatible with the subspaces,
$V_{(r)}$, $\tilde{V}_{(r)}$, and $V_{(r)}\tilde{V}_{(r)}$,
expanding the elements $A\in V_{(r)}$,
$\tilde{A}\in\tilde{V}_{(r)}$ in that basis, and using
the norm $||~||$ as in (\ref{EqNormsLimDefSupNorm})
by extending the chosen basis to all of $\mathcal{A}$
we get the intermediate estimate
\begin{equation*}
 \exists~C_{V_{(r)}\tilde{V}_{(r)}}\in\mathbb{R},
  C_{V_{(r)}\tilde{V}_{(r)}}\geqslant 0~~
 \forall~A\in V_{(r)}~\forall~
 \tilde{A}\in \tilde{V}_{(r)}:~~~
 ||A\tilde{A}||\leqslant C_{V_{(r)}\tilde{V}_{(r)}}
           ||A||~||\tilde{A}||.
\end{equation*}
This shows that the restriction of the multiplication
to $V_{(r)}\otimes\tilde{V}_{(r)}$ is a continuous map
onto its image $V_{(r)}\tilde{V}_{(r)}$: this fact together with (\ref{EqNormsLimFTildeFRInFiniteDim})
proves the statement (\ref{EqNormsLimLimitProdAreProdOfLimits}).
\end{proof}

For the rest of this Section we choose the maximum norm $|~|$ on $\mathbb{R}^\ell$, see (\ref{EqNormsLimDefSupNorm}) w.r.t.~the canonical basis, and suppose that
\begin{equation}\label{EqNormsLimJWithZeroAccumPoint}
    \emptyset \neq \mathbf{J} \subset 
    \big\{\boldsymbol\delta\in \mathbb{R}^\ell\setminus \{0\}~\big|~|\boldsymbol\delta|\leqslant 1/4\big\}
    ~~~\mathrm{and}~~~0~
    \mathrm{is~an~accumulation~point~of~J}.
\end{equation}
We shall now distinguish three important subsets
$\mathcal{L}$, $\mathcal{B}$, and $\mathcal{H}$
of the algebra (\ref{EqNormsLimAlgebraFunOtimesAOfLambda}):
we shall refer to them as the set of all \textbf{at most logarithmically divergent}, \textbf{bounded} and
\textbf{harmless elements}, respectively:  for an element
$F=\sum_{r=0}^\infty F_r\lambda^r$ of
$\Big(\mathrm{Fun}\left(\mathbf{J},\mathbb{C}\right)
\otimes \mathcal{A}\Big)[[\lambda]]$ we say 
\begin{equation}\label{EqNormsLimDefRBH}
 \begin{array}{cccccccl}
    F \in \mathcal{L} & \mathrm{iff}
    & \forall r\in\mathbb{N}& \exists C_r,\alpha_r\in\mathbb{R},~C_r\geqslant 0,~\alpha_r>0 &\forall 
    \boldsymbol\delta\in \mathbf{J}: & ||F_r(\boldsymbol\delta)||&\leqslant& C_r|\ln(|\boldsymbol\delta|)|^{\alpha_r},\\
    F\in  \mathcal{B} &\mathrm{iff}
    & \forall r\in\mathbb{N}&
    \exists C_r\in\mathbb{R},~C_r\geqslant 0, 
    &\forall \boldsymbol\delta\in \mathbf{J}:&
      ||F_r(\boldsymbol\delta)||&\leqslant &C_r,\\
    F \in \mathcal{H} &\mathrm{iff}
    & \forall r\in\mathbb{N}&
    \exists C_r,\beta_r\in\mathbb{R},~C_r\geqslant 0,~\beta_r>0~& \forall 
     \boldsymbol\delta\in \mathbf{J}: &||F_r(\boldsymbol\delta)||&\leqslant& C_r |\boldsymbol\delta|^{\beta_r}.
 \end{array}
\end{equation}
In all the subsequent computations in this paper all the
terms which we shall deal with are at most logarithmically divergent in the above sense.
Let 
$\mathcal{G}
:=1+\lambda\Big(\mathrm{Fun}\left(J,\mathbb{C}\right)
\otimes \mathcal{A}\Big)[[\lambda]]$, and define the following subsets
\begin{equation}\label{EqFilNormsDefGroupsGRGBGH}
 \mathcal{G}_{\mathcal{L}}:=1+\lambda \mathcal{L},~~
 \mathcal{G}_{\mathcal{B}}:=1+\lambda \mathcal{B},~~
 \mathcal{G}_{\mathcal{H}}:=1+\lambda \mathcal{H}.
\end{equation}
and refer to them as the \emph{at most logarithmically divergent, bounded and harmless subgroups} of the group $\mathcal{G}$, respectively. These terms become clear in
the following
\begin{proposition}\label{PFilNormsRBHProperties}
 With the above hypotheses we have the following statements for the algebra
 (\ref{EqNormsLimAlgebraFunOtimesAOfLambda}):
 \begin{enumerate}
 	\item The definition (\ref{EqNormsLimDefRBH})
 	 does not depend on the chosen norm.
 	\item $\mathcal{L}$ and $\mathcal{B}$ are unital subalgebras over $\mathbb{C}[[\lambda]]$ of
 	$\Big(\mathrm{Fun}\left(\mathbf{J},\mathbb{C}\right)
 	\otimes \mathcal{A}\Big)[[\lambda]]$,
 	and $\mathcal{H}$ is a two-sided ideal of
 	$\mathcal{L}$. There are the following inclusions:
 	\begin{equation}\label{EqFilNormsRSupBSupB}
 	  \mathcal{L}\supset \mathcal{B}\supset
 	  \mathcal{H}.
 	\end{equation}
 	\item For all $H\in\mathcal{H}$:
 	 $\lim_{\boldsymbol\delta\to 0}H(\boldsymbol\delta)=0$.
 	\item $\mathcal{G}$ is a subgroup of the group of all invertible elements of 	$\Big(\mathrm{Fun}\left(\mathbf{J},\mathbb{C}\right)
 	\otimes \mathcal{A}\Big)[[\lambda]]$,
 	and $\mathcal{G}_\mathcal{L}\supset 
 	\mathcal{G}_\mathcal{B}\supset 
 	\mathcal{G}_\mathcal{H}$ are subgroups of
 	$\mathcal{G}$, where $\mathcal{G}_\mathcal{H}$
 	is a normal subgroup of $\mathcal{G}_\mathcal{L}$, 
 	i.e.~it is stable by conjugations with all elements
 	in $\mathcal{G}_\mathcal{L}$.
 	\item For all $\Psi\in \mathcal{G}_\mathcal{H}$:~~
 	  $\lim_{\boldsymbol\delta\to 0}\Psi(\boldsymbol\delta)=1$.
 \end{enumerate}
\end{proposition}
Before giving the proof of this Proposition we shall recall some elementary inequalities in the following
\begin{lemma} 
	For all $\delta\in~]0,1/4]$ and $\alpha,\beta \in \mathbb{R}$, $\alpha,\beta>0$, we have the following inequalities
	\begin{eqnarray}
	  \delta ~\leqslant ~\frac{1}{2} & \leqslant & 
	  |\ln(\delta)|,  \label{EqFilNormsDeltaLnDeltaIneq} \\
	  |\ln(\delta)|^\alpha \delta^\beta &\leqslant &
	  \left(\frac{2\alpha}{\beta}\right)^\alpha
	  \delta^{\beta/2}.
	  \label{EqFilNormsDeltaTimesLnDeltaIneq}
	\end{eqnarray}
\end{lemma}
\begin{proof} 
 Recall the following elementary inequalities for
 every real number $t$ such that $0<t\leqslant 1$
 \[
      1\leqslant \frac{1}{t}
      ~~~~\mathrm{hence}~~~~
      1-\delta =\int_{\delta}^{1}dt \leqslant 
        \int_{\delta}^{1}\frac{1}{t}dt =-\ln(\delta)
        =|\ln(\delta)|
 \]
 which proves (\ref{EqFilNormsDeltaLnDeltaIneq}) upon noting that $\delta\leqslant 1/4 < 1/2 <1-\delta$.
 Moreover, since
 $\frac{1}{t}\leqslant 
 \left(\frac{1}{t}
 \right)^{1+\frac{\beta}{2\alpha}}$
 we get
 \[
 |\ln(\delta)|=-\ln(\delta)=
 \int_{\delta}^{1}\frac{1}{t}dt \leqslant
 \int_{\delta}^{1}\left(\frac{1}{t}
 \right)^{1+\frac{\beta}{2\alpha}}dt
 =\frac{2\alpha}{\beta}
 \left(\delta^{-\beta/(2\alpha)}-1\right)
 \leqslant \frac{2\alpha}{\beta}\delta^{-\beta/(2\alpha)},
 \]
 from which --upon multiplying both sides of this inequality by $\delta^{\beta/\alpha}$ and then raising to
 the power of $\alpha$-- we deduce the result
 (\ref{EqFilNormsDeltaTimesLnDeltaIneq}). 
\end{proof}

\noindent \textit{Proof} (of Proposition \ref{PFilNormsRBHProperties}): \\
$i.)$ Since each $F_r$ takes its values in a finite dimensional vector space the restriction of any other norm to this subspace is equivalent to the norm $||~||$: this would only change the `$C$-constants' of the definition, but not the criterion to be an element of $\mathcal{L}$,
$\mathcal{B}$ or $\mathcal{H}$.\\
$ii.)$ It is immediate that
$\mathcal{L}$, $\mathcal{B}$ and $\mathcal{H}$ are
complex vector spaces: if $F,F'$ are in one of the three subsets then by means of the upper bounds of the
norms of each $F_r(j)$ and $F'_r(j)$ we get an upper bound of the norm 
of each $zF_r+z'F'_r$ (where $z,z'\in\mathbb{C}$) by passing to twice the maximum of
the two constants $|z|C_r,|z'|C_r$ and to the maximum of the
exponents $\alpha_r,\alpha'_r$ of $|\ln(|\boldsymbol\delta|)|$ --the latter being $>1$-- (resp. to the minimum of
the exponents $\beta_r,\beta'_r$ of $|\boldsymbol\delta|$ --the latter being $<1$).
Next, for the multiplication of $FF'$ we have that
each $\big(F(\boldsymbol\delta)F'(\boldsymbol\delta)\big)_r$ ($r\in\mathbb{N}$) is equal to the
sum $\sum_{u=0}^rF_u(\boldsymbol\delta)F'_{r-u}(\boldsymbol\delta)$. Suppose first that
$F,F'\in \mathcal{L}$.
Since by Proposition \ref{PNormsLimLimitRules} for each $\boldsymbol\delta\in \mathbf{J}$ every $F_u(\boldsymbol\delta)$ is an element of some finite-dimensional subspace $V_u$ (only depending on $F_u$) 
and every $F'_{r-u}(\boldsymbol\delta)$ is an element of some other finite-dimensional subspace $V'_{r-u}$ (only depending on $F_{r-u}$) it follows as in the proof of Proposition \ref{PNormsLimLimitRules}, equation
(\ref{EqNormsLimFTildeFRInFiniteDim}) that --upon setting
$V_{(r)}=V_0+\cdots + V_r$ and 
$V'_{(r)}=V_0+\cdots + V'_r$-- the following estimate
holds for all $\boldsymbol\delta\in \mathbf{J}$
\begin{eqnarray*}
 ||\big(F(\boldsymbol\delta)F'(\boldsymbol\delta)\big)_r|| &\leqslant &
 \sum_{u=0}^r C_uC'_{r-u}C_{V_{(r)}V'_{(r)}}
  |\ln(|\boldsymbol\delta|)|^{\alpha_u+\alpha'_{r-u}}
  ~\leqslant ~C |\ln(|\boldsymbol\delta|)|^\alpha
\end{eqnarray*}
where $C$ is $r+1$ times the maximum of all the triple products of the `$C$-constants' and $\alpha$ is the maximum of all the numbers $\alpha_u+\alpha'_{r-u}$.
This is done in an analogous way for $\mathcal{B}$ and 
$\mathcal{H}$ proving that $\mathcal{L}$,
$\mathcal{B}$ and 
$\mathcal{H}$ are closed under multiplication. Evidently, $\mathbb{C}[[\lambda]]$ belongs to $\mathcal{L}$ and
$\mathcal{B}$, hence $\mathcal{L}$ and
$\mathcal{B}$ 
in particular are $\mathbb{C}[[\lambda]]$-submodules
and unital associative algebras. The inclusion (\ref{EqFilNormsRSupBSupB}) follows at once from inequality (\ref{EqFilNormsDeltaLnDeltaIneq}).
Finally, for any $F\in\mathcal{L}$ and $F'\in\mathcal{H}$
it is shown in a similar way as above that $FF'$ and $F'F$ are in $\mathcal{H}$ upon using the second inequality (\ref{EqFilNormsDeltaTimesLnDeltaIneq}).
This shows that $\mathcal{H}$ is a two-sided ideal of $\mathcal{L}$ and hence also a 
$\mathbb{C}[[\lambda]]$-submodule.\\
$iii.)$ and $v.)$ immediately follow from the upper bounds defining $\mathcal{H}$.\\
$iv.)$ We only have to observe that every element
$1+\lambda F$ (where $F$ is in a 
$\mathbb{C}[[\lambda]]$-subalgebra) always has an inverse, namely the well-known geometric series
$\sum_{r=0}^\infty(-\lambda)^r F^r$, for which the terms of
positive order are all in the given subalgebra, which proves that $\mathcal{G}_\mathcal{L}$, 
$\mathcal{G}_\mathcal{B}$ and
$\mathcal{G}_\mathcal{H}$ are subgroups of 
$\mathcal{G}$. The normality of
$\mathcal{G}_\mathcal{H}$ follows from the fact that
$\mathcal{H}$ is a two-sided ideal of $\mathcal{L}$.
\hfill $\Box$\\

We shall now apply these limit considerations to
 the term $Y$ appearing in a formal linear
ODE, see (\ref{EqFormalLinODEDefinition}). $Y$
normally belongs to the algebra (\ref{EqFormalLinODEPiecewiseTensorAofLambda}).
In order to incorporate limits we shall make $Y$ dependent on the parameter 
$\boldsymbol\delta$ in the set $\mathbf{J}\subset \mathbb{R}^\ell$, see (\ref{EqNormsLimJWithZeroAccumPoint}), i.e.~we consider
\begin{equation}
  \label{EqFilNormsFunDeltaToPiecewiseTensorAofLambda}
 Y\in \bigg(\mathrm{Fun}\Big(\mathbf{J},
 \mathcal{C}_D^\infty\big([a,b],\mathbb{C}\big)\Big)
  \otimes
  \mathcal{A}\bigg)[[\lambda]].
\end{equation}
Hence, each element $Y$ is a formal power series $\sum_{r=0}^\infty Y_r\lambda^r$, where each $Y_r$ is a (non unique) finite
sum $Y_r=Y_{r1}\otimes A_{r1}+\cdots+Y_{rN_r}\otimes
A_{rN_r}$ where $N_r$ is an non-negative integer, $A_{r1},\ldots,A_{rN_r}\in
\mathcal{A}$ and $Y_{r1},\ldots,Y_{rN_r}$ are functions
on $\mathbf{J}$ with values in
$\mathcal{C}^\infty_D([a,b],\mathbb{C})$ (see (\ref{EqPiecewiseDefCKD})) where $a<b$ are two fixed real numbers, and $D$ is a finite set such that 
$\{a,b\}\subset D\subset [a,b]$. It makes sense to consider the formal
linear ODE (\ref{EqFormalLinODEDefinition}) for these
$\mathbf{J}$-dependent $Y$:

\begin{proposition}
Let $Y$ be an element of the algebra
(\ref{EqFilNormsFunDeltaToPiecewiseTensorAofLambda}).\\
For each $\alpha\in [a,b]$ there exists a unique element
\begin{equation}\label{EqNormsLimDefFundamentalSolDepJ}
   W_{\cdot\alpha}\in 
   \bigg(\mathrm{Fun}\Big(\mathbf{J},
   \mathcal{C}_D^{\infty}\big([a,b],\mathbb{C}\big)^0\Big)
   \otimes
   \mathcal{A}\bigg)[[\lambda]]
\end{equation}
satisfying the formal linear $\mathbf{J}$-dependent ODE
(\ref{EqFormalLinODEDefinitionOfW}) w.r.t.~the parameter $s\in [a,b]$ for each $\boldsymbol\delta\in \mathbf{J}$ such that $W_{\alpha\alpha}=1$. We shall call
$W_{\cdot\alpha}$ the fundamental solution of the formal linear $\mathbf{J}$-dependent ODE
(\ref{EqFormalLinODEDefinitionOfW}) normalized at $\alpha$.
Moreover, for each $\beta\in[a,b]$ the $\mathbf{J}$-dependent propagator $W_{\beta\alpha}\coloneqq W_{\cdot\alpha}(s=\beta)$ satisfies
\begin{equation}\label{EqNormsLimDefPropagatorDepJ}
     W_{\beta\alpha}\in \Big(\mathrm{Fun}\left(\mathbf{J},\mathbb{C}\right)
     \otimes \mathcal{A}\Big)[[\lambda]].
\end{equation} 
\end{proposition}
\begin{proof}
 This is done in complete analogy to the treatment
 in $\S$\ref{SubSecFormalLinearODEs} where we can
 literally follow (\ref{EqFormalLinODEDefinition}),
 the integral equation (\ref{EqFormalLinODEIntEq}) --the primitive $I_\alpha$ being extended to the algebra occurring in (\ref{EqFilNormsFunDeltaToPiecewiseTensorAofLambda}) by first composing it with the functions of $\mathbf{J}$ with values in
 $\mathcal{C}_D^{\infty}\big([a,b],\mathbb{C}\big)$ on the left tensor factor, then tensoring with the identity on $\mathcal{A}$, and finally extending componentwise--
 and the iterated integrals equation (\ref{EqFormalLinODEIteratedIntegrals}). This proves
 the existence of a unique fundamental solution $W_{\cdot\alpha}$ as in (\ref{EqNormsLimDefFundamentalSolDepJ}). Evaluating $s$ at $\beta$ gives (\ref{EqNormsLimDefPropagatorDepJ}).
\end{proof}
The following elementary Lemma will be the key criterion later on to prove that certain factors in a propagator are in the `harmless group' 
$\mathcal{G}_\mathcal{H}$.
First, as usual, having fixed a norm $||~||$ on $\mathcal{A}$ we shall denote by the same symbol
$||~||$ the map
\begin{eqnarray*}
 \mathrm{Fun}\Big(\mathbf{J},\mathcal{C}_D^{\infty}
 \big([a,b],
 \mathbb{C}\big)\Big)\otimes \mathcal{A}
 &\to&
 \mathrm{Fun}\Big(\mathbf{J},\mathcal{C}_D^{0}
 \big([a,b],
 \mathbb{R}\big)\Big): \nonumber \\
 \Big((\boldsymbol\delta,s)\mapsto G(\boldsymbol\delta,s)\Big)  
 &\mapsto& \Big((\boldsymbol\delta,s)\mapsto ||G(\boldsymbol\delta,s)||\Big) 
\end{eqnarray*}
for all $s\in [a,b]\setminus D$. Writing $G$ in
 $\mathrm{Fun}\Big(\mathbf{J},\mathcal{C}_D^{\infty}
 \big([a,b],
 \mathbb{C}\big)\Big)\otimes \mathcal{A}$
 in a basis $\left((e_i)_{i\in\mathfrak{S}}\right)$ of $\mathcal{A}$ as $G=G_0e_0+\cdots+G_Ne_N$ we get the well-known estimate --using the monotonicity of the Riemann integral-- for any
 $\alpha\leqslant s\leqslant \beta\in[a,b]$:
 \begin{equation}\label{EqFilNormsElementaryIntEstim}
 \begin{split}
  \left|\left|\int_{\alpha}^{s}
  \hat{G}(\boldsymbol\delta,s_1)\mathsf{d}s_1\right|\right|
 & \leqslant \int_{\alpha}^{s}||\hat{G}(\boldsymbol\delta,s_1)||\mathsf{d}s_1 \\
 & \leqslant \int_{\alpha}^{\beta}||\hat{G}(\boldsymbol\delta,s_1)||\mathsf{d}s_1 
  \leqslant
  (\beta-\alpha) \sup\left.\left\{||\hat{G}(\boldsymbol\delta,s_1)||
       ~\right|~s_1\in[0,1] \right\}
 \end{split}
 \end{equation}
  where --as usual-- we have written $\hat{G}$ for any extension of
  the function 
  $G$ from $\mathbf{J}\times \big([a,b]\setminus D\big)$
  to $\mathbf{J}\times [a,b]$.
  \begin{lemma}\label{LNormsLimIntegralBounds}
  	Let $Y=\sum_{r=0}\lambda^rY_r$ be an element of the algebra
  (\ref{EqFilNormsFunDeltaToPiecewiseTensorAofLambda}).
   Fix an arbitrary norm  $||~||$ on $\mathcal{A}$ and two arbitrary elements $\alpha,\beta\in[a,b]$. 
  	Consider the following three conditions on $Y$ referred to as \textbf{upper bounds for $Y$}:
  	there is an extension $\hat{Y}$ of $Y$ to $[a,b]$
  	such that
  	for each non-negative integer $r$:
  	\begin{equation}\label{LEqFilNormsDefRBH}
  	\begin{array}{cccccl}
  	(L): & \exists C_r,\alpha_r\in\mathbb{R}, C_r\geqslant 0,\alpha_r>0 &\forall 
  	\boldsymbol\delta\in \mathbf{J}: 
  	\sup\left.\left\{||\hat{Y}_r(\boldsymbol\delta,s)||
  	~\right|~s\in[a,b] \right\}
  	\leqslant& C_r|\ln(|\boldsymbol\delta|)|^{\alpha_r},\\
  	(B): & \exists C_r\in\mathbb{R}, C_r\geqslant 0,&\forall 
  	 \boldsymbol\delta\in \mathbf{J}: 
  	 \sup\left.\left\{||\hat{Y}_r(\boldsymbol\delta,s)||
  	 ~\right|~s\in[a,b] \right\}
  	   \leqslant& C_r,\\
  	(H): &  \exists C_r,\beta_r\in\mathbb{R}, C_r\geqslant 0,\beta_r>0 &\forall 
  	 \boldsymbol\delta\in \mathbf{J}: 
  	 \sup\left.\left\{||\hat{Y}_r(\boldsymbol\delta,s)||
  	 ~\right|~s\in[a,b] \right\}
  	  \leqslant& C_r|\boldsymbol\delta|^{\beta_r}.
  	 \end{array}
  	\end{equation}
  	Then the three conditions do not depend on the norms used.
  	Moreover, if condition $(L)$ (resp.~$(B)$ resp.~$(H)$) is
  	satisfied then the propagator $W_{\beta\alpha}$ 
  	(see (\ref{EqNormsLimDefPropagatorDepJ}))
  	for $Y$ belongs to the subgroup 
  	$\mathcal{G}_\mathcal{L}$ (resp.~$\mathcal{G}_\mathcal{B}$ 
  	resp.~$\mathcal{G}_\mathcal{H}$) of $\mathcal{G}$.
  \end{lemma}
  \begin{proof} The norm independence follows from the fact
  	that the norms will always be restricted to finite-dimensional subspaces of $\mathcal{A}$. Concerning the second statement, we first do the case 
  	$\alpha\leqslant\beta$:
  	writing out the propagator
  	$W_{\beta\alpha}$ in terms of iterated integrals as in equation (\ref{EqFormalLinODEIteratedIntegrals}) 
  	(for $s=\beta$) 
  	it can be seen by an easy induction using the estimate 
  	(\ref{EqFilNormsElementaryIntEstim})
  	that each iterated integral has as upper bound a product of integrals of the form 
  	$\boldsymbol\delta\mapsto \int_{\alpha}^{\beta}\big|\big|
  	 \hat{Y}_{i}(\boldsymbol\delta,s)\big|\big|\mathsf{d}s$ where each such integral has an upper bound by the last
  	 inequality of (\ref{EqFilNormsElementaryIntEstim}) and thus the desired upper bound according to the
  	 conditions $(L)$, $(B)$ or $(H)$. Passing to
  	 suitable maxima of products of constants of `type $C$', to suitable maxima of sums of exponents 
  	 of `type $\alpha_i$', and to suitable minima of
  	 exponents of `type $\beta$' we get the desired
  	 upper bounds for (\ref{EqNormsLimDefRBH}).
  	 The case $\alpha\geqslant \beta$ is done in a completely analogous manner by
  	 using the rule (\ref{EqPiecewiseIntegralSign}).
  \end{proof}

\subsection{Formal connections and parallel transports}
\label{SubSecFormalConnectionsAndParallelTransports}

Let $N\geqslant 1$ be an integer, and let $U\subset \mathbb{R}^N$ be a non-empty open subset.
\begin{definition}\label{DFormConnDefConnection}
	A \textbf{formal connection $\Gamma$} on $U$ is given by $N$ elements 
	\begin{equation*} 
	\Gamma_1 , \ldots , \Gamma_N \in \big(C^\infty(U, \mathbb{C}) \otimes \mathcal{A}\big)[[\lambda]]. \end{equation*}
\end{definition}
Here $C^\infty(U, \mathbb{C})$ denotes the unital associative commutative complex algebra of all functions
of $N$ variables $f: U \to \mathbb{C}$ which are smooth, i.e. in the class $C^\infty$: this means that all the higher order partial derivatives of $f$
exist and are continuous.
 We shall consider each $\Gamma_i$ as a function
$\Gamma_i : U \to \mathcal{A}[[\lambda]]$
by evaluating at $x \in U$. Note however that these functions are particular thanks to the algebraic tensor product $\otimes$. Every element $\Gamma_i$ in the
algebra $\big(C^\infty(U, \mathbb{C}) \otimes \mathcal{A}\big)[[\lambda]]$
is thus a formal power series $\Gamma_i=\sum_{r=0}^\infty \Gamma_{ir}\lambda^r$ such that each non-negative integer $r$
the component $\Gamma_{ir}$ is a \textbf{finite sum of terms} 
of the form $f\otimes A$ with
$f\in \mathcal{C}^\infty\big(U,\mathbb{C}\big)$ and
$A\in \mathcal{A}$.
 A very common notation borrowed from differential geometry is 
$\Gamma = \sum_{i=1}^N \Gamma_i \mathsf{d}x_i$,
which relates to differential forms (connection $1$--forms). We shall, however, use a sign convention for $\Gamma$ which is different from the one used in differential geometry to avoid additional signs.

Next, fix two real numbers $a<b$, fix a finite subset
$D\subset [a,b]$ such that $\{a,b\}\subset D$. We consider \textbf{continuous piecewise $\mathcal{C}^\infty$-paths
$c:[a,b]\to U$}, hence
\begin{equation}\label{EqFormConnDefPWSmoothPath}
   c\in \mathcal{C}^\infty_D\big([a,b],U\big)^0.
\end{equation}
This means that each real component $c_1,\ldots, c_N$
of $c$ is an element of $\mathcal{C}^\infty_D\big([a,b],\mathbb{R}\big)^0$ and that for each $s\in [a,b]$ the value $c(s)$ lies in $U\subset\mathbb{R}^N$. The most elementary paths
are \emph{line segments}, i.e.~given two points
$\xi,\eta\in \mathbb{R}^N$ we can consider the \textbf{affine path joining the initial point $\xi$ and the final point $\eta$} which is defined in the usual way by
\begin{equation}\label{EqFormConnDefAffinePath}
   c_{\eta\leftarrow\xi}=c:[0,1]\to \mathbb{R}^N:
     s\mapsto (1-s)\xi +s\eta,
\end{equation}
hence $c(0)=\xi$ and $c(1)=\eta$.
In case $\xi,\eta$ are elements of the open subset $U$ it has of course to be checked whether all the values of the affine
path also lie in $U$. If this is the case then it is clear that there is $\epsilon\in\mathbb{R}$, $\epsilon>0$, such that the right hand side
of (\ref{EqFormConnDefAffinePath}) makes sense as a $\mathcal{C}^\infty$-function from the
larger open interval
$]-\epsilon,1+\epsilon[$ to $U$.

Returning to general continuous piecewise smooth paths, we can associate to each such path $c$ defined in (\ref{EqFormConnDefPWSmoothPath}) the element $Y:=\Gamma^{(c)}$ for a formal linear ODE by
\begin{equation*}
  \Gamma^{(c)}\coloneqq 
    \sum_{i=1}^N \big(\Gamma_i\circ c\big)
                   \frac{\mathsf{d}c_i}{\mathsf{d}s}
                   ~~\in~
 \Big(\mathcal{C}^\infty_D\big([a,b],\mathbb{C}\big)
    \otimes
  \mathcal{A}\Big)[[\lambda]].
\end{equation*}
Fix $\alpha,\beta\in[a,b]$, then
we can consider the formal linear ODE
(\ref{EqFormalLinODEDefinitionOfW}) for
the choice $Y=\Gamma^{(c)}$ and its particular
solution ${}^\Gamma W^{(c)}_{\cdot\alpha}$ 
normalized at $\alpha$.
In differential geometry the propagator 
${}^\Gamma W^{(c)}_{\beta\alpha}=W_{\beta\alpha}
\in\mathcal{A}[[\lambda]]$, see (\ref{EqFormalLinODEDefOfPropagatorW}), is called the \textbf{parallel transport from $c(\alpha)$ to
	$c(\beta)$ along the path $c$}
(with respect to the connection $\Gamma$). Since parallel transports are propagators all the statements of Proposition \ref{PFormalLinODEPropagatorPropert}
are true for parallel transports. Note that for
a \textbf{constant path} $c_{\xi}(s)=\xi\in U$ for all
$s\in [a,b]$, the element $\Gamma^{(c_{\xi})}=0$,
and the parallel transport is reduced to the
unit element of $\mathcal{A}$.

Next, we shall need a very important tool for the computations to come, namely the \textbf{composition of two continuous piecewise smooth paths} $c_1:[0,1]\to U$ with singular set $D_1$
and $c_2:[0,1]\to U$ with singular set $D_2$ which are compatible in the groupoid sense $c_2(0)=c_1(1)$.
Recall the classical definition from algebraic topology (in the convention `from right-to-left') of the \emph{composed path}
$c_2*c_1:[0,1]\to U$ with singular set 
$D_{12}\coloneqq\left(\frac{1}{2}D_1\right)\cup 
\left(\frac{1}{2}D_2+\frac{1}{2}\right)$:
\begin{equation}\label{EqFormConnDefCompPaths}
  \mathrm{if}~c_1(1)=c_2(0):~~~~ (c_2*c_1)(s)\coloneqq\left\{\begin{array}{ccl}
       c_1(2s) & \mathrm{if} & 0\leqslant s\leqslant \frac{1}{2},\\
       c_2(2s-1) & \mathrm{if} & \frac{1}{2}\leqslant s\leqslant 1.
   \end{array}\right.
\end{equation}
It is evident from the definition that $c_2*c_1$ is continuous (thanks to the condition $c_2(0)=c_1(1)$
and piecewise smooth with singular set $D_{12}$. \textbf{The crucial fact  is that composition may create new singularities at the point $s=\frac{1}{2}$} for the higher $k$-fold derivatives of the path for $k\geqslant 1$.

We now mention another important tool, the
\textbf{pull-back of formal connections}: Let
$N'$ be a positive integer, let $U'$ be a non-empty open subset of $\mathbb{R}^{N'}$, and let
$\Theta:U'\to U$ be a $\mathcal{C}^\infty$-map. For any
formal connection $\Gamma$ on $U$ define the 
\emph{pulled-back
connection} $\Gamma':=\Theta^*\Gamma$ on $U'$ defined by
\begin{equation}\label{EqFormConnDefPullBackConn}
  \forall~j\in\mathbb{N},~1\leqslant j\leqslant N':~~\left(\Theta^*\Gamma\right)_j\coloneqq
    \sum_{i=1}^N
    \big(\Gamma_i\circ\Theta\big)~
    \frac{\partial \Theta_i}{\partial x'_j}.
\end{equation}

We now state how the above operations on formal connections and paths translate to parallel transports:
\begin{theorem}\label{TFormConnOperationsForParTransp}
 Let $U\subset \mathbb{R}^N$ be a non-empty open subset and
 let $\Gamma$ be a formal connection defined on $U$. Let
 $c:[a,b]\to U$ and $c_1,c_2:[0,1] \to U$ be continuous piecewise smooth paths. Then we have the following:
 \begin{enumerate}
 	\item Let $\theta$ be a continuous piecewise 
 	$\mathcal{C}^\infty$-reparametrization of $\big([a,b],D\big)$
 	(i.e.~$a'<b'$ are real numbers, $\{a',b'\}\subset D'\subset [a',b']$ is a finite subset, and $\theta\in \mathcal{C}_{D'}^\infty\big([a',b'],\mathbb{R}\big)^0$
 	satisfying (\ref{EqPiecewiseCompositionCompatible})). Then
 	\begin{equation}\label{EqFormConnWReparamInv}
 	\Gamma^{(c\circ \theta)}
 	=\left(\Gamma^{(c)}\circ\theta\right)
 	\frac{\mathsf{d}\theta}{\mathsf{d}s'},~~
 	\mathrm{and}~~~
 	{}^\Gamma W^{(c)}_{\theta(\beta')\theta(\alpha')}
 	={}^\Gamma W^{(c\circ\theta)}_{\beta'\alpha'}
 	~~\in~~\mathcal{A}[[\lambda]].
 	\end{equation}
 	This shows that
 	reparametrizations (in the sense of
 	(\ref{EqPiecewiseCompositionCompatible})) of these paths do not change
 	parallel transport as long as the initial and final points remain the same.
 	\item In the previous statement suppose that $a'=a$, $b'=b$, $D'=D$
 	such that the continuous piecewise smooth reparametrization $\theta$ of $\big([a,b],D\big)$ satisfying in addition the inversion
 	condition
 	\begin{equation*}
 	\theta(\alpha)=\beta~~~\mathrm{and}~~~
 	\theta(\beta)=\alpha.
 	\end{equation*} 
 	Then  we get the well-known \textbf{inversion formula}
 	\begin{equation}\label{EqFormConnInversionFormula}
 	{}^\Gamma W^{(c\circ\theta)}_{\beta\alpha}
 	= \left({}^\Gamma W^{(c)}_{\beta\alpha}\right)^{-1}
 	\end{equation}
 	 \item The parallel transport along the composed path
 	 $c_2*c_1:[0,1]\to U$, see (\ref{EqFormConnDefCompPaths}) is given as follows
 	 \begin{equation}\label{EqFormConnComposOfPathsParTransp}
 	 {}^\Gamma W^{(c_2 * c_1)}_{10} 
 	 ~ = ~{}^\Gamma W^{(c_2)}_{10}~
 	 {}^\Gamma W^{(c_1)}_{10}.
 	 \end{equation}
 	 \item Let $U'\subset \mathbb{R}^{N'}$ be a non-empty open subset, let $\Theta:U'\to U$ be a smooth, let
 	 $\Gamma'=\Theta^*\Gamma$ be the pulled-back formal connection, and let $c':[a,b]\to U'$ be a continuous
 	 piecewise smooth path. Then
 	  for all $\alpha,\beta\in[a,b]$
 	 \begin{equation}\label{EqFormConnPullBackParTransp}
 	 \left(\Theta^*\Gamma\right)^{(c')} =
 	 \Gamma^{(\Theta\circ c')}~~\mathrm{and}~~
 	 {}^{\Theta^*\Gamma}W^{(c')}_{\beta\alpha}
 	 ={}^{\Gamma}W^{(\Theta\circ c')}_{\beta\alpha}.
 	 \end{equation}
 \end{enumerate}
\end{theorem}
\begin{proof}
$i.)$ The formula for $\Gamma^{(c\circ\theta)}$ is straight-forward and the equation for the parallel transport is deduced from (\ref{EqFormalLinODEReparam}).\\	
$ii.)$ This is an immediate consequence of (\ref{EqFormConnWReparamInv}) and the last equation of (\ref{EqFormalLinODEGroupoidProp}).\\
$iii.)$ Define the two smooth reparametrizations $\theta_1,\theta_2:
[0,1]\to [0,1]$ given by $\theta_1(s)=\frac{1}{2}s$
and $\theta_2(s)=\frac{1}{2}s+\frac{1}{2}$. We have
(suppressing the symbol $\Gamma$)
\[
W^{(c_2 * c_1)}_{10} 
~ \stackrel{(\ref{EqFormalLinODEGroupoidProp})}{=} 
~ 
W^{(c_2* c_1)}_{1\frac{1}{2}}~
W^{(c_2*c_1)}_{\frac{1}{2}0}
~ \stackrel{(\ref{EqFormConnWReparamInv})}{=} 
~ W^{((c_2*c_1)\circ \theta_2)}_{10}~
W^{((c_2*c_1)\circ \theta_1)}_{10} 
~ = ~ W^{(c_2)}_{10}~
W^{(c_1)}_{10}.
\]
$iv.)$ This is straight-forward from the definitions
and the chain rule for partial derivatives.
\end{proof}
\noindent For instance,
 the usual \textbf{affine inversion $j$ of the interval} $[a,b]$
 given by $\iota(s)= a+b-s$ for all $s\in[a,b]$ serves as
 such a reparametrization for the choice $\alpha=a$ and
 $\beta=b$ for an inversion in $ii.)$.\\
 It is well-known that composition of paths is in general NOT associative, i.e.~if $c_3:[0,1]\to U$ is a third path with $c_2(1)=c_3(0)$ then in general
 $c_3*(c_2*c_1)\neq (c_3*c_2)*c_1$, but the corresponding
 product of parallel transports does not depend on the bracketing, i.e.
 \begin{equation*}
 W^{(c_3*(c_2*c_1))}_{10}
 \stackrel{(\ref{EqFormConnComposOfPathsParTransp})}
 {=}
 W^{(c_3)}_{10}~W^{(c_2)}_{10}~W^{(c_1)}_{10}
 \stackrel{(\ref{EqFormConnComposOfPathsParTransp})}
 {=}  
 W^{((c_3*c_2)*c_1)}_{10}.
 \end{equation*}

It turns out that certain connections are formulated
by \emph{complex coordinates} which allow for much more compact
computations: we do not have to go into the detail of general
holomorphic connections, since for this work it suffices to
study complex rational ones. More precisely,
let $U\subset \mathbb{C}^N$ be a non-empty open set. Recall 
that a \emph{complex rational function} in $N$ complex variables
$z=(z_1,\ldots,z_N)$ defined on $U$ is a quotient 
$f(z)=\frac{g(z)}{h(z)}$ where $f,g\in\mathbb{C}[z_1,\ldots,z_N]$, hence are complex polynomials in $N$ variables such that $g$ is different
from the zero polynomial, and the zeros of $g$ all belong
to $\mathbb{C}^N\setminus U$. Hence, the function $z\mapsto \frac{g(z)}{h(z)}$ is a well-defined function on $U$. 
Decomposing each complex variable in real and imaginary part as usual,
\begin{equation*}
z_1=x_1+\mathbf{i}y_1,\ldots,z_N=x_N+\mathbf{i}y_N
~~~\mathrm{or}~~~z=x+\mathbf{i}y,
\end{equation*}
it is clear that each complex rational function $f$ is a
particular complex-valued rational function in $2N$ real variables $x_1,\ldots,x_N,y_1,\ldots,y_N=:(x,y)$ which we can write in the following way
\begin{equation}
  \label{EqFormConnComplexFunctionRealImPart}
f(z)= f(x+\mathbf{i}y) =: \check{f}(x,y) 
=: f^{(1)}(x,y) + 
\mathbf{i}f^{(2)}(x,y)
\end{equation} 
with unique real rational functions $f^{(1)},f^{(2)}$. Hence, each
complex rational function is a $\mathcal{C}^\infty$-function in the real variables. Let $\mathbb{C}_U(z)$ denote the set of all complex rational functions which are well-defined on $U$. It is easy to check that they form a complex unital subalgebra of
$\mathcal{C}^\infty(U,\mathbb{C})$. Recall the following well-known rules for the complex derivatives
for all integers $1\leqslant j \leqslant N$:
\begin{equation}\label{EqFormFlatHolomorphicCondandDer}
\left(\dfrac{\partial f}{\partial z_j}\right)^\vee
=\frac{1}{2}\left(
\dfrac{\partial \check{f}}{\partial x_j}
-\mathbf{i}
\dfrac{\partial \check{f}}{\partial y_j}\right)
~~\mathrm{and}~~
\dfrac{\partial \check{f}}{\partial x_j}
+\mathbf{i}
\dfrac{\partial\check{f} }{\partial y_j}=0,
~~\mathrm{hence}~~
\dfrac{\partial \check{f}}{\partial x_j}
 =\left(\dfrac{\partial f}{\partial z_j}\right)^\vee
   =-\mathbf{i}\dfrac{\partial \check{f}}{\partial y_j}
\end{equation}
for all complex rational functions $f$ where the first two equations are easy to check on polynomials and the second is nothing but the well-known `holomorphicity condition' $\partial f/\partial \bar{z}_j=0$ for the complex conjugate variables.
Next, let
$\Gamma$ be a formal connection on $U$ which is complex rational in the following way:
\begin{equation}\label{EqFormFlatDefComplexRationalConn}
\Gamma(z)=\sum_{j=1}^N\Gamma_j(z)\mathsf{d}z_j
~~~~\mathrm{and}~~~~
\forall~1\leqslant j\leqslant N:~~
\Gamma_j\in \Big(\mathbb{C}_U(z)\otimes 
\mathcal{A}\Big)[[\lambda]].
\end{equation}
Note that the linear combination is over the complex $\mathsf{d}z_j$! We can rewrite this expression in the $2N$ real $x$ and $y$ coordinates:
\begin{eqnarray}
\Gamma(z) &=& \sum_{j=1}^N\Gamma_j(z)\mathsf{d}z_j
=\sum_{j=1}^N\Gamma_j(x+\mathbf{i}y)(\mathsf{d}x_j
+\mathbf{i}\mathsf{d}y_j)
=\sum_{j=1}^N
\check{\Gamma}_j(x,y)\mathsf{d}x_j
+\sum_{j=1}^N
\mathbf{i}\check{\Gamma}_j(x,y)\mathsf{d}y_j\nonumber \\
&=:&\sum_{j=1}^N
\check{\Gamma}^{[1]}_j(x,y)\mathsf{d}x_j
+\sum_{j=1}^N
\check{\Gamma}^{[2]}_j(x,y)\mathsf{d}y_j
=:\check{\Gamma}(x,y). 
\label{EqFormFlatRealFormCompRatConn}
\end{eqnarray}
and get an ordinary formal connection with
components $\check{\Gamma}^{[1]}_j=\check{\Gamma}_j$ in the $x_j$-directions,
and $\check{\Gamma}^{[2]}_j=\mathbf{i}\check{\Gamma}_j$ in the $y_j$-directions.\\
We have the following completely unsurprising, but useful
result for complex pull-backs: let $N'$ be a positive integer, let $U'\subset \mathbb{C}^{N'}$ be a non-empty open set, let $z'=x'+\mathbf{i}y'=(z'_1,\ldots,z'_{N'})$ be complex coordinates, and let
$\Theta_1,\ldots \Theta_{N}:U'\to \mathbb{C}$ be complex rational functions such that the map 
$\Theta=(\Theta_1,\ldots,\Theta_N):U'\to \mathbb{C}^N$
takes its values in $U$. We shall write
\begin{equation}\label{EqFormConnComplexRealPhi}
   \Theta(z')=\Theta(x'+\mathbf{i}y')=\check{\Theta}(x',y')
   =\Theta^{(1)}(x',y')+\mathbf{i}\Theta^{(2)}(x',y')
\end{equation}
where $\Theta^{(1)}, \Theta^{(2)}:U'\to\mathbb{R}^N$ are real
rational functions. See the Appendix for the proof of the
following
\begin{proposition}
	\label{PFormConnComplexPullBackEqualReal}
	Let $\Gamma$ be a complex rational connection on 
	$U\subset \mathbb{C}^N$, and define the complex pullback $\Theta^*\Gamma$ by formula (\ref{EqFormConnDefPullBackConn}) with $x'_j$ replaced
	by $z'_j$. Then
	\begin{equation}\label{EqFormConnComplexPullBackConn}
	    \left(\Theta^*\Gamma\right)^\vee = 
	    \check{\Theta}^*\check{\Gamma}.
	\end{equation}
\end{proposition}

\subsection{Flat formal connections}
\label{SubSecFlatFormalConnections}
 
Let $N$ be a positive integer,  $U\subset\mathbb{R}^N$ a non-empty open subset,
and $\Gamma$ a formal connection on $U$.
$\Gamma$ is called \textbf{flat} (more precisely \emph{formally flat}) if the following
conditions hold:
\begin{equation}
\label{EqFlatDefFlatConnection}
\forall~i,j\in\mathbb{N},~1\leqslant i,j\leqslant N:~~~
0 = \dfrac{\partial \Gamma_i}{\partial x_j} - \dfrac{\partial \Gamma_j}{\partial x_i} + \lambda \Big(\Gamma_i \Gamma_j - \Gamma_j \Gamma_i\Big).
\end{equation} 
Obviously, \textbf{any formal connection on an open set of $\mathbb{R}^1$ is flat}.
Moreover, complex rational flatness is equivalent to
flatness in the following sense:
\begin{proposition}\label{PFormFlatComplexFlatEqRealFlat}
	Let $U\subset \mathbb{C}^N$ be an open set and let
	$\Gamma$ be a formal connection which is complex rational in the sense of (\ref{EqFormFlatDefComplexRationalConn}). Let
	$\check{\Gamma}$ be the formal connection in the sense of
	(\ref{EqFormFlatRealFormCompRatConn}).
	Then $\Gamma$ is flat in the complex sense, i.e.
	\begin{equation}\label{EqFormFlatDefComplexFlatness}
	\forall~i,j\in\mathbb{N},~1\leqslant i,j\leqslant N:~~~
	0 = \dfrac{\partial \Gamma_i}{\partial z_j} - \dfrac{\partial \Gamma_j}{\partial z_i} + \lambda \Big(\Gamma_i \Gamma_j - \Gamma_j \Gamma_i\Big)
     \end{equation}
	if and only if $\check{\Gamma}$ is flat in the normal
	`real sense', see (\ref{EqFlatDefFlatConnection})
	for $2N$ real variables $(x,y)$.
\end{proposition}
\begin{proof}
	We denote the right-hand side of
	(\ref{EqFormFlatDefComplexFlatness}) by $R_{ij}$.
	Equation (\ref{EqFormFlatHolomorphicCondandDer}) allows to replace complex derivatives $\partial/\partial z_i$ by the real ones, and thanks to (\ref{EqFormFlatRealFormCompRatConn})
	$\check{R}_{ij}$ equals
	\begin{equation*}
	\begin{split}
	  \dfrac{\partial \check{\Gamma}^{[1]}_i}{\partial x_j}
	  -\dfrac{\partial \check{\Gamma}^{[1]}_j}
	      {\partial x_i}
	  +\lambda \big[\check{\Gamma}^{[1]}_i,
	     \check{\Gamma}^{[1]}_j\big]
	   &= 
	   -\left(
	   \dfrac{\partial \check{\Gamma}^{[2]}_i}
	                  {\partial y_j}
	   -\dfrac{\partial \check{\Gamma}^{[2]}_j}
	       {\partial y_i}
	   +\lambda \big[\check{\Gamma}^{[2]}_i,
	       \check{\Gamma}^{[2]}_j\big]\right) \\
	   &= 
	   -\mathbf{i}\left(
	   \dfrac{\partial \check{\Gamma}^{[2]}_i}
	   {\partial x_j}
	   -\dfrac{\partial \check{\Gamma}^{[1]}_j}
	   {\partial y_i}
	   +\lambda \big[\check{\Gamma}^{[2]}_i,
	   \check{\Gamma}^{[1]}_j\big]\right)    
	   \end{split}
	\end{equation*}
	which exactly gives the components of the right hand
	side of 
	(\ref{EqFlatDefFlatConnection}) for $\check{\Gamma}$
	whence the result.
\end{proof}
\noindent In particular, \textbf{any complex rational formal connection on an open subset of $\mathbb{C}^1$ is flat}. Moreover, we mention the following well-known result (see the Appendix for a proof):
\begin{proposition}\label{PFlatConnPullBackOfFlatIsFlat}
	Let $N,N'$ be positive integers, let $U\subset \mathbb{R}^N$ and $U'\subset \mathbb{R}^{N'}$ be  non-empty open subsets, let
	$\Theta:U'\to U$ be a $\mathcal{C}^\infty$-map, and let
	$\Gamma$ be a flat formal connection on $U$.\\
	Then the pulled-back formal connection $\Gamma'\coloneqq\Phi^*\Gamma$, see (\ref{EqFormConnDefPullBackConn}), is also flat.
\end{proposition}
The significance of flat connections is the following well-known result about the \textbf{path-indepen\-dence of parallel transports}, for which we give a proof in
the Appendix:
\begin{theorem}
	\label{TFormFlatPathIndep}
	Let $N\geqslant 1$, $U \subset \mathbb{R}^N$ be a non-empty open subset and $\Gamma$ be a flat formal connection. Let $p,q \in U$, $\epsilon\in\mathbb{R}$, $\epsilon>0$, and $c_0,c_1 : ]a-\epsilon,b+\epsilon[ ~\to U$ be two smooth paths (NOT only piecewise smooth!) such that
	\begin{itemize}
		\item[(i)] $c_0(a)= p = c_1(a)$, 
		  $c_0(b) = q = c_1(b)$;
		\item[(ii)] there exists a \textbf{smooth homotopy} $F$ between $c_0$ and $c_1$: more precisely, there is an open subset $\mathcal{O}\subset \mathbb{R}^2$ with $]a-\epsilon,b+\epsilon[ ~\times~ 
		]-\epsilon,1+\epsilon[~\subset~ \mathcal{O}$, and $F:\mathcal{O} \to U$ is a $\mathcal{C}^\infty$-map satisfying
		\begin{equation}
		\label{EqFormFlatBoundaryCondHomotF}
		\begin{array}{ccccc}
		\forall~s \in ~]a-\epsilon,b+\epsilon[  & : & F(s,0) = c_0(s) & \mathrm{and} & F(s,1) = c_1(s), \\
		\forall~t \in ~]-\epsilon,1+\epsilon[ & : & F(a,t) = p & \mathrm{and} & F(b,t) = q.
		\end{array}
		\end{equation}
	\end{itemize}
	Then the parallel transport with respect to $\Gamma$ from $p$ to $q$ along $c_0$ is equal to the parallel transport along $c_1$ with respect to $\Gamma$ from $p$ to $q$:
	\begin{equation*}
	{}^\Gamma W_{b a }^{(c_0)} = {}^\Gamma W_{b a }^{(c_1)}.
	\end{equation*}
\end{theorem}

We shall give a corollary to the preceding Theorem which
will cover all the cases we shall discuss later: we need to establish first a relation between continuous piecewise smooth paths (which will turn up while doing composition
of paths), and overall smooth paths: 
specializing to $[a,b]=[0,1]$ (which will be only parameter interval in the sequel) we can prove the following
 very useful Corollary, see the Appendix for a proof:
\begin{corollary}\label{CFormFlatContractibleLoops}
	Let $N\in\mathbb{N}\setminus \{0\}$, $U\subset \mathbb{R}^N$ be a non-empty open set, $\Gamma$ a flat formal connection on $U$, and $c_1,c_2:[0,1]\to U$ two \emph{continuous piecewise smooth paths having
	 the same initial point $p$ and final point $q$}.
	Suppose that there is an open set $U'\subset \mathbb{R}^N$ and a point $\varpi\in U'$ such that
	$c_1$ and $c_2$ take all their values in 
	$U'\subset U$ and which is \textbf{star-shaped around $\varpi\in U'$},
	i.e.
	\begin{equation*}
	 \forall~t\in[0,1]~\forall~x\in U':~~~
	    (1-t)x+t\varpi\in U'.
	\end{equation*}
	Then the parallel transports $p\to q$ along $c_1$ and
	along $c_2$ are equal.
	In particular, the parallel transport along any
	continuous piecewise smooth loop $c_3:[0,1]\to U'\subset U$ (recall $c_3(0)=c_3(1)$) is trivial,
	i.e.~equal to $1\in\mathcal{A}$.
\end{corollary}
\noindent

 
\section{The Drinfel'd associator and its identities}
\label{SecDrinfeldAssociator}

\subsection{The Knizhnik-Zamolodchikov connection and the Drinfel'd-Kohno (Lie) algebras }
	\label{SubSecDKAlgebrasKnZhConnections}

We shall only need the material of this Section for the treatment of the pentagon equations, see $\S$\ref{SubSecPentagonEquation}. The geometric considerations,
however, may serve as motivations for the constructions
in all the ensuing sections. The following definition appears in the work of T.Kohno, 
\cite[p.142, eqn (1.1.4)]{Koh87} and
\cite[p.146, eqn (1.3.2)]{Koh87}:

\begin{definition}
	\label{DK-algebras}
	Let $n \geqslant 2$ be an integer and let $\mathbb{K}$ be a field of characteristic zero. The 
	$\mathbf{n}$\textbf{th} \textbf{Drinfel'd--Kohno algebra} is the unital associative $\mathbb{K}$--algebra $\mathcal{T}_n$ generated by $n^2-n$ elements 
	\[
	t_{ij}, \qquad 1 \leqslant i\neq j \leqslant n
	\]
	subject to the relations, sometimes called 
	\textbf{infinitesimal braid relations}:
	\begin{subequations}
		\begin{align}
		t_{ij}-t_{ji}&= 0 \qquad \forall ~1\leqslant i\neq j\leqslant n,
		\label{eq:DK-definition-one}\\
		[t_{ij}+t_{ik}, t_{jk}] &= 0 \qquad \forall i,j,k \in \{1, \ldots, n \} \ \text{ such that } \# \{i,j,k \} = 3, \label{eq:DK-definition-two} \\
		[t_{ij},t_{kl}] &= 0 \qquad \forall i,j,k,l \in \{1, \ldots, n \} \ \text{ such that } \# \{i,j,k,l \} = 4.
		 \label{eq:DK-definition-three}
		\end{align}
	\end{subequations}
where $[~,~]$ denotes the commutator $[A,B]=AB-BA$ in 
associative algebras.
\end{definition}
\noindent In view of the first relation (\ref{eq:DK-definition-one}) one could have defined the Drinfel'd-Kohno algebra just by $\binom{n}{2}$ generators $t_{ij}$ with $i < j$. However, for concrete computations
it is much more practical to use `unordered' generators $t_{ij}=t_{ji}$ (one may think of the index being the two-element set $\{i,j\}$). The Lie counterpart of the definition above is better known: 
\begin{definition}
	Let $n \geqslant 2$ be an integer and let $\mathbb{K}$ be a field of characteristic zero. The
	$\mathbf{n}$\textbf{th} \textbf{Drinfel'd--Kohno Lie algebra} is the $\mathbb{K}$--Lie algebra $\mathfrak{t}_n$ generated by $n^2-n$ generators
	$t_{ij}$, $1 \leqslant i \neq  j \leqslant n$
	subject to the relations 
	(\ref{eq:DK-definition-one}),
	 (\ref{eq:DK-definition-two}), and 
	 (\ref{eq:DK-definition-three}), seen as Lie brackets in the free Lie algebra generated by the $t_{ij}$.
\end{definition}

\noindent It is not hard to see (but not necessary for
the proofs of the ensuing subsections), that the 
Drinfel'd-Kohno (Lie) algebras are \emph{free} (Lie) algebras with relations in the 
usual sense of a left adjoint functor, that is
that for any unital associative algebra $\mathcal{A}$ 
(resp.~Lie algebra $\mathfrak{g}$) the set of morphisms of algebras $\mathcal{T}_n\to \mathcal{A}$ 
(resp.~of Lie algebras $\mathfrak{t}_n\to \mathfrak{g}$) is 
in natural bijection with the set of elements
$A_{ij}=A_{ji}$ ($1\leqslant i<j\leqslant n$) of $A$ (resp.~of $\mathfrak{g}$) satisfying
relations (\ref{eq:DK-definition-one}),
(\ref{eq:DK-definition-two}), and (\ref{eq:DK-definition-three}) which is defined by
evaluation of a morphism on the generators $(t_{ij})$. \\
In the sequel we shall prefer the more general situation
of a general unital associative complex 
algebra $\mathcal{A}$ with elements
$A_{ij}=A_{ji}$ satisfying the above-mentioned relations.
It is not hard to check that each $\mathcal{T}_n$ is isomorphic to the universal enveloping algebra
of the Lie algebra $\mathfrak{t}_n$, but we shall not
need this result.

For example, denoting by $\mathbb{K}\langle a_1,\ldots,a_n\rangle$ (resp.~$\mathbb{K}[ a_1,\ldots,a_n]$) the free unital associative algebra (resp.~the free unital associative commutative algebra) generated by the symbols $a_1,\ldots,a_n$ there are the following isomorphisms
\begin{equation*}
    \mathcal{T}_2~~\cong~~ \mathbb{K}\langle t_{12}\rangle~~\cong~~ \mathbb{K}[t_{12}]~~~~
    \mathrm{and}~~~~\mathcal{T}_3
    ~~\cong~~ \mathbb{K}\langle t_{12},t_{23} \rangle
    ~\otimes~\mathbb{K}[t_{12}+t_{13}+t_{23}].
\end{equation*}
The first isomorphism is obvious. For the second, use first relation (\ref{eq:DK-definition-one}) to express everything in terms of $t_{ij}$ with $i<j$, perform the base change from $\big(t_{12},t_{13},t_{23}\big)$
to $\big(t_{12},t_{23},\Lambda
\coloneqq t_{12}+t_{13}+t_{23}\big)$ and
observe that the relations (\ref{eq:DK-definition-two})
 of Definition \ref{DK-algebras} for $\mathcal{T}_3$ are equivalent to $[t_{12},\Lambda]=0$
and $[t_{23},\Lambda]=0$. Condition (\ref{eq:DK-definition-three}) is of course empty for $n=3$. \\
Returning to general $n$, note further
that the group of all permutations $n$ letters, $S_n$, \textbf{acts on the
left} on $\mathcal{T}_n$ by sending each generator
$t_{ij}$ to $t_{\sigma(i)\sigma(j)}$ by automorphisms
of unital algebras.
Next, for each integer $n\geqslant 2$ let $Y_n\subset 
\mathbb{C}^n$ be the open subset (see e.g.~\cite[p.267]{Kas95})
\begin{equation*}
   Y_n\coloneqq \big\{z\in \mathbb{C}^n~\big| ~
   \forall~i,j\in\mathbb{N},~1\leqslant i,j\leqslant n:~
    \mathrm{if~}i\neq j~\mathrm{then}~z_i\neq z_j
   \big\}
\end{equation*}
which is the well-known \textbf{ordered configuration space}: every element of $Y_n$ contains the information of the coordinates of $n$ distinguishable particles in the plane where no two particles occupy the same position.

Recall that the usual permutation of coordinates defines
a well-defined \textbf{right action of the permutation group}
$S_n$ on $\mathbb{C}^n$ given by 
$z=(z_1,\ldots,z_n)\mapsto(z_{\sigma(1)},\ldots, z_{\sigma(n)})=:z\sigma$ for each permutation 
$\sigma\in S_n$. This right action obviously preserves $Y_n$ on which it acts freely, and the quotient $X_n\coloneqq
Y_n/S_n$ is  a complex $n$-dimensional manifold called the (unordered) \textbf{configuration space}. Configuration spaces have been studied a lot,
see e.g.~the book \cite{FD01}. The fundamental groups of
$X_n$ and of $Y_n$ are well-known to be isomorphic to the \textbf{braid group of $n$ strands}, $B_n$, and to the \textbf{pure braid group of $n$ strands}, $P_n\subset B_n$, respectively. Moreover, but this fact is not necessary for the sequel, the naming `infinitesimal braid
relations' stems from the fact that the completion
of the $n$th Drinfel'd-Kohno algebra with respect to
the obvious filtration induced by the free algebra is
isomorphic to the completion of the group algebra of the pure braid group
$P_n$ with respect to its augmentation ideal, see e.g.~\cite{Koh85} and \cite[p.147, Prop.~1.3.3]{Koh87} for details.

The following well-known (formal) connection is very important, see
\cite{KZ84}:
\begin{definition}
	Let $n \geqslant 2$ and $\mathcal{A}$ a unital associative complex algebra containing $n(n-1)/2$
	elements $A_{ij}=A_{ji}$ (indexed by $1\leqslant i\neq j\leqslant n$) satisfying the infinitesimal braid relations
    (\ref{eq:DK-definition-two})
	and (\ref{eq:DK-definition-three}) where the generators
	$t_{ij}$ are replaced by the elements $A_{ij}$. The formal \textbf{Knizhnik--Zamolodchikov 
		$(\RuK \Run \RuZ \Rua)$-connection} ${}^{(n)}{\Gamma_{\RuK \Run \RuZ \Rua}}$ on $Y_n$ (with respect to $\mathcal{A}$) is defined as follows: 
	\begin{equation}\label{EqDKKZDefKnZaConnection}
	^{(n)}{\Gamma_{\RuK \Run \RuZ \Rua}}
	    (z_1,\ldots,z_n)
	 \coloneqq 
	  \sum_{1 \leqslant i < j \leqslant n} \frac{A_{ij}}{z_i-z_j} 
	  (\mathsf{d} z_i - \mathsf{d} z_j). 
	  \end{equation}
\end{definition}
\noindent Clearly, the $\RuK \Run \RuZ \Rua$-connection is complex rational in the sense of (\ref{EqFormFlatDefComplexRationalConn}). We have the following 
\begin{theorem}\label{TKnZhConnectionIsFlat}
	For all integers $n\geqslant 2$ the Knizhnik--Zamolodchikov connection is (formally) flat.
\end{theorem}
\noindent A very detailed proof of this statement can be found in 
C.Kassel's book \cite[p.452-454]{Kas95}.

It is easy to see that for all integers $n\geqslant 2$ 
the $\RuK \Run \RuZ \Rua$-connection 
$^{(n)}{\Gamma_{\RuK \Run \RuZ \Rua}}$ is \textbf{invariant by all}
pull-backs with respect to \textbf{translations} $T_v:Y_n\to Y_n:z\mapsto z+(v,v,\ldots,v)$ (for all $v\in\mathbb{C}$)
and with respect to all \textbf{complex homotheties} $H_p:Y_n\to Y_n$ given by $z\mapsto pz$ for all $p\in\mathbb{C}^\times=\mathbb{C}\setminus\{0\}$, in the sense that
$T_v^*\left(^{(n)}\Gamma_{\RuK \Run \RuZ \Rua}\right)=
^{(n)}{\Gamma_{\RuK \Run \RuZ \Rua}}$ and
$H_p^* \left(^{(n)}\Gamma_{\RuK \Run \RuZ \Rua}\right)=
^{(n)}{\Gamma_{\RuK \Run \RuZ \Rua}}$. Note further that,
for any integers $1\leqslant i\neq j\leqslant n$, if `particle $i$ is near to particle $j$', (i.e.~the distance $|z_i-z_j|$
becomes `very small') then the term containing $A_{ij}$ 
in the $\RuK \Run \RuZ \Rua$-connection will
be `very large' compared to the others: this intuition will motivate the choice of paths in the following sections.

For $n=2$ and $n=3$
there are the following isomorphisms of open sets of
$\mathbb{C}^2$ and of $\mathbb{C}^3$ which are given by
explicit bijective complex rational maps:
\begin{equation*}
Y_2~\cong~ \mathbb{C}^\times\times \mathbb{C}
~~~\mathrm{and}~~~
Y_3~\cong~ \mathbb{C}^{\times\times}\times\mathbb{C}^\times\times \mathbb{C}
\end{equation*}
where
$\mathbb{C}^\times\coloneqq \mathbb{C}\setminus\{0\}$
and $\mathbb{C}^{\times\times}\coloneqq \mathbb{C}\setminus\{0,1\}$ are the simply punctured complex plane and the \textbf{doubly punctured complex plane} -which will be important in the sequel--, respectively.\\
In fact, the invertible linear map $(z_1,z_2)\mapsto (z_1-z_2,z_1+z_2)$ gives the first isomorphism. For $n=3$
the invertible rational map $Y_3~\to~ \mathbb{C}^{\times\times}\times\mathbb{C}^\times\times \mathbb{C}$ given by
\begin{equation}\label{EqDKKZIsoYThree}
(z_1,z_2,z_3)\mapsto 
\left(\frac{z_2-z_1}{z_3-z_1}, z_3-z_1, z_1\right)~~~
\mathrm{with~inverse~}~~\vartheta:(z,v,w)\mapsto (w,zv+w,v+w)
\end{equation}
defines an isomorphism concerning $Y_3$, see also \cite[p.1453]{Dri89} or
\cite[p.469, eqn (7.3)]{Kas95}.
 An elementary computation shows that for
$n=3$ the pullback of 
${}^{(3)}{\Gamma_{\RuK \Run \RuZ \Rua}}$ with respect to
the rational map $\vartheta:\mathbb{C}^{\times\times}\times\mathbb{C}^\times\times \mathbb{C}\to Y_3$, see (\ref{EqDKKZIsoYThree}), is equal to
\begin{equation}\label{EqDKKZPullBackOfKZThree}
  \left(\vartheta^*\left({}^{(3)}{\Gamma_{\RuK \Run \RuZ \Rua}}
  \right)\right)(z,v,w) =
  \left(\frac{A_{12}}{z}+\frac{A_{23}}{z-1}\right)
     \mathsf{d}z ~+~
    \frac{A_{12}+A_{13}+A_{23}}{v}\mathsf{d} v
\end{equation}
where we have used Proposition \ref{PFormConnComplexPullBackEqualReal} and equations
(\ref{EqFormConnDefPullBackConn}) and
(\ref{EqFormConnComplexPullBackConn}). It provides
an important motivation for the connections
used in the following Sections since they resemble the 
first summand on the right hand side of the previous equation (\ref{EqDKKZPullBackOfKZThree}).
Note further that the right action of the permutation group
$S_3$ on $Y_3$ can be transferred to $\mathbb{C}^{\times\times}\times\mathbb{C}^\times\times \mathbb{C}$ and projected to the doubly punctured plane
$\mathbb{C}^{\times\times}$ by means of the maps
(\ref{EqDKKZIsoYThree}). This gives the following
maps on $\mathbb{C}^{\times\times}$ as can easily be computed:
\begin{equation}\label{EqDKKZPermActionOnCTimesTimes}
  \begin{array}{ccccc}
  \tau_{12}(z)  = \frac{z}{z-1}, 
   & 
  \tau_{23}(z)  = \frac{1}{z},
  &
  \tau_{13}(z)  = 1-z, 
  & 
  \zeta(z)=\frac{1}{1-z},
  &
  \zeta^{-1}(z)=\zeta\big(\zeta(z)\big)=\frac{z-1}{z}
  \end{array}
\end{equation}
where $\tau_{ij}$ denotes the transposition exchanging
$i$ and $j$, and $\zeta$ denotes the cyclic permutation
$(1,2,3)\mapsto (3,1,2)$. The fact that the maps in (\ref{EqDKKZPermActionOnCTimesTimes}) are well-defined
complex rational bijections on the doubly punctured plane
satisfying the identities of a right action of the symmetric group $S_3$ can also be shown directly without reference to the configuration space $Y_3$.

\subsection{The Drinfel'd associator: Definition and
	  elementary properties}
	\label{SubSecDrinfeldAssociatorDefinitionAndProperties}
	
In this Section the Drinfel'd associator is treated:
we are not following the usual definition, but use the statement of \cite[p.465, Lemma XIX.6.3]{Kas95} as a definition. The parallel transport we are interested in
is denoted there by $G_a(1-a)$ with $a=\delta$.
	
Let $\mathcal{A}$ be an arbitrary complex unital associative algebra. Set $U \coloneqq\  ]0,1[ \subset \mathbb{R}$. For any two given elements
$A,B\in\mathcal{A}$ define the formal connection
\begin{equation}\label{EqDrinfeldAssDefConn}
 \Gamma(B,A)(x)\coloneqq \left(\frac{1}{x}A+\frac{1}{x-1}B\right)
                 \mathsf{d}x
\end{equation}
on $U$ which is obviously well-defined and in addition a flat formal connection, see $\S$\ref{SubSecFlatFormalConnections}, because $U$
is one-dimensional. One of the motivations to use it
is the real version of the first summand of (\ref{EqDKKZPullBackOfKZThree}). Note further that the interval inversion $\iota:~]0,1[~\to~]0,1[~$ defined by
\begin{equation}
\label{eq:inversion-formula-interval}
    \iota(x)=1-x
\end{equation}
is well-defined and smooth, and it is easy to compute
the pulled-back connection
\begin{equation}\label{EqDrinfeldAssInversionPullBack}
   \iota^*\Gamma(B,A)=\Gamma(A,B).
\end{equation}
For all $\delta,\epsilon\in J\coloneqq~]0,1/4]$ we define the affine
path $c_{(\delta\epsilon)}:[0,1]\to U$ from $\delta$ to $1-\epsilon$, viz.
\begin{equation}\label{EqDrinfeldAssDefPathDeltaEpsilon}
     c_{(\delta\epsilon)}(s)\coloneqq (1-s)\delta+s(1-\epsilon)=
       \delta+s(1-\delta-\epsilon).
\end{equation}
Then
\begin{equation}\label{EqDrinfeldAssConnectionOnAffinePath}
   \Gamma(B,A)^{(c_{(\delta\epsilon)})}(s)
   =\frac{1-\delta-\epsilon}{\delta+s(1-\delta-\epsilon)}
      A
      +
      \frac{1-\delta-\epsilon}
       {\delta-1+s(1-\delta-\epsilon)}
      B.
\end{equation}
We are interested in the parallel transport 
${}^{\Gamma(B,A)}W^{(c_{(\delta\epsilon)})}_{10}$ along the path $c_{(\delta\epsilon)}$ from $\delta$ to $1-\epsilon$,
see $\S$\ref{SubSecFormalConnectionsAndParallelTransports} for definitions and notations. Setting $J'\coloneqq J\times J$ it follows from the general theory described in $\S$\ref{SubSec Norms and limits} that 
the map $(\delta,\epsilon)\mapsto {}^{\Gamma(B,A)}W^{(c_{(\delta\epsilon)})}_{10}$ is  an element  of the algebra 
$\Big(\mathrm{Fun}\big(J',\mathbb{C}\big)\otimes 
\mathcal{A}\Big)[[\lambda]]$, see (\ref{EqNormsLimDefPropagatorDepJ}).
It can be expressed in terms of iterated integrals in the following way: we make a change of variables
$u\coloneqq \delta+s(1-\delta-\epsilon)$, and we set
$A_0\coloneqq A$, $A_1\coloneqq B$. Hence, 
\begin{eqnarray}
  \lefteqn{
  	{}^{\Gamma(B,A)}W^{(c_{(\delta\epsilon)})}_{10} = 1} 
      \nonumber \\
     & & +
     \sum_{r=1}^\infty \lambda^r\sum_{i_1,\ldots,i_r=0}^1 
     \left(
 \int_{\delta}^{1-\epsilon}\frac{1}{u_1-i_1}
  \left(\int_{\delta}^{u_1}\frac{1}{u_2-i_2} 
      \left(\cdots
   \left(\int_{\delta}^{u_{r-1}}\frac{1}{u_r-i_r}
        \mathsf{d}u_r\right)\cdots 
        \right)\mathsf{d}u_{2}\right)\mathsf{d}u_1\right)
        \nonumber \\
      & &  \qquad \qquad\qquad\qquad \qquad\qquad\qquad
      A_{i_1}\cdots A_{i_r} 
      \label{EqDrinfeldAssParTranspIterIntegrals}
\end{eqnarray}

It is to be expected that
the preceding expression becomes singular whenever $\delta\to 0$ or $\epsilon\to 0$: in order to see this assume for a moment that
$A$ and $B$ commute. Clearly, $\Gamma(B,A)^{(c_{(\delta\epsilon)})}$ commutes with its primitive, and a straight-forward computation following formula (\ref{EqFormalLinODEYCommutesWithIOfY}) of 
Proposition \ref{PFormalLinODEPropagatorPropert} gives
\begin{equation}\label{EqDrinfeldAssParTranspCommutingAB}
  \mathrm{if}~AB=BA~~~~\mathrm{then}~~~~
  {}^{\Gamma(B,A)}W^{(c_{(\delta\epsilon)})}_{10}
    =e^{\lambda \ln(\epsilon) B}
      e^{\lambda\big(\ln(1-\epsilon)A-\ln(1-\delta)B\big)}
      e^{-\lambda \ln(\delta)A}
\end{equation}
showing that the divergences of the parallel transport are the left and the right factors and are logarithmic for
$\delta\to 0$ or $\epsilon\to 0$ in that particular case
whereas the middle factor converges to $1$.\\
Returning to the general case, in order to capture the singular terms we shall break the computation in two parts
separated by the mid-point $1/2$: consider the following
`exponential half-paths' $\tilde{c}_{(1,\delta)},\tilde{c}_{(2,\epsilon)}:[0,1]\to U$ defined by
\begin{equation}\label{EqDrinfeldAssExponentialPaths}
  \begin{array}{ccccccc}
  \tilde{c}_{(1,\delta)}(s) & \coloneqq
    &\frac{1}{2}e^{\ln(2\delta)(1-s)} & \mathrm{joining}&
    \delta &\to &\frac{1}{2}, \\
   \tilde{c}_{(2,\epsilon)}(s) & \coloneqq
   &1-\frac{1}{2}e^{\ln(2\epsilon)s} & \mathrm{joining}&
   \frac{1}{2} &\to& 1-\epsilon.
  \end{array}
\end{equation}
Hence, the composed path $\tilde{c}_{(2,\epsilon)}*
\tilde{c}_{(1,\delta)}$ is continuous and piecewise
smooth with singular set $D=\{0,1/2,1\}$ and joins $\delta\to 1-\epsilon$.  The following
continuous piecewise smooth reparametrization
$\gamma:[0,1]\to [0,1]$ (with singular set $\{0,1/2,1\}$) obviously 
links the affine path
 $c_{(\delta\epsilon)}$ with $\tilde{c}_{(2,\epsilon)}*
 \tilde{c}_{(1,\delta)}$:
\begin{equation}\label{EqDrinfeldAssReparAffineToExp}
 \gamma(s)\coloneqq \left\{
      \begin{array}{cl}
        \frac{\frac{1}{2}e^{\ln(2\delta)(1-2s)}-\delta}
         {1-\delta-\epsilon} & \mathrm{if~} 0\leqslant s
                                  \leqslant \frac{1}{2}, \\
      \frac{1-\frac{1}{2}e^{\ln(2\delta)(2s-1)}-\delta}
      {1-\delta-\epsilon} & \mathrm{if~} \frac{1}{2} \leqslant s\leqslant 1 
      \end{array}                           
 \right.
   ~~~\mathrm{hence}~~~
  c_{(\delta\epsilon)}\circ\gamma = \tilde{c}_{(2,\epsilon)}*\tilde{c}_{(1,\delta)}.
\end{equation}
Using the interval inversion $\iota$ as a continuous piecewise smooth reparametrization $[0,1]\to [0,1]$ given by \eqref{eq:inversion-formula-interval} we can write 
\[
    \tilde{c}_{(1,\delta)}
    =\iota\circ\tilde{c}_{(2,\delta)}\circ \iota.   
\] 
Since parallel transport is
independent on reparametrizations, see (\ref{EqFormConnWReparamInv}), we get
\begin{eqnarray}
 {}^{\Gamma(B,A)}W^{(c_{(\delta\epsilon)})}_{10}
& = &
 {}^{\Gamma(B,A)}W^{(c_{(\delta\epsilon)}\circ\gamma)}_{10} \stackrel{(\ref{EqDrinfeldAssReparAffineToExp})}{=}
 {}^{\Gamma(B,A)}W_{10}^{(\tilde{c}_{(2,\epsilon)}*
	\tilde{c}_{(1,\delta)})} \nonumber \\
& \stackrel{(\ref{EqFormConnComposOfPathsParTransp})}{=}&
 {}^{\Gamma(B,A)}W_{10}^{(\tilde{c}_{(2,\epsilon)})}
 ~{}^{\Gamma(B,A)}W_{10}^{
 	(\iota\circ\tilde{c}_{(2,\delta)}\circ \iota)}
      \nonumber \\
   & \stackrel{(\ref{EqFormConnPullBackParTransp}),
   	(\ref{EqDrinfeldAssInversionPullBack})
   	 (\ref{EqFormConnInversionFormula})}{=} &
{}^{\Gamma(B,A)}W_{10}^{(\tilde{c}_{(2,\epsilon)})}
~\left({}^{\Gamma(A,B)}W_{10}^{(\tilde{c}_{(2,\delta)})}
   \right)^{-1}.
   \label{EqDrinfeldAssParTranspAffineFactorizationExpo}
\end{eqnarray}
It follows that it suffices to compute the parallel transport along the exponential half-path 
$\tilde{c}_{(2,\epsilon)}$, the parallel transport along the other
half $\tilde{c}_{(1,\delta)}$ follows from the symmetry and an exchange of $A$ and $B$.\\
The choice of the exponential function in the path
$\tilde{c}_{(2,\epsilon)}$ becomes clear when
computing
\begin{equation}\label{EqDrinfeldAssGammaBAComp}
   \Gamma(B,A)^{(\tilde{c}_{(2,\epsilon)})}(s) 
   =
   \ln(2\epsilon)B +
     \frac{-\ln(2\epsilon)}{2e^{-\ln(2\epsilon)s}-1}A,
\end{equation}
and we see that \textbf{the term in front of $B$ does not depend on $s$}.
\begin{lemma}\label{LDrinfeldAssFactorizationHalfAss}
	We have the following factorization of the parallel transport $s\mapsto {}^{\Gamma(B,A)}W_{s 0}^{(\tilde{c}_{(2,\epsilon)})}$
	in the algebra 
	$\Big(\mathrm{Fun}\big(]0,1/4],
	\mathcal{C}_{\{0,1\}}^\infty\big([0,1],\mathbb{C}\big)
	\big)
	\otimes \mathcal{A}\Big)[[\lambda]]$
	\begin{equation}
	\label{EqDrinfeldAssFactorizatonHalfParTransp}
     {}^{\Gamma(B,A)}W_{s0}^{(\tilde{c}_{(2,\epsilon)})}
     =
     e^{\lambda\ln(\epsilon)sB}\psi_\epsilon(B,A)(s)
	\end{equation}
	where $(s,\epsilon)\mapsto\psi_\epsilon(B,A)(s)$ 
	is in the group 
	$\mathcal{G}_{\mathcal{B}}$ of bounded terms 
	(w.r.t.~$(s,\epsilon)$, see (\ref{EqNormsLimDefRBH})
	and (\ref{EqFilNormsDefGroupsGRGBGH})). We set
	\begin{equation*}
	   \psi_\epsilon(B,A)\coloneqq \psi_\epsilon(B,A)(1).
	\end{equation*}
	Moreover,
	there is a well-defined element
	$\psi(B,A)\in\mathcal{A}[[\lambda]]$ such that the
	 following limit exists
	\begin{equation}\label{EqDrinfeldAssLimitOfLittlePhis}
	 \lim_{\epsilon\to 0}\psi_\epsilon(B,A)(s)=
	 \left\{
	       \begin{array}{cl}
	        1 & \mathrm{if}~s=0, \\
	        \psi(B,A)~\in~\mathcal{A}[[\lambda]]
	           & \mathrm{if}~0<s\leqslant 1,
	        \end{array}
	        \right.
	\end{equation}
	in the sense of limits discussed in $\S$\ref{SubSec Norms and limits}, see
	(\ref{EqNormsLimitDefLimit}),
	 (\ref{EqNormsLimDefLimFormPowerSer}) and
	 Proposition \ref{PNormsLimLimitRules}.
\end{lemma}
\begin{proof}
 In (\ref{EqDrinfeldAssGammaBAComp}) we set
 \[
     Y_\epsilon(s) \coloneqq \ln(2\epsilon)B~~~\mathrm{and}
     ~~~Z_\epsilon(s) \coloneqq \frac{-\ln(2\epsilon)}{2e^{-\ln(2\epsilon)s}-1}A
 \]	
 and use the factorization statement 
 (\ref{EqFormalLinODEFactorization}): 
 \begin{equation}\label{EqDrinfeldAssFirstFactorization}
 {}^{\Gamma(B,A)}W_{s0}^{(\tilde{c}_{(2,\epsilon)})}
 = U_{s0}^{(\epsilon)}\Xi_{s0}^{(\epsilon)}.
 \end{equation}
 Clearly, the formal linear ODE $\mathsf{d}U_{\cdot 0}^{(\epsilon)}/\mathsf{d}s
 =\lambda Y_\epsilon U_{\cdot 0}^{(\epsilon)}$ with initial condition
 $1$ is trivially given by the exponential function
 $U_{s 0}^{(\epsilon)}= e^{\lambda\ln(2\epsilon)sB}$, and we have
 to solve the formal linear ODE with initial condition
 $1$,
 \[
     \frac{\mathsf{d}\Xi_{s0}^{(\epsilon)}}{\mathsf{d}s}(s)
     =e^{-\lambda\ln(2\epsilon)sB}Z_\epsilon(s)
     e^{\lambda\ln(2\epsilon)sB}\Xi_{s0}^{(\epsilon)}
     =e^{-\lambda\ln(2\epsilon)s\mathsf{ad}_B}
     \big(Z_\epsilon(s)\big)\Xi_{s0}^{(\epsilon)}
 \]
 where $\mathsf{ad}_B:\mathcal{A}\to \mathcal{A}$ denotes the
 usual adjoint map $\xi\mapsto B\xi-\xi B$, and we have used the well-known identity that conjugation with exponentials
 is the exponential of $\mathsf{ad}$ which is standard
  in Lie group theory, see e.g.~\cite[p.38,~Cor.4.25]{KMS93},
 and can easily be proved algebraically. We can compute
 the solution $\Xi_{s0}^{(\epsilon)}$ in terms of iterated integrals, 
 see (\ref{EqFormalLinODEIteratedIntegrals}), where the
 following abbreviations make computations easier: set
 \[
    \nu\coloneqq -\ln(2\epsilon)
     ~~~\mathrm{and}~~~\tau\coloneqq \nu s, 
     \tau_i\coloneqq \nu s_i ~~\forall~i\in\mathbb{N}.
 \]
 Since $0<2\epsilon \leqslant 1/2 <1$ it follows
 that $\nu> 0$ and that the limit $\lim_{\epsilon\to 0}$ 
 corresponds to $\lim_{\nu\to +\infty}$. Then $\Xi_\epsilon$ is given by the following expression:
 \begin{eqnarray}
 \Xi_{s0}^{(\epsilon)}	& = & 1
 +\sum_{r=1}^\infty \lambda^r \sum_{\ell_1,\ldots,\ell_r=0}^\infty 
 \frac{\lambda^{\ell_1+\cdots +\ell_r}}{\ell_1!\cdots \ell_r!}
 \mathsf{ad}_B^{\ell_1}(A)\cdots\mathsf{ad}_B^{\ell_r}(A) 
     \nonumber \\
 & &
  \underbrace{\left(
  \int_{0}^{\nu s}\frac{\tau_1^{\ell_1}}{2e^{\tau_1}-1}
  \left(
  \int_{0}^{\tau_1}\frac{\tau_2^{\ell_2}}{2e^{\tau_2}-1}
  \left(
  \cdot \cdot
  \left(
  \int_{\delta}^{\tau_{r-1}}
          \frac{\tau_r^{\ell_r}}{2e^{\tau_r}-1}
  \mathsf{d}\tau_r\right)\cdot \cdot 
  \right)\mathsf{d}\tau_{2}\right)\mathsf{d}\tau_1\right)
  }_{\coloneqq I_{r,\ell_1,\ldots,\ell_r}(s,\nu)}.
    \label{EqDrinfeldAssIterIntForPsi}
 \end{eqnarray}
 We shall prove that for all non-negative integers
 $r,\ell_1,\ldots,\ell_r$ with $r\geqslant 1$ and all $s\in[0,1]$
 the iterated real integral $I_{r,\ell_1,\ldots,\ell_r}(s,\nu)$
 at the end of (\ref{EqDrinfeldAssIterIntForPsi}) converges to a non-negative real number for $\nu\to +\infty$: this will prove that the limit 
 $\lim_{\epsilon\to 0}\Xi_\epsilon(s)$ exists. In case $s=0$ this is of course obvious since all these integrals vanish. For $s>0$, the crucial observation is that all the real numbers 
 $\tau_1,\ldots,\tau_r$ are non-negative whence all the functions $\tau_i\mapsto \frac{\tau_i^{\ell_i}}{2e^{\tau_i}-1}$, $i\in \mathbb{N}\setminus \{0\}$, take non-negative values
 on the interval $[0,\nu s]$. Thanks to the monotonicity of the Riemann integral it follows that enlarging $\nu$
 makes the interval $[0,\nu s]$ bigger which in turn
 makes the value of the iterated integral larger: hence the
 function $[\ln(2),+\infty[~\mapsto ~[0,+\infty[$ given by
 $\nu\mapsto I_{r,\ell_1,\ldots,\ell_r}(s,\nu)$ is strictly increasing. By the well-known principle stating that every increasing bounded sequence of
 real numbers converges it suffices to show that all the
 integrals $I_{r,\ell_1,\ldots,\ell_r}(s,\nu)$ admit an upper bound independent on
 all $s\in[0,1]$ and $\nu\in[\ln(2),+\infty[$: indeed, the elementary inequality $e^{\tau_i}-1\geqslant 0$ 
 for all positive integer $i$ (since $\tau_i\geqslant 0$)
 implies
 \[
 \forall~i\in\mathbb{N}\setminus \{0\}:~~
   2e^{\tau_i}-1
   =e^{\tau_i}+e^{\tau_i}-1\geqslant e^{\tau_i},~~\mathrm{hence}~~
    \frac{\tau_i^{\ell_i}}{2e^{\tau_i}-1}\leqslant 
      \tau_i^{\ell_i}e^{-\tau_i},
 \]
 and the integral $I_{r,\ell_1,\ldots,\ell_r}(s,\nu)$ can be bounded by
 \[
 I_{r,\ell_1,\ldots,\ell_r}(s,\nu)\leqslant
 \left(
  \int_0^{\nu s}\tau_1^{\ell_1}e^{-\tau_1}\mathsf{d}\tau_1
  \right)
  \cdots
  \left(
  \int_0^{\nu s}\tau_r^{\ell_r}e^{-\tau_r}\mathsf{d}\tau_r
  \right) \leqslant \ell_1!\cdots \ell_r!
 \]
 thanks to the well-known integral (for all non-negative integers $n$)
 \[
       \int_0^\infty \tau^ne^{-\tau}\mathsf{d}\tau = n!.
 \]
 This shows that the limit 
 $\lim_{\epsilon\to 0}\Xi_\epsilon(s)$ exists and does not depend on $0<s\leqslant 1$.
 Using the factorization equation (\ref{EqDrinfeldAssFirstFactorization}), the trivial fact
 that $\ln(2\epsilon)=\ln(2)+\ln(\epsilon)$ and defining
 \[
     \psi_\epsilon(B,A)(s)\coloneqq e^{\lambda\ln(2) sB}
                             \Xi_\epsilon(s)          
                        \]
 shows the factorization equation (\ref{EqDrinfeldAssFactorizatonHalfParTransp})
 and the limit (\ref{EqDrinfeldAssLimitOfLittlePhis}).
 In particular, it implies that $\epsilon\mapsto
 \psi_\epsilon(B,A)$ is bounded, i.e.~it is an element of
 $\mathcal{G}_\mathcal{B}$.
\end{proof}
\noindent
This Lemma --together with the factorization equation
(\ref{EqDrinfeldAssParTranspAffineFactorizationExpo})-- has the following immediate and very important
consequence
\begin{theorem} \label{TDrinfeldAssDefAssociator}
	With the above notations:
The parallel transport 
	${}^{\Gamma(B,A)}W^{(c_{(\delta\epsilon)})}_{10}$
	along the path $c_{(\delta\epsilon)}$,
	see (\ref{EqDrinfeldAssDefPathDeltaEpsilon}),
	factorizes in the following way
	\begin{equation}
	\label{EqDrinfeldAssFactorizatonAffinePathParTransp}
	{}^{\Gamma(B,A)}W^{(c_{(\delta\epsilon)})}_{10}
	   =e^{\lambda\ln(\epsilon)B}
	     \Phi_{\delta,\epsilon}(A,B)
	     e^{-\lambda\ln(\delta)A},
	\end{equation}
	with
	\begin{equation}\label{EqDrinfeldAssDefAssDeltaEpsilon}
	  \Phi_{\delta,\epsilon}(A,B)\coloneqq 
	  \psi_\epsilon (B,A)
	  \left(\psi_\delta (A,B)\right)^{-1}.
	\end{equation}
	The following limit exists,
	\begin{equation}\label{EqDrinfeldAssDefAss}
	  \lim_{(\delta,\epsilon)\to (0,0)}
	  \Phi_{\delta,\epsilon}(A,B)
	  \coloneqq \Phi(A,B)~\in~\mathcal{A}[[\lambda]],
	\end{equation}
	and is called the \textbf{Drinfel'd associator} 
	w.r.t.~$A,B\in\mathcal{A}$.	
\end{theorem}
\noindent We kept the notation for the Drinfel'd associator used in all the literature although a transposition of the arguments, $\Phi(B,A)$,
would `better concatenate' in the identities. 

We collect some properties of the Drinfel'd associator:
\begin{equation}\label{EqDrinfeldAssInverseEqAGoesToB}
  \Phi_{\delta,\epsilon}(A,B)^{-1}
  =\Phi_{\epsilon,\delta}(B,A)~~~
  \mathrm{hence}~~~\boxed{\Phi(A,B)^{-1}=\Phi(B,A)}
\end{equation}
which immediately follows from the definitions
(\ref{EqDrinfeldAssDefAssDeltaEpsilon}) and
(\ref{EqDrinfeldAssDefAss}). Next, computing the coefficient of $\lambda^1$ of (\ref{EqDrinfeldAssFactorizatonAffinePathParTransp}) we get from the right hand side
\[
     \ln(\epsilon)B +
     \left(\Phi_{\delta,\epsilon}(A,B)\right)_1
     -\ln(\delta)A
\] 
and from the left hand side the integral (compare (\ref{EqDrinfeldAssParTranspIterIntegrals}))
\[
 \int_{\delta}^{1-\epsilon}\frac{du}{u}A
 -\int_{\delta}^{1-\epsilon}\frac{du}{1-u}B
 = \ln(1-\epsilon)A -\ln(\delta)A
     +\ln(\epsilon)B -\ln(1-\delta)B
\]
showing 
\begin{equation*}
\left(\Phi_{\delta,\epsilon}(A,B)\right)_1
  =\ln(1-\epsilon)A-\ln(1-\delta)B,~~~
  \mathrm{hence}~~~
  \boxed{\Phi(A,B
  )-1~\in ~\lambda^2\mathcal{A}[[\lambda]]}.
\end{equation*}
because obviously $\ln(1-x)\to 0$ if $x\to 0$ whence
$\lim_{(\delta,\epsilon)\to(0,0)}
\left(\Phi_{\epsilon,\delta}(B,A)\right)_1=0$.\\
Suppose furthermore that there are elements $\Lambda,\Lambda'\in\mathcal{A}$ which are \textbf{central
for $A,B$} in the sense that
\begin{equation*}
 [\Lambda,A]=0=[\Lambda',A]~~\mathrm{and}~~
 [\Lambda,B]=0=[\Lambda',B]~~\mathrm{and}
 ~~[\Lambda,\Lambda']=0.
\end{equation*}
Since the connection $\Gamma(B+\Lambda',A+\Lambda)$ evaluated on the path $c_{(\delta\epsilon)}$, 
$\Gamma(B+\Lambda',A+\Lambda)^{(c_{(\delta\epsilon)})}$,
see (\ref{EqDrinfeldAssConnectionOnAffinePath}) is equal to $\Gamma(\Lambda',\Lambda)^{(c_{(\delta\epsilon)})}
+
\Gamma(B,A)^{(c_{(\delta\epsilon)})}$ we can use the factorization statement
(\ref{EqFormalLinODEFactorization}) with
$Y=\Gamma(\Lambda',\Lambda)^{(c_{(\delta\epsilon)})}$, 
$Z=\Gamma(B,A)^{(c_{(\delta\epsilon)})}$,
and the fact that $\Lambda$ and $\Lambda'$ commute with all 
words in $\mathcal{A}$ whose letters are $A,B,\Lambda$ or
$\Lambda'$ (hence $U_{\cdot 0}^{-1}ZU_{\cdot 0}=Z$ in 
(\ref{EqFormalLinODEFactorization})) we can use
(\ref{EqFormalLinODEYCommutesWithIOfY}) and
(\ref{EqDrinfeldAssParTranspCommutingAB}) to conclude
that
\begin{equation*}
  \Phi_{\delta,\epsilon}(A+\Lambda,B+\Lambda')
  =e^{\lambda\big(\ln(1-\epsilon)\Lambda
  	      -\ln(1-\delta)\Lambda'\big)}
  	   \Phi_{\delta,\epsilon}(A,B)  
\end{equation*}
hence, passing to the limit $(\delta,\epsilon)\to(0,0)$,
\begin{equation}\label{EqDrinfeldAssIndepOnCentralTerms}
 \mathrm{if}~\Lambda,\Lambda'~\mathrm{are~central~for~}
 A,B:
 ~~~\boxed{\Phi(A+\Lambda,B+\Lambda')=\Phi(A,B)}.
\end{equation}
Note that --but this is not needed later on-- if $\mathcal{A}$ carries the structure of a bialgebra (see \cite{Kas95}, Sect. III.2) and $A,B \in \mathcal{A}$ are primitive elements then $\Phi(A,B)$ is a (formally) group-like element (see \cite{Dri90}, p. 836).
In particular, when $\mathcal{A}=\mathbb{C}\langle A,B\rangle$ the Drinfel'd associator is thus a formal
exponential series whose exponent is an element of the
formal power series with coefficients in the free Lie algebra generated by two elements.

\subsection{The Hexagon Equation}
 \label{TheHexagonEquation}
Let $\mathcal{A}$ be an arbitrary complex unital associative algebra. Let $A,B,C\in\mathcal{A}$ be three arbitrary
elements such that the sum $\Lambda\coloneqq A+B+C$ commutes with
all the three, i.e.
\begin{equation}\label{EqHexAPlusBPlusCCentral}
   A\Lambda=\Lambda A,~B\Lambda=\Lambda B,~C\Lambda=\Lambda C.
\end{equation}
An important particular case would be a free choice of $A$ and $B$, and $C=-A-B$ whence $\Lambda=0$. Another important particular case is $\mathcal{A}=\mathcal{T}_3$
with $A=t_{12}$, $B=t_{23}$, and $C=t_{13}$.
We wish to prove the \textbf{Hexagon Equation for the Drinfel'd associator}, i.e.
\begin{equation}\label{EqHexDefHexagonEqn}
 e^{\lambda \pi\mathbf{i}\Lambda}
 ~=~e^{\lambda \pi\mathbf{i}A}~\Phi(C,A)
 ~e^{\lambda \pi\mathbf{i}C}~\Phi(B,C)
 ~e^{\lambda \pi\mathbf{i}B}~\Phi(A,B).
\end{equation}
The equation has an obvious cyclic symmetry $A\to B\to C\to A$, and contains the complex numbers $\pi$ and $\mathbf{i}$.
It is not so far-fetched to choose the following data:
let $U\subset \mathbb{C}$ be the doubly punctured complex
plane, namely
\begin{equation*}
    U\coloneqq \mathbb{C}^{\times\times}\coloneqq
       \mathbb{C}\setminus \{0,1\}.
\end{equation*}
Choose the complex version of the connection
$\Gamma(B,A)$, see (\ref{EqDrinfeldAssDefConn}), i.e.
\begin{equation}\label{EqHexDefConn}
\Gamma(B,A)\coloneqq \left(\frac{1}{z}A+\frac{1}{z-1}B\right)
\mathsf{d}z 
\end{equation}
which is flat, see Proposition \ref{PFormFlatComplexFlatEqRealFlat}.
The original motivation is the use of the 
$\RuK \Run \RuZ \Rua$-connection 
${}^{(3)}\Gamma_{\RuK \Run \RuZ \Rua}$ on the configuration
space $Y_3$, see (\ref{EqDKKZDefKnZaConnection}),
and its pull-back, see the first summand of
(\ref{EqDKKZPullBackOfKZThree}). Some formulas would look more natural on $Y_3$, but the computations are quicker
in the doubly punctured complex plane.

Recall the rational maps 
$\zeta,\zeta\circ \zeta =\zeta^{-1}:
\mathbb{C}^{\times\times}\to \mathbb{C}^{\times\times}$
defined by $\zeta(z)=\frac{1}{1-z}$ and $\zeta^{-1}(z)
=\frac{z-1}{z}$
coming from the cyclic permutations in $Y_3$, see
(\ref{EqDKKZPermActionOnCTimesTimes}). We compute
the pull-backs of the connection (\ref{EqHexDefConn}):
using Proposition \ref{PFormConnComplexPullBackEqualReal}
we get, upon setting $\tilde{C}\coloneqq -A-B=C-\Lambda$,
\begin{eqnarray}
 \left(\zeta^*\Gamma(B,A)\right)(z)
    & = & \Gamma(B,A)\left(\frac{1}{1-z}\right)
            \frac{1}{(1-z)^2}
            =\frac{1}{1-z}A+\frac{1}{z(1-z)}B
             \nonumber \\
    & \stackrel{\mathrm{partial~fraction~dec.}}{=}&
       \frac{1}{z}B+ \frac{-A-B}{z-1}
       = \Gamma(\tilde{C},B)(z).
       \label{EqHexZetaStarGamma}
\end{eqnarray}
Iterating this formula (recall that $\zeta\circ\zeta=\zeta^{-1}$) gives
\begin{equation}\label{EqHexZetaSquareStarGamma}
  \left(\big(\zeta^{-1}\big)^*\Gamma(B,A)\right)(z)
  =\Gamma(A,\tilde{C})(z).
\end{equation}
We shall now consider the parallel transport 
w.r.t.~the connection $\Gamma(B,A)$ along a continuous piecewise smooth loop $c_\delta$ depending on a parameter $\delta\in J=~]0,1/4]$ based at the point 
$\delta\in\mathbb{C}^{\times\times}$, which is the 
composition of six paths,
\begin{equation}\label{EqHexComposedLoop}
 c_\delta \coloneqq 
 c_{(\mathrm{\RoM{6}},\delta)}
  *\Big(c_{(\mathrm{\RoM{5}},\delta)}
  *\Big(c_{(\mathrm{\RoM{4}},\delta)}
  *\big(c_{(\mathrm{\RoM{3}},\delta)}
  *\big(c_{(\mathrm{\RoM{2}},\delta)}
  *c_{(\mathrm{\RoM{1}},\delta)}
  \big)\big)\Big)\Big)
\end{equation}
given by
\begin{equation}
  \begin{array}{ccccccccc}
  c_{(\mathrm{\RoM{1}},\delta)}(s) & \coloneqq & (1-s)\delta+s(1-\delta)
        & =&
      \delta+s(1-2\delta) & \mathrm{joining} & 
          \delta &\to & 1-\delta, \\
  c_{(\mathrm{\RoM{2}},\delta)}(s) & \coloneqq & 
            \frac{
  	           1-\frac{\delta}{2}
  	              -\frac{\delta}{2}e^{\mathbf{i}\pi s}
  	                  }
                   {
     	       1-\frac{\delta}{2}
     		+\frac{\delta}{2}e^{\mathbf{i}\pi s}
     		       }
     	         & =&
     1-\frac{\delta}
      {\left(1-\frac{\delta}{2}\right)e^{-\mathbf{i}\pi s}
     +\frac{\delta}{2} } & \mathrm{joining} & 
  1-\delta &\to & \frac{1}{1-\delta}, \\  
  c_{(\mathrm{\RoM{3}},\delta)} (s) &\coloneqq &
     \zeta\left(c_{(\mathrm{\RoM{1}},\delta)}(s)\right)  & = &
      \frac{1}{1-\delta-s(1-2\delta)}  &\mathrm{joining} & 
     \frac{1}{1-\delta} &\to & \frac{1}{\delta}, \\ 
   c_{(\mathrm{\RoM{4}},\delta)} (s) &\coloneqq &
   \zeta\left(c_{(\mathrm{\RoM{2}},\delta)}(s)\right)  
   & = &
   \frac{1}{2}+\left(\frac{1}{\delta}-\frac{1}{2}\right)
      e^{-\mathbf{i}\pi s} 
    &\mathrm{joining} & 
   \frac{1}{\delta} &\to & -\frac{1}{\delta}+1, \\ 
   c_{(\mathrm{\RoM{5}},\delta)} (s) &\coloneqq &
   \zeta\left(\zeta\left( c_{(\mathrm{\RoM{1}},\delta)}(s)\right)\right) 
    & = &
   \frac{\delta-1+s(1-2\delta)}{\delta+s(1-2\delta)}  &\mathrm{joining} & 
   -\frac{1}{\delta}+1 &\to & -\frac{\delta}{1-\delta}, \\ c_{(\mathrm{\RoM{6}},\delta)} (s) &\coloneqq &
   \zeta\left(\zeta\left(c_{(\mathrm{\RoM{2}},\delta)}(s)
   \right)\right)  
   & = & \frac{\delta}
   {
   	-\left(1-\frac{\delta}{2}\right)
   e^{-\mathbf{i}\pi s}+\frac{\delta}{2}
   } 
   &\mathrm{joining} & 
   -\frac{\delta}{1-\delta} &\to & \delta.    
  \end{array}
  \label{EqHexagonSixPaths}
\end{equation}
It is fairly easy to check that all the six paths take all their values in
the lower half plane (including the $x$ axis and excluding $0$ and $1$). \\
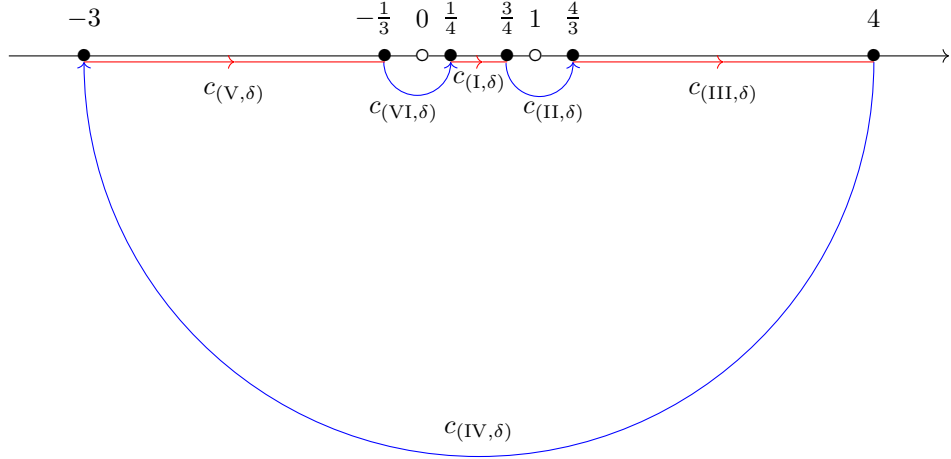
\begin{figure}[h!!]
	\centering
	\begin{tikzpicture}
	\draw (-5.5,0)--(0-0.08,0);
	\draw (0+0.08,0)--(1.5-0.08,0);
	\draw[->] (1.5+0.08,0)--(7,0);
	\node at (-4.5,0) {$\bullet$};
	\node at (-4.5,0.5) {\footnotesize $-3$};
	\node at (-0.5,0) {$\bullet$};
	\node at (-0.65,0.5) {\footnotesize $-\frac{1}{3}$};
	\node at (0,0) {$\circ$};
	\node at (0,0.5) {\footnotesize $0$};
	\node at (0.375,0) {$\bullet$};
	\node at (0.375,0.5) {\footnotesize $\frac{1}{4}$};
	\node at (1.125,0) {$\bullet$};
	\node at (1.125,0.5) {\footnotesize $\frac{3}{4}$};
	\node at (1.5,0) {$\circ$};
	\node at (1.5,0.5) {\footnotesize $1$};
	\node at (2,0) {$\bullet$};
	\node at (2,0.5) {\footnotesize $\frac{4}{3}$};
	\node at (6,0) {$\bullet$};
	\node at (6,0.5) {\footnotesize $4$};
	\draw[blue,->] (6,-0.08) arc (0:-180:5.25);
	\draw[blue,<-] (0.375,-0.08) arc (0:-180:0.445);
	\draw[blue,<-] (2,-0.08) arc (2:-180:0.445);
	\draw[red,->] (-4.5,-0.08)--(-2.5,-0.08);
	\draw[red] (-2.5,-0.08)--(-0.5,-0.08);
	\draw[red,->] (0.375,-0.08)--(0.8,-0.08);
	\draw[red] (0.8,-0.08)--(1.125,-0.08);
	\draw[red,->] (2,-0.08)--(4,-0.08);
	\draw[red] (4,-0.08)--(6,-0.08);
	\node at (0.75,-5) {\footnotesize $c_{(\mathrm{\RoM{4}}, \delta)}$};
	\node at (1.75,-0.75) {\footnotesize $c_{(\mathrm{\RoM{2}}, \delta)}$};
	\node at (-0.25,-0.75) {\footnotesize $c_{(\mathrm{\RoM{6}}, \delta)}$};
	\node at (-2.5,-0.5) {\footnotesize $c_{(\mathrm{\RoM{5}}, \delta)}$};
	\node at (0.76,-0.35) {\footnotesize $c_{(\mathrm{\RoM{1}}, \delta)}$};
	\node at (4,-0.5) {\footnotesize $c_{(\mathrm{\RoM{3}}, \delta)}$};
	\end{tikzpicture}
	\caption{The paths (\ref{EqHexagonSixPaths}) at $\delta=1/4$} \label{fig:Hexagon}
\end{figure}
The singular set $D$ for the loop $c_\delta$ is thus
equal to $\left\{0,\frac{1}{32},\frac{1}{16},\frac{1}{8},
    \frac{1}{4},\frac{1}{2}, 1\right\}$.

The following geometric description of the paths, see
Figure $\ref{fig:Hexagon}$, may perhaps clarify the whole procedure:  the three `odd' paths
$c_{(\mathrm{\RoM{1}},\delta)}$, 
$c_{(\mathrm{\RoM{3}},\delta)}$, and
$c_{(\mathrm{\RoM{5}},\delta)}$ parametrize the closed intervals $[\delta,1-\delta]$, 
$\zeta\big([\delta,1-\delta]\big)=[1/(1-\delta),1/\delta]$,
and
$\zeta^2\big([\delta,1-\delta]\big)=
[-(1/\delta)+1,-\delta/(1-\delta)]$, respectively, all along the $x$-axis. Note that $c_{(\mathrm{\RoM{3}},\delta)}$ and
$c_{(\mathrm{\RoM{5}},\delta)}$ are not affine paths in the sense of (\ref{EqFormConnDefAffinePath}). 
The three `even' paths
$c_{(\mathrm{\RoM{2}},\delta)}$, 
$c_{(\mathrm{\RoM{4}},\delta)}$, and
$c_{(\mathrm{\RoM{6}},\delta)}$ parametrize 
lower half circles with centres 
$1+\frac{\delta^2}{2-2\delta}$, $\frac{1}{2}$, and
$-\frac{\delta^2}{2-2\delta}$, respectively, having radii
$\frac{2\delta-\delta^2}{2-2\delta}$, $\frac{1}{\delta}-\frac{1}{2}$, and $\frac{2\delta-\delta^2}{2-2\delta}$, respectively, as can be checked by a lengthy, but elementary computation.
$c_{(\mathrm{\RoM{2}},\delta)}$ and 
$c_{(\mathrm{\RoM{6}},\delta)}$ are traced 
counterclockwise (where the parametrization is NOT uniform), and
 $c_{(\mathrm{\RoM{4}},\delta)}$ is traced clockwise
 with uniform parametrization.
 If one likes to get a motivation --which is not necessary
 for the arguments we are going to give-- on the
  `funny' form of the lower half cycle
 $c_{(\mathrm{\RoM{2}},\delta)}$: if we consider the uniformly parametrized half cycle 
 $[0,1]\to Y_3$ in the configuration space $Y_3$ given by
 \[
    s\mapsto \left(0, 
        1-\frac{\delta}{2}-\frac{\delta}{2}
           e^{\mathbf{i}\pi s},
         1-\frac{\delta}{2}+\frac{\delta}{2}
         e^{\mathbf{i}\pi s}  \right)
 \]
and apply the first component of the isomorphism
(\ref{EqDKKZIsoYThree}), $(z_1,z_2,z_3)\mapsto 
\frac{z_2-z_1}{z_3-z_1}$, then we get the path
$c_{(\mathrm{\RoM{2}},\delta)}$. The picture in
C.Kassel's book \cite[p.474,~Fig.~8.1]{Kas95}
illustrating these paths in $Y_3$ in a qualitative manner had proved to be very inspiring for us. Note finally that the whole picture
of the six paths in the doubly punctured plane has also
an obvious reflection symmetry (the `antiholomorphic' map 
$\mathsf{s}:z\mapsto 1-\bar{z}$) with respect to the straight line $x=\frac{1}{2}$. Hence, with 
the usual interval inversion $\iota$ of the interval $[0,1]$, $\iota(s)=1-s$, it is easy to see that the following holds by using the concrete formulas
(\ref{EqHexagonSixPaths}):
$\mathsf{s}\circ c_{(\mathrm{\RoM{1}},\delta)}\circ \iota
=c_{(\mathrm{\RoM{1}},\delta)}$, 
$\mathsf{s}\circ c_{(\mathrm{\RoM{2}},\delta)}\circ \iota
=c_{(\mathrm{\RoM{6}},\delta)}$,
$\mathsf{s}\circ c_{(\mathrm{\RoM{3}},\delta)}\circ \iota
=c_{(\mathrm{\RoM{5}},\delta)}$, and
$\mathsf{s}\circ c_{(\mathrm{\RoM{4}},\delta)}\circ \iota=c_{(\mathrm{\RoM{4}},\delta)}$.

We shall now compute the six parallel transports along the
six paths.
First, it is immediate that the parallel transport
${}^{\Gamma(B,A)}W^{(c_{(\mathrm{\RoM{1}},\delta)})}_{10}$
coincides with the parallel transport
${}^{\Gamma(B,A)}W^{(c_{\delta\delta})}_{10}$ of the
preceding Section, see (\ref{EqDrinfeldAssParTranspIterIntegrals}): this can be seen by using the smooth injection $]0,1[\to 
\mathbb{C}^{\times\times}$ to pull back the connection
(\ref{EqHexDefConn}) to the interval. Using the formulas
(\ref{EqHexZetaStarGamma}) and (\ref{EqHexZetaSquareStarGamma}) and the fact that
$c_{(\mathrm{\RoM{3}},\delta)}
=\zeta\circ c_{(\mathrm{\RoM{1}},\delta)}$ and
$c_{(\mathrm{\RoM{5}},\delta)}
=\zeta\circ \zeta\circ c_{(\mathrm{\RoM{1}},\delta)}$,
see (\ref{EqHexagonSixPaths}) we get --upon using (\ref{EqFormConnPullBackParTransp})-- the following
formulas
\begin{equation}
\begin{array}{ccccc}
  {}^{\Gamma(B,A)}W^{(c_{(\mathrm{\RoM{1}},\delta)})}_{10}
  	& 
  	=
  	& 
  {}^{\Gamma(B,A)}W^{(c_{\delta\delta})}_{10} 
  	& 
\stackrel
  {(\ref{EqDrinfeldAssFactorizatonAffinePathParTransp})}
  {=}
  	 &
  e^{\lambda\ln(\delta)B}
  \Phi_{\delta,\delta}(A,B)
  e^{-\lambda\ln(\delta)A}, \\	
   {}^{\Gamma(B,A)}W^{(c_{(\mathrm{\RoM{3}},\delta)})}_{10}
   & 
   \stackrel
   {(\ref{EqHexZetaStarGamma})	
   }
   {=}
   & 
   {}^{\Gamma(\tilde{C},B)}W^{(c_{\delta\delta})}_{10}
   &
   \stackrel
   {(\ref{EqDrinfeldAssFactorizatonAffinePathParTransp})}
   {=}
   &
   e^{\lambda\ln(\delta)\tilde{C}}
   \Phi_{\delta,\delta}(B,\tilde{C}
   )
   e^{-\lambda\ln(\delta)B}, \\
   {}^{\Gamma(B,A)}W^{(c_{(\mathrm{\RoM{5}},\delta)})}_{10}
   & 
   \stackrel
   {
   	(\ref{EqHexZetaSquareStarGamma})	
   }
   {=}
   & 
   {}^{\Gamma(A,\tilde{C})}W^{(c_{\delta\delta})}_{10} 
   &
   \stackrel
   {(\ref{EqDrinfeldAssFactorizatonAffinePathParTransp})}
   {=}
   &
   e^{\lambda\ln(\delta)A}
   \Phi_{\delta,\delta}(\tilde{C},A
   )
   e^{-\lambda\ln(\delta)\tilde{C}}.
\end{array}
\label{EqHexParTranspCICIICV}
\end{equation}
For the even paths we can proceed in exactly the same way: the fact that
$c_{(\mathrm{\RoM{4}},\delta)}
=\zeta\circ c_{(\mathrm{\RoM{2}},\delta)}$ and
$c_{(\mathrm{\RoM{6}},\delta)}
=\zeta\circ \zeta\circ c_{(\mathrm{\RoM{1}},\delta)}$,
see (\ref{EqHexagonSixPaths}) we get --upon using
(\ref{EqFormConnPullBackParTransp})-- the following
formulas:
\begin{equation}\label{EqHexParTranspCIVCVIByCII}
  {}^{\Gamma(B,A)}W^{(c_{(\mathrm{\RoM{4}},\delta)})}_{10}
   =   {}^{\Gamma(\tilde{C},B)}
   W^{c_{(\mathrm{\RoM{2}},\delta)}}_{10} 
   ~~~~\mathrm{and}~~~~
   {}^{\Gamma(B,A)}W^{(c_{(\mathrm{\RoM{6}},\delta)})}_{10}
   =   {}^{\Gamma(A,\tilde{C})}
   W^{c_{(\mathrm{\RoM{2}},\delta)}}_{10} ,
\end{equation}
hence it suffices to compute
${}^{\Gamma(B,A)}W^{(c_{(\mathrm{\RoM{2}},\delta)})}_{10}$.
\begin{lemma}
 With the above definitions and notations: the parallel
 transport
${}^{\Gamma(B,A)}W^{(c_{(\mathrm{\RoM{2}},\delta)})}_{10}$
 along the path $c_{(\mathrm{\RoM{2}},\delta)}$ with respect to the connection $\Gamma(B,A)$ factorizes in the following way
 \begin{equation}\label{EqHexParTranspCII}
   {}^{\Gamma(B,A)}W^{(c_{(\mathrm{\RoM{2}},\delta)})}_{10}
   =
   e^{\lambda\mathbf{i}\pi B}H^{(\delta)}(B,A)
 \end{equation}
 where the element $\delta\mapsto H^{(\delta)}(B,A)$ of
 $\Big(\mathrm{Fun}\left(]0,1/4],\mathbb{C}\right)\otimes
   \mathcal{A}\Big)[[\lambda]]$ is a harmless group term,
   see (\ref{EqNormsLimDefRBH}),
   (\ref{EqFilNormsDefGroupsGRGBGH}), and Proposition
   \ref{PFilNormsRBHProperties}.
\end{lemma}
\begin{proof}
 We compute
 \begin{eqnarray}
   \Gamma(B,A)^{(c_{(\mathrm{\RoM{2}},\delta)})}(s)
   & = &   
   -\mathbf{i}\pi
        \frac
          {
          	\delta\left(1-\frac{\delta}{2}\right)
   	      e^{\mathbf{i}\pi s}
   	      	}{
   	      	\left(1-\frac{\delta}{2}\right)^2
   	      	-\frac{\delta^2}{4}
   	      	 e^{\mathbf{i}2\pi s}
   	      	 }
   	      	 A
   	      	 +
   	 \mathbf{i}\pi
   	   \frac
   	   {
   	   	1-\frac{\delta}{2}
   	   }{
   	   1-\frac{\delta}{2}
   	   +
   	   \frac{\delta}{2}e^{\mathbf{i}\pi s}
   	}
   	B  \nonumber\\
   	& = &
   \underbrace{\mathbf{i}\pi B}_{=: Y_0(s)}
   ~~\underbrace{-
   	\mathbf{i}	\pi\delta \left(
   \frac
   {
   	\frac{1}{2}e^{\mathbf{i}\pi s}
   }{
   1-\frac{\delta}{2}
   +
   \frac{\delta}{2}e^{\mathbf{i}\pi s}
}
B 
   +
   \frac
   {
   	\left(1-\frac{\delta}{2}\right)
   	e^{\mathbf{i}\pi s}
   }{
   \left(1-\frac{\delta}{2}\right)^2
   -\frac{\delta^2}{4}
   e^{\mathbf{i}2\pi s}
}
A
   \right)
}_{=: Z^{(\delta)}(s)}.\label{EqHexGammaBAOfCII}
 \end{eqnarray}
 We can now apply the factorization statement $iv)$ of Proposition \ref{PFormalLinODEPropagatorPropert}, see
(\ref{EqFormalLinODEFactorization}): clearly, the fundamental solution $U_{\cdot 0}(s)$ of the formal linear
ODE $\mathsf{d}U_{\cdot 0}/\mathsf{d}s
=\lambda Y_0U_{\cdot 0}$
is simply given by the formal exponential
$U_{\cdot 0}(s)=e^{\lambda\mathbf{i}\pi s B}$,
and the parallel transport
${}^{\Gamma(B,A)}
W^{(c_{(\mathrm{\RoM{2}},\delta)})}_{\cdot 0}$ thus factorizes thanks to (\ref{EqFormalLinODEFactorization}) as follows
for all $s\in[0,1]$
\[
 {}^{\Gamma(B,A)}
 W^{(c_{(\mathrm{\RoM{2}},\delta)})}_{s 0}  
 ~~=~~ e^{\lambda\mathbf{i}\pi s B}~H^{(\delta)}(B,A)(s)
\] 
where $H^{(\delta)}(B,A)(s)$ is a fundamental solution
for the formal linear ODE
\[
\dfrac{\mathsf{d}H^{(\delta)}(B,A)}{\mathsf{d}s}(s)
= \lambda e^{-\lambda\mathbf{i}\pi s B}
Z^{(\delta)}(s)e^{\lambda\mathbf{i}\pi s B}
H^{(\delta)}(B,A)(s)
=\lambda \underbrace{
	\left(e^{-\lambda\mathbf{i}\pi s\mathsf{ad}_B}\left(
Z^{(\delta)}(s)\right)\right)}_{=:\tilde{Z}^{(\delta)}(s)}
H^{(\delta)}(B,A)(s)
\]
with initial condition $H^{(\delta)}(B,A)(0)=1$.
We shall make the upper bound test (\ref{LEqFilNormsDefRBH})
for $\tilde{Z}^{(\delta)}(s)$, see Lemma 
\ref{LNormsLimIntegralBounds}: writing
$\tilde{Z}^{(\delta)}(s)=
\sum_{r=0}^\infty\lambda^r\tilde{Z}_r^{(\delta)}(s)$ we
get --upon using an arbitrary norm $||~||$ on
the complex vector space $\mathcal{A}$ (see $\S$\ref{SubSec Norms and limits} for more details)--
for each $r\in\mathbb{N}$ upon setting $d_{r0}\coloneqq 1$
if $r=0$ and $d_{r0}\coloneqq 0$ otherwise:
\begin{eqnarray*}
\left|\left|\tilde{Z}_r^{(\delta)}(s)\right|\right|
 & = & 
 \left|\left|\frac{\mathbf{i}^r\pi^rs^r}{r!}
    \mathsf{ad}_B^{\circ r}\left(Z^{(\delta)}(s)\right)
\right|\right| \nonumber \\
& \stackrel{(\ref{EqHexGammaBAOfCII})}{=}& 
\frac{\pi^{r+1}s^r}{r!}\delta\left|\left| 
\frac
{
d_{r0}~	\frac{1}{2}e^{\mathbf{i}\pi s}
}{
1-\frac{\delta}{2}
+
\frac{\delta}{2}e^{\mathbf{i}\pi s}
}
B 
+
\frac
{
	\left(1-\frac{\delta}{2}\right)
	e^{\mathbf{i}\pi s}
}{
\left(1-\frac{\delta}{2}\right)^2
-\frac{\delta^2}{4}
e^{\mathbf{i}2\pi s}
}
\mathsf{ad}_B^{\circ r}\left(A\right)
\right|\right|\nonumber \\
&\leqslant &
\underbrace{\frac{\pi^{r+1}}{r!}
 \left(||B||+
   2\left|\left|
    \mathsf{ad}_B^{\circ r}
    \left(A\right)\right|\right|\right)
    }_{=:\check{C}_r}\delta
\end{eqnarray*}
where we have used $|e^{\mathbf{i}\tau}|=1$ for each real number $\tau$, and the elementary lower bounds
\begin{eqnarray*}
  \left|1-\frac{\delta}{2}
  +
  \frac{\delta}{2}e^{\mathbf{i}\pi s}\right|
 & \geqslant & 1-\frac{\delta}{2}-\left|\frac{\delta}{2}
  e^{\mathbf{i}\pi s}\right| =1-\delta \geqslant \frac{1}{2}, \\
\left| \left(1-\frac{\delta}{2}\right)^2
 -\frac{\delta^2}{4}
 e^{\mathbf{i}2\pi s} \right|
 &\geqslant &\left(1-\frac{\delta}{2}\right)^2-
 \left|\frac{\delta^2}{4}
 e^{\mathbf{i}2\pi s} \right| = 1-\delta\geqslant \frac{1}{2}.
\end{eqnarray*}
for the denominators. It follows
that $\delta\mapsto H^{(\delta)}(B,A)$ is a harmless
group term.
\end{proof}
We now need to put the loop $c_\delta$ in a star-shaped
open set $U'$ of $\mathbb{C}^{\times\times}$ because
$\mathbb{C}^{\times\times}$ is NOT star-shaped
(it is not even simply connected). Define
\begin{equation}
\label{EqHexUPrimeStarshaped}
  U'\coloneqq
  \mathbb{C}\setminus \left.\left\{
    \frac{1-\mathbf{i}}{2}+
      t(\gamma+\mathbf{i})\in\mathbb{C}~\right|~
        t\in\mathbb{R},~t\geqslant \frac{1}{2},~
        \gamma\in\{-1,1\}\right\}
\end{equation}
which is the complex plane minus two closed half-lines
emanating from $0$ with slope $-1$ and from $1$ with slope
$1$, both in the direction of non-negative imaginary part.
Clearly, $0$ and $1$ do not belong to $U'$ whence
$U'\subset\mathbb{C}^{\times\times}$ is an open subset.
By elementary linear algebra it is shown that every
point $z$ of the plane $\mathbb{C}$ lies on a straight 
 (real) line
passing through the point $q=(1-\mathbf{i})/2$. If $z$
does not lie on one of the two straight lines having slope
$1$ or $-1$, it is clear by the definition of $U'$
that the unique straight line passing through $q$ and $z$ belongs to $U'$. On the other hand, if a point of $z\in U'$ is on one of the straight lines with slope $1$ or $-1$, its imaginary part is necessarily strictly negative. Clearly, the so-called lower half plane $H^-$ (i.e.~the subset of all points having
strictly negative imaginary part) belongs to $U'$, contains
$q$, and is obviously convex, i.e.~the line segment joining two distinct points $z_1,z_2\in H^-$ is again contained in $H^-$. It follows from the preceding discussion that $U'$ is star-shaped around $q=\frac{1-\mathbf{i}}{2}$ and $\forall \delta \in ]0, 1/4]$ we have $c_\delta(s)\in U'$.
We now have all the necessary information to prove the following
\begin{theorem}[Hexagon Equation]
	Let $\mathcal{A}$ be an arbitrary complex associative unital algebra, and let $A,B,C\in\mathcal{A}$ be three
	elements satisfying
	(\ref{EqHexAPlusBPlusCCentral}).
	Then the Hexagon equation (\ref{EqHexDefHexagonEqn}) for the Drinfel'd associator holds.
\end{theorem}
\begin{proof}
	Since the composed loop $c_\delta$, see
	(\ref{EqHexComposedLoop}), is contained in the
	star-shaped open subset $U'$,
	see (\ref{EqHexUPrimeStarshaped}), we can apply
    Corollary \ref{CFormFlatContractibleLoops} to conclude that the parallel transport around the loop $c_\delta$
    along the flat connection $\Gamma(B,A)$, see
    (\ref{EqHexDefConn}), is equal to $1$. Abbreviating
    ${}^{\Gamma(B,A)}W_{10}^{(c)}$ by $W^{(c)}_{10}$ for
    any piecewise smooth path $c:[0,1]\to U'$ we get
    \begin{eqnarray}
    1=W^{(c_\delta)}_{10} 
    &
  \stackrel{(\ref{EqFormConnComposOfPathsParTransp})}{=}
     &
      W^{(c_{(\mathrm{\RoM{6}},\delta)})}_{10}
      W^{(c_{(\mathrm{\RoM{5}},\delta)})}_{10}
      W^{(c_{(\mathrm{\RoM{4}},\delta)})}_{10}
       W^{(c_{(\mathrm{\RoM{3}},\delta)})}_{10}
       W^{(c_{(\mathrm{\RoM{2}},\delta)})}_{10}
       W^{(c_{(\mathrm{\RoM{1}},\delta)})}_{10}
       \nonumber\\
     &
  \stackrel{
  	   (\ref{EqHexParTranspCIVCVIByCII}),
  	   (\ref{EqHexParTranspCII}),
  	   (\ref{EqHexParTranspCICIICV})
  	   }{=}
  	 &    e^{\lambda\mathbf{i}\pi A}~
  	      H^{(\delta)}(A,\tilde{C})~~~
  	      e^{\lambda\ln(\delta) A}~
  	      \Phi_{\delta,\delta}(\tilde{C},A)~
  	      e^{-\lambda\ln(\delta) \tilde{C}}~
  	      \nonumber\\
  	 & &  e^{\lambda\mathbf{i}\pi \tilde{C}}~
  	 H^{(\delta)}(\tilde{C},B)~~~
  	 e^{\lambda\ln(\delta) \tilde{C}}~
  	 \Phi_{\delta,\delta}(B,\tilde{C})~
  	 e^{-\lambda\ln(\delta) B}~\nonumber\\
  	 & & e^{\lambda\mathbf{i}\pi B}~
  	 H^{(\delta)}(B,A)~~~
  	 e^{\lambda\ln(\delta) B}~
  	 \Phi_{\delta,\delta}(A,B)~
  	 e^{-\lambda\ln(\delta) A}.
  	  \label{EqHexHexEqFirstIdentity}
     \end{eqnarray}
     Clearly, the three singular terms 
     $\delta\mapsto e^{\lambda\ln(\delta) A}$, 
     $\delta\mapsto e^{\lambda\ln(\delta) B}$, and 
     $\delta\mapsto e^{\lambda\ln(\delta) \tilde{C}}$
     belong to the at most logarithmically diverging
     group terms, $\mathcal{G}_{\mathcal{L}}$, see  (\ref{EqFilNormsDefGroupsGRGBGH}) for
     $J=]0,1/4]$. Hence, the
 following three conjugations again
define harmless group terms according to statement $iv.)$
of Proposition \ref{PFilNormsRBHProperties}:
\[
\begin{array}{ccc}
\tilde{H}^{(\delta)}(A,\tilde{C}) &\coloneqq &
e^{-\lambda\ln(\delta) A}~
  H^{(\delta)}(A,\tilde{C})~e^{\lambda\ln(\delta) A},\\
 \tilde{H}^{(\delta)}(\tilde{C},B) &\coloneqq &
 e^{-\lambda\ln(\delta) \tilde{C}}~
 H^{(\delta)}(\tilde{C},B)
    ~e^{\lambda\ln(\delta) \tilde{C}},\\
 \tilde{H}^{(\delta)}(B,A) &\coloneqq &
 e^{-\lambda\ln(\delta) B}~
 H^{(\delta)}(B,A)~e^{\lambda\ln(\delta) B}. 
 \end{array} 
\]
Rewriting (\ref{EqHexHexEqFirstIdentity}) by means of
these harmless group terms we see that the three singular
terms mentioned above, $e^{\lambda\ln(\delta) A}$, 
$e^{\lambda\ln(\delta) B}$, and 
$e^{\lambda\ln(\delta) \tilde{C}}$, cancel out, and we are left with the following identity:
\[
1=e^{\lambda\mathbf{i}\pi A}~
  \tilde{H}^{(\delta)}(A,\tilde{C})~
  \Phi_{\delta,\delta}(\tilde{C},A)~
e^{\lambda\mathbf{i}\pi \tilde{C}}~
  \tilde{H}^{(\delta)}(\tilde{C},B)~
  \Phi_{\delta,\delta}(B,\tilde{C})~
e^{\lambda\mathbf{i}\pi B}~
\tilde{H}^{(\delta)}(B,A)~
\Phi_{\delta,\delta}(A,B).
\]
Passing to the limit $\delta\to 0$ we get --thanks to the limit rules (\ref{EqNormsLimLimitProdAreProdOfLimits}),
the definition of the Drinfel'd associator
(\ref{EqDrinfeldAssDefAss}), and the fact that harmless
group terms tend to $1$ for $\delta\to 0$ (see
statement $v.)$ of Proposition \ref{PFilNormsRBHProperties})-- 
the following equation (recall that $\tilde{C}=C-\Lambda$)
\[
1=e^{\lambda \pi\mathbf{i}A}\Phi(C-\Lambda,A)
e^{\lambda \pi\mathbf{i}(C-\Lambda)}\Phi(B,C-\Lambda)
e^{\lambda \pi\mathbf{i}B}\Phi(A,B).
\]
This equation immediately results in the Hexagon equation
(\ref{EqHexDefHexagonEqn}) thanks to the fact that
$\Lambda$ commutes with $A,B$ and $C$ whence
$\Phi(C-\Lambda,A)=\Phi(C,A)$, 
$\Phi(B,C-\Lambda)=\Phi(B,C
)$ by (\ref{EqDrinfeldAssIndepOnCentralTerms}).
\end{proof}

\subsection{The Pentagon Equation}  
\label{SubSecPentagonEquation}

Let $\mathcal{A}$ be an arbitrary complex unital associative algebra. Let 
$A_{12}=A_{21},A_{13}=A_{31},A_{14}=A_{41},A_{23}=A_{32},
A_{24}=A_{42},A_{34}=A_{43}\in\mathcal{A}$ six elements satisfying the infinitesimal braid relations
(\ref{EqIntroIB-definition-two})
and (\ref{EqIntroIB-definition-three}). 
We wish to prove the \textbf{Pentagon Equation for the Drinfel'd associator}, see (\ref{EqPentDefPentagonEquation}).

The first naive observations are the following: the element
$A_{14}$ does not occur in the equation, and there is an
obvious symmetry by passing to the inverse and sending
$(1,2,3,4)$ to $(4,3,2,1)$ where (\ref{EqDrinfeldAssInverseEqAGoesToB}) is used.
Moreover, the right associator on the right hand side only depends on the numbers $1,2,3$ and
the left associator on the right hand side only depends on the numbers $2,3,4$. Finally, the element $A_{23}$ does not occur in the middle factor on the right hand side.\\
We shall try to represent each side of the Pentagon Equation (\ref{EqPentDefPentagonEquation}) by the parallel transport along the composition of three paths, 
$c_{(\mathrm{\RoM{3}},\delta)}*
(c_{(\mathrm{\RoM{2}},\delta)}*
c_{(\mathrm{\RoM{1}},\delta)})$, for the right hand side,
and along the composition of two paths 
$c_{(\mathrm{\RoM{5}},\delta)}*
c_{(\mathrm{\RoM{4}},\delta)}$ for the left hand side,
both having the same initial and final points.
Since the numbers $1,2,3,4$ occur it seems plausible to use the open set $U\subset \mathbb{R}^4$,
\begin{equation*}
    U\coloneqq \big\{ x\in\mathbb{R}^4~\big|~
      x_1<x_2<x_3<x_4\big\},
\end{equation*}
which can be interpreted as the space of all ordered
configurations of four particles on the real line, see
\cite[p. 834, line 2]{Dri90}. Clearly, $U$ is a subset
of the configuration space $Y_4$. Following Drinfel'd
we shall use the pull-back of the $\RuK \Run \RuZ \Rua$-connection
${}^{(4)}{\Gamma_{\RuK \Run \RuZ \Rua}}$  
from $Y_4$ to $U$ (with respect to the canonical injection $U\to Y_4$) which is still
flat, see Theorem \ref{TKnZhConnectionIsFlat} and
Proposition \ref{PFlatConnPullBackOfFlatIsFlat}. Since
the $\RuK \Run \RuZ \Rua$-connections are invariant under
simultaneous translations, it is not unreasonable to assume that the first
coordinate $x_1$ of all the paths is fixed to be $0$. Next, the fact that
$A_{14}$ does not occur in the Pentagon equation may lead
us to the ansatz that the difference $x_4-x_1$ should remain constant; on the other hand the fact that there are terms in the Pentagon Equation not containing $1$ and not containing $4$ suggests that $x_4$ should be
`far away from $x_1=0$', hence we set $x_1=0$ and $x_4=1$.
 Define the open full triangle
\begin{equation*}
   U'\coloneqq \big\{(x_2,x_3)\in
   \mathbb{R}^2~\big|~0<x_2<x_3<1\big\}.
\end{equation*}
Note that $U'$ is invariant under the involutive
diffeomorphism
\begin{equation*}
  \Theta: U'\to U':
  (x_2,x_3)\mapsto \left(1-x_3,
              1-x_2\right).
\end{equation*}
Moreover, it is easy to see that $U'$
is convex, i.e.~that for each pair of elements
$(x_2,x_3)$ and $(y_2,y_3)$ the line segment
$s\mapsto \big((1-s)x_2+sy_2,(1-s)x_3+sy_3\big)$
is contained in $U'$. Moreover, the point
$\left(\frac{1}{2},\frac{2}{3}\right)$ is contained in $U'$ 
whence \textbf{$U'$ is star-shaped around}
$\left(\frac{1}{2},\frac{2}{3}\right)$.\\
Using the injection $i:U'\to U\to Y_4:
(x_2,x_3)\mapsto (0,x_2,x_3,1)$ we can
pull back the $\RuK \Run \RuZ \Rua$-connection
${}^{(4)}{\Gamma_{\RuK \Run \RuZ \Rua}}$ to $U'$.
Writing $(A_{ij})$ for 
$(A_{12},A_{13},A_{14},A_{23},A_{24},A_{34})$
an easy computation gives
\begin{eqnarray}
\Gamma(x_2,x_3) \coloneqq
 \Gamma\big((A_{ij})\big)(x_2,x_3)&\coloneqq&
  \left( i^*\left({}^{(4)}{\Gamma_{\RuK \Run \RuZ \Rua}}\right)
  \right)(x_2,x_3)\nonumber \\
 & = &
 \left(\frac{1}{x_2}A_{12}
       +\frac{1}{x_2-x_3}A_{23}
       +\frac{1}{x_2-1}A_{24}\right)\mathsf{d}x_2
          \nonumber \\
   & &  +
   \left(\frac{1}{x_3}A_{13}
   -\frac{1}{x_2-x_3}A_{23}
   +\frac{1}{x_3-1}A_{34}\right)\mathsf{d}x_3. 
   \label{EqPentEqGammaEqualsPullBackKnZhConn}
\end{eqnarray}
Clearly, $\Gamma$ is (formally) flat according to
Theorem \ref{TKnZhConnectionIsFlat} and Proposition
\ref{PFlatConnPullBackOfFlatIsFlat}, but its 
(formal) flatness can
easily be computed directly from formula
(\ref{EqPentEqGammaEqualsPullBackKnZhConn}). Moreover, it is easy to compute that
\begin{equation}\label{EqPentEqThetaPullBackOfGamma}
  \Theta^*\big(\Gamma\big((A_{ij})\big)\big) = 
     \Gamma\big((A_{\sigma(i)\sigma(j)})\big)
     ~~~\mathrm{with}~~~
     \sigma: (1,2,3,4)\mapsto (4,3,2,1).
\end{equation}
Next, we would like to substantiate in terms of paths the
five `zones' which Drinfel'd mentiones in his articles,
see cf.~\cite[p.1454]{Dri89} or
\cite[p.834, line 3,4]{Dri90} (where the fifth zone in
\cite{Dri90} has been forgotten in the English translation, see the
original article in Russian language, p.154, paragraph before the Lemma, for a complete description):
here certain pairs of coordinates are `very close' to each others, others are `medium close' and still others are
`far' which is expressed in terms of inequalities using the symbol $\ll$: using the real number $\delta\in~]0,1/4]$ 
--which is meant to be sent to zero-- we interpret --as a
rule of thumb-- `very close' as $\approx\delta^2$,
`medium close' as around $\approx \delta$, and `far' as 
$\approx 1$. 
Inspired by the picture \cite[p.478, Fig.8.2.]{Kas95}
we first use the following subdivision of the interval $]0,1[$,
in which we 
imagine that both $x_2$ and $x_3$ `move' between the selected positions,
\begin{equation*}
0~<~ \delta^2 ~<~ \delta-\delta^2 ~<~ \delta ~<~
1-\delta ~<~ 1-\delta+\delta^2
~<~1-\delta^2~<~ 1,
\end{equation*}
 and --being fully aware of the nonuniqueness of our choice-- associate the following five points in $U'$ (as part of the $(x_2,x_3)$ plane) as an interpretation of the five zones (recall that $x_1=0$
 and $x_4=1$):
\begin{equation}
  \begin{array}{ccccccccc}
 \mathrm{zone}~1: & "x_2-x_1 &\ll& x_3-x_1& \ll& x_4-x_1" &
          \mathrm{interpreted~as}&  (\delta^2,\delta)
             & =: p_1,\\
 \mathrm{zone}~2: & "x_3-x_2 &\ll& x_3-x_1& \ll& x_4-x_1" &
 \mathrm{interpreted~as}&  (\delta-\delta^2,\delta)
          & =: p_2,\\ 
 \mathrm{zone}~3: & "x_3-x_2 &\ll& x_4-x_2& \ll& x_4-x_1" &
 \mathrm{interpreted~as}&  
 \left(1-\delta,1-\delta+\delta^2 \right)
       & =: p_3 ,\\ 
  \mathrm{zone}~4: & "x_4-x_3 &\ll& x_4-x_2& \ll
    & x_4-x_1" &
  \mathrm{interpreted~as}& 
    \left(1-\delta,1-\delta^2 \right)
        & =: p_4 ,\\ 
  \mathrm{zone}~5: & "x_2-x_1 
  & \ll &  & \ll &  x_4-x_1 & \mathrm{and} &  \\
  & x_4-x_3 & \ll & 
  & \ll& x_4-x_1" & 
  \mathrm{interpreted~as}&  \left(\delta^2,
  1-\delta^2\right) & =: p_5.      
  \end{array} \label{EqPentEqDrinfeldZones}
\end{equation}
There is thus the following simple ansatz for the following five 
affine paths subsequently joining the above five points by the unique line segments between them, see
(\ref{EqFormConnDefAffinePath}), where $\iota:[0,1]\to
[0,1]$ denotes the usual interval inversion $s\mapsto 1-s$:
\begin{equation}\label{EqPentEqAffinePathsOfPentagon}
 \begin{array}{ccc}
 c_{(\mathrm{\RoM{1}},\delta)} \coloneqq 
   c_{p_2\leftarrow p_1}, &
 c_{(\mathrm{\RoM{2}},\delta)} \coloneqq
 c_{p_3\leftarrow p_2}=\Theta\circ c_{(\mathrm{\RoM{2}},\delta)}\circ \iota,&
 c_{(\mathrm{\RoM{3}},\delta)}\coloneqq 
 c_{p_4\leftarrow p_3}
 =\Theta\circ c_{(\mathrm{\RoM{1}},\delta)}\circ \iota,
     \\
 c_{(\mathrm{\RoM{4}},\delta)}\coloneqq 
 c_{p_5\leftarrow p_1},~~&
 c_{(\mathrm{\RoM{5}},\delta)}\coloneqq 
 c_{p_4\leftarrow p_5}
 =\Theta\circ c_{(\mathrm{\RoM{4}},\delta)}\circ \iota. &
 \end{array}
\end{equation}
which can be depicted in Figure $\ref{fig:Pentagon}$ describing
a non-regular pentagon whose vertices are the five `zone' points (\ref{EqPentEqDrinfeldZones}) and whose edges
are the images of the five affine paths (\ref{EqPentEqAffinePathsOfPentagon}).

\begin{figure}[h!!]
	\centering
	\begin{tikzpicture}
	\draw[->] (-0.5, 0) -- (8.5, 0) node[right] {$x_2$};
	\draw[->] (0, -0.5) -- (0, 8.5) node[above] {$x_3$};
	\draw[dashed] (0,0)--(7.7,7.7);
	\draw[dashed] (0,7.7)--(7.7,7.7);
	\draw[blue,->] (0.25,7.25)--(3.25,7.25);
	\draw[blue](3.25,7.25)--(5.25,7.25);
	\draw[blue] (0.25,7.25)--(0.25,5.25); 
	\draw[blue,->] (0.25,2.25)--(0.25,5.25); 
	\draw[red,->] (0.25,2.25)--(1,2.25); 
	\draw[red](1,2.25)--(1.75,2.25);
	\draw[red,->] (5.25,5.75)--(5.25,6.5);
	\draw[red] (5.25,6.5)--(5.25,7.25);
	\draw[red] (5.25,5.75)--(3.75,4.25);
	\draw[red,->] (1.75,2.25)--(3.75,4.25);
	\draw (0.25,0.1)--(0.25,-0.1);
	\node at (0.25,-0.4) {\footnotesize $\delta^2$};
	\draw (5.25,0.1)--(5.25,-0.1);
	\node at (5.25,-0.4) {\footnotesize $1-\delta$};
	\draw (1.75,0.1)--(1.75,-0.1);
	\node at (1.75,-0.4) { \footnotesize $\delta - \delta^2$};
	\draw (7.7,0.1)--(7.7,-0.1);
	\node at (7.7,-0.4) {\footnotesize $1$};
	\node at (-0.4,7.7) {\footnotesize $1$};
	\draw (-0.1,7.25)--(0.1,7.25);
	\node at (-0.55,7.25) {\footnotesize $1-\delta^2$};
	\draw (-0.1,5.75)--(0.1,5.75);
	\node at (-0.85,5.75) {\footnotesize $1-\delta+ \delta^2$};
	\draw (-0.1,2.25)--(0.1,2.25);
	\node at (-0.4,2.25) {\footnotesize $\delta$};
	\node at (0.7,5) {\footnotesize $c_{(\mathrm{\RoM{4}}, \delta)}$};
	\node at (3.1,6.8) {\footnotesize $c_{(\mathrm{\RoM{5},\delta})}$};
	\node at (1.2,2.5) {\footnotesize $c_{(\mathrm{\RoM{1}},\delta)}$};
	\node at (4.8,6.3) {\footnotesize $c_{(\mathrm{\RoM{3}},\delta)}$};
	\node at (3.6,4.7) {\footnotesize $c_{(\mathrm{\RoM{2}},\delta)}$};
	\end{tikzpicture}
	\caption{The paths (\ref{EqPentEqAffinePathsOfPentagon}) in the 
		$x_2$-$x_3$-plane } \label{fig:Pentagon}
\end{figure}
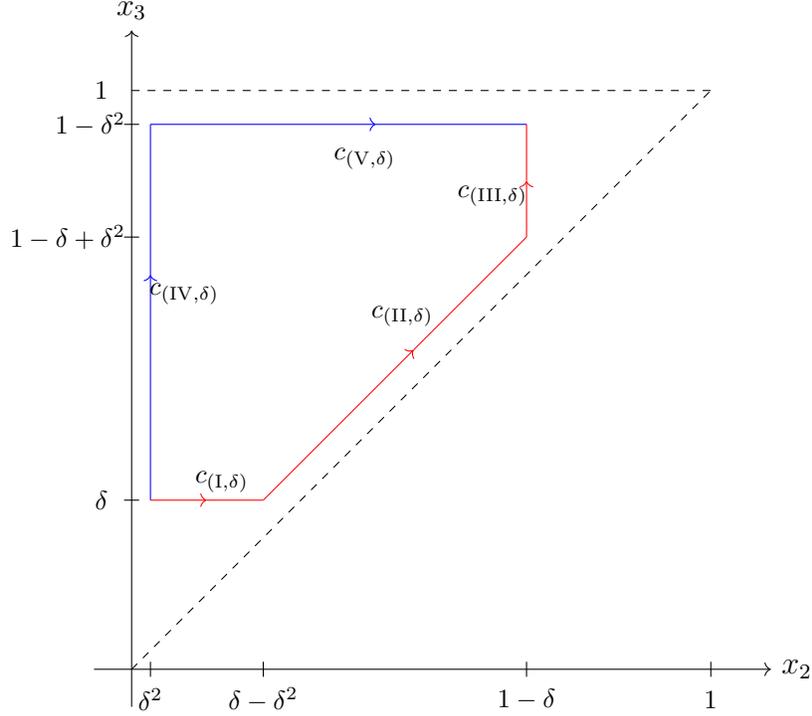
Note that the symmetry $\Theta$ --which is a symmetry of the pentagon-- is the reflection with respect to the straight line whose equation is 
$x_3=1-x_2$. Note furthermore that the
second path is the only path where both coordinates
$x_2,x_3$ are moving, but their distance is kept constant
(inspired by the observation of the absence of $A_{23}$
in the middle factor on the right hand side of the Pentagon Equation (\ref{EqPentDefPentagonEquation})).

We are now going to compute the parallel transports
along the five paths of (\ref{EqPentEqAffinePathsOfPentagon}) with respect to
 the connection $\Gamma$, see
  (\ref{EqPentEqGammaEqualsPullBackKnZhConn}). Thanks to (\ref{EqPentEqThetaPullBackOfGamma}),
  (\ref{EqFormConnInversionFormula}), and 
  (\ref{EqFormConnPullBackParTransp}) it suffices to compute the parallel transports $W_{10}^{(c_{(\mathrm{\RoM{1}},\delta)})}$,
  $W_{10}^{(c_{(\mathrm{\RoM{2}},\delta)})}$, and
  $W_{10}^{(c_{(\mathrm{\RoM{4}},\delta)})}$ --we henceforth suppress the symbol $\Gamma$ attached to
  $W$--, the other
  two will be immediate from the above identities for
  symmetries and reparametrizations of parallel 
  transports.\\
  We get the following Lemma:
  \begin{lemma}
  	\label{TPentEqFiveParTrDeltaPositive}
  	With the above notations, we get the following results 
  	for the five parallel transports for all 
  	$\delta\in \mathbf{J}=J=~]0,1/4]$:
  	\begin{equation}\label{EqPentEqFiveParTrDeltaPositive}
  	\begin{array}{ccl}
  	W_{10}^{(c_{(\mathrm{\RoM{1}},\delta)})}
  	 & = &
  	 e^{\lambda\ln(\delta)A_{23}}
  	 ~\psi_\delta(A_{23},A_{12})
  	 ~H^{(\mathrm{\RoM{1}},\delta)}
  	 ~\psi_\delta(A_{12},A_{23})^{-1}
  	 ~e^{-\lambda\ln(\delta)A_{12}},\\
  	 W_{10}^{(c_{(\mathrm{\RoM{2}},\delta)})}
  	 & = &
  	 e^{\lambda\ln(\delta)(A_{24}+A_{34})}
  	 ~\psi_\delta(A_{24}+A_{34},A_{12}+A_{23})
  	 ~H^{(\mathrm{\RoM{2}},\delta)}
  	 ~\psi_\delta(A_{12}+A_{13},A_{24}+A_{34})^{-1}\\
  	 & & ~e^{-\lambda\ln(\delta)(A_{12}+A_{13})},\\
  	 W_{10}^{(c_{(\mathrm{\RoM{3}},\delta)})}
  	 & = &
  	 e^{\lambda\ln(\delta)A_{34}}
  	 ~\psi_\delta(A_{34},A_{23})
  	 ~H^{(\mathrm{\RoM{3}},\delta)}
  	 ~\psi_\delta(A_{23},A_{34})^{-1}
  	 ~e^{-\lambda\ln(\delta)A_{23}},\\
  	 W_{10}^{(c_{(\mathrm{\RoM{4}},\delta)})}
  	 & = &
  	 e^{\lambda\ln(\delta^2)A_{34}}
  	 ~\psi_{\delta^2}(A_{34},A_{13}+A_{23})
  	 ~H^{(\mathrm{\RoM{4}},\delta)}
  	 ~\psi_\delta(A_{13}+A_{23},A_{34})^{-1}
  	 ~e^{-\lambda\ln(\delta)(A_{13}+A_{23})},\\
  	 W_{10}^{(c_{(\mathrm{\RoM{5}},\delta)})}
  	 & = &
  	 e^{\lambda\ln(\delta)(A_{23}+A_{24})}
  	 ~\psi_{\delta}(A_{23}+A_{24},A_{12})
  	 ~H^{(\mathrm{\RoM{5}},\delta)}
  	 ~\psi_{\delta^2}(A_{12},A_{23}+A_{24})^{-1}
  	 ~e^{-\lambda\ln(\delta^2)A_{12}}.
  	 \end{array}
  	\end{equation}
  	where the $\psi$-terms are defined in (\ref{EqDrinfeldAssFactorizatonHalfParTransp}), see also (\ref{EqDrinfeldAssLimitOfLittlePhis}) and
  	(\ref{EqDrinfeldAssDefAssDeltaEpsilon}), and the terms $\delta\mapsto H^{(i,\delta)}$
  	for $i=
  	\mathrm{\RoM{1}},\mathrm{\RoM{2}},\mathrm{\RoM{3}},
  	\mathrm{\RoM{4}},\mathrm{\RoM{5}}$ are harmless
  	group terms, see (\ref{EqNormsLimDefRBH}),
  	(\ref{EqFilNormsDefGroupsGRGBGH}), and Proposition
  	\ref{PFilNormsRBHProperties}. Note that the appearance of
  	$\delta^2$ in the terms $W_{10}^{(c_{(\mathrm{\RoM{4}},\delta)})}$
  	and $W_{10}^{(c_{(\mathrm{\RoM{5}},\delta)})}$ is crucial for
  	Theorem \ref{TPentEqFinalProofOfPentEq}.
  \end{lemma}
  \begin{proof}
  	In order to compute the parallel transports we shall
  	again use the half exponential paths already used in $\S$\ref{SubSecDrinfeldAssociatorDefinitionAndProperties},
  	of the type (\ref{EqDrinfeldAssExponentialPaths}).
  		More precisely, for each of the three paths
  	$c_{(\mathrm{\RoM{1}},\delta)},
  	c_{(\mathrm{\RoM{2}},\delta)},
  	c_{(\mathrm{\RoM{4}},\delta)}$ --recall that the parallel transport along
  	$c_{(\mathrm{\RoM{3}},\delta)}$ and $c_{(\mathrm{\RoM{5}},\delta)}$ can be computed using the symmetry indicated in (\ref{EqPentEqAffinePathsOfPentagon})-- we choose a
  	mid-point on the corresponding line segment in $U'$, and we replace each affine path $c_{(i,\delta)}$ by a
  	composition of two exponential half paths tracing the same line segment, $\tilde{c}_{(i,2,\delta)}*
  	\big(\check{c}_{(i,1,\delta)}\circ \iota\big)$
  	where $i=
  	\mathrm{\RoM{1}},\mathrm{\RoM{2}},\mathrm{\RoM{3}},
  	\mathrm{\RoM{4}},\mathrm{\RoM{5}}$, $\iota:[0,1]\to [0,1]:s\mapsto 1-s$ is the usual
  	interval inversion, $\check{c}_{(i,1,\delta)}$ joins the
  	midpoint to the initial point, and  $\tilde{c}_{(i,2,\delta)}$ joins the midpoint to the final point. Since $c_{(i,\delta)}$ and $\tilde{c}_{(i,2,\delta)}*
  	\big(\check{c}_{(i,1,\delta)}\circ \iota\big)$
  	just differ by a reparametrization they induce the same parallel transport, see (\ref{EqFormConnWReparamInv}), which implies
  	\begin{equation}\label{EqPentEqParTrAffinePEqProdExpP}
  	W^{(c_{(i,\delta)})}_{10}=
  	 W^{(\tilde{c}_{(i,2,\delta)})}_{10}
  	 \left(W^{(\check{c}_{(i,1,\delta)})}_{10}\right)^{-1}
  	\end{equation}
  	where the inversion fomula (\ref{EqFormConnInversionFormula}) has been used.
  	Hence, we choose the following paths
  	(where $\hat{\delta}\coloneqq \delta(1-\delta/2)$)
  	\begin{equation}
  	\begin{array}{ccccccc}
  	\check{c}_{(\mathrm{\RoM{1}},1,\delta)}(s) & \coloneqq
  	&\left(\frac{\delta}{2}e^{\ln(2\delta)s},\delta\right) & \mathrm{joining}&
  	\left(\frac{\delta}{2},\delta\right) &\to &\left(\delta^2,\delta\right), \\
  	\tilde{c}_{(\mathrm{\RoM{1}},2,\delta)}(s) & \coloneqq
  	&\left(\delta-\frac{\delta}{2}e^{\ln(2\delta)s},
  	 \delta\right) & \mathrm{joining}&
  	\left(\frac{\delta}{2},\delta\right) &\to& 
  	  \left(\delta-\delta^2, \delta\right),\\
  	\check{c}_{(\mathrm{\RoM{2}},1,\delta)}(s) & \coloneqq
    &    
\Theta\big(\tilde{c}_{(\mathrm{\RoM{2}},2,\delta)}(s)\big)&
\mathrm{joining}&
\left(\frac{1-\delta^2}{2},
  \frac{1+\delta^2}{2}\right) &\to& (\delta-\delta^2, \delta),\\
  \tilde{c}_{(\mathrm{\RoM{2}},2,\delta)}(s) & \coloneqq
  & \left(1-
  \frac{1}{2}e^{\ln(2\hat{\delta})s}\right)(1,1)
     +(-\frac{\delta^2}{2},\frac{\delta^2}{2})&
\mathrm{joining}&
\left(\frac{1-\delta^2}{2},
\frac{1+\delta^2}{2}\right) &\to&
\left(1-\delta, 1-\delta+\delta^2\right), \\
  \check{c}_{(\mathrm{\RoM{4}},1,\delta)}(s) & \coloneqq
  &  \left(\delta^2, \frac{1}{2}
                     e^{\ln(2\delta)s}\right) &
\mathrm{joining}& \left(\delta^2,\frac{1}{2}\right)&
\to &  (\delta^2,\delta), \\ 
  \tilde{c}_{(\mathrm{\RoM{4}},2,\delta)}(s) & \coloneqq
  &  \left(\delta^2, 1
  -\frac{1}{2}e^{\ln(2\delta^2)s}\right) &
  \mathrm{joining}& \left(\delta^2,\frac{1}{2}\right)&
  \to &  \left(\delta^2,1-\delta^2\right).              
  	\end{array}
  	\end{equation}
 When we compute $\Gamma^{(c_{(i,u,\delta)})}(s)$
 (for $i=
 \mathrm{\RoM{1}},\mathrm{\RoM{2}},\mathrm{\RoM{3}},
 \mathrm{\RoM{4}},\mathrm{\RoM{5}}$ and $u=1,2$)
 we shall see further down that we always get the following form --writing $d$ for the paths
 $\check{c}_{(i,1,\delta)}$ or
 $\tilde{c}_{(i,2,\delta)}$:
 \begin{equation}\label{EqPentEqGeneralFormOfGammaOfPath}
  \Gamma^{(d)}(s)=
  \underbrace{\ln(2\epsilon(\delta))B +
  \frac{-\ln(2\epsilon(\delta))}
    {2e^{-\ln(2\epsilon(\delta))s}-1}A
}_{=:Y_0(B,A)(s,\epsilon)}
  +
  \underbrace{
  	\sum_{1\leqslant i<j\leqslant 4}f_{ij}(s,\delta)A_{ij}
  	}_{=:Z(s,\delta)},
 \end{equation}
 --which is well-known from $\S$\ref{SubSecDrinfeldAssociatorDefinitionAndProperties}, (\ref{EqDrinfeldAssGammaBAComp})-- where $\epsilon:
 J\to J$
 is a monomial of $\delta$ of the form
 $\epsilon(\delta)=\delta^\ell$ where $\ell$ is a positive integer, $A,B$ are certain linear combinations
 of the algebra elements $A_{ij}$, and $f_{ij}$ are real-valued functions of $s\in[0,1]$ and 
 $\delta\in J\coloneqq ~]0,1/4]$. From the general factorization statement (\ref{EqFormalLinODEFactorization}) and the solution (\ref{EqDrinfeldAssFactorizatonHalfParTransp})
 of Lemma \ref{LDrinfeldAssFactorizationHalfAss} --which lead to associators-- we have the factorization
 \begin{equation}\label{EqPentEqFactorOfParTrHalfExpPath}
    W^{(c_{(i,u,\delta)})}_{s0}~=~
    \underbrace{
    e^{\lambda\ln(\epsilon)sB}~\psi_\epsilon(B,A)(s)
      }_{=:U_{s0}^{(\epsilon)} }~
     H^{(i,u,\delta)}_{s0},
 \end{equation}
 where $s\mapsto H^{(i,u,\delta)}_{s0}$ is a fundamental solution to the formal linear ODE
 \begin{equation*}
  \dfrac{\mathsf{d} H^{(i,u,\delta)}_{s0}}{\mathsf{d}s}
    ~=~\lambda {U_{s0}^{(\epsilon)}}^{-1}~Z(s,\delta)
    ~U_{s0}^{(\epsilon)}
    ~H^{(i,u,\delta)}_{s0}
    ~=:~\lambda \tilde{Z}(s,\delta)~H^{(i,u,\delta)}_{s0}.
 \end{equation*} 
 We shall show later that all these factors $\delta\mapsto H^{(i,u,\delta)}_{10}$ are harmless group terms: first,
 we shall prove that it suffices to
 show that for each $1\leqslant i<j\leqslant 4$ there
 are non-negative real numbers $C_{ij}$ and $\beta_{ij}>0$
 such that
 \begin{equation}\label{EqPentEqfijUpperBound}
  \forall~s\in[0,1],~\forall~\delta\in J:~~
   |f_{ij}(s,\delta)|\leqslant C_{ij}\delta^{\beta_{ij}}.
 \end{equation}
 Indeed, we check that the preceding condition (\ref{EqPentEqfijUpperBound}) implies the estimate of type $(H)$, (\ref{LEqFilNormsDefRBH}):
 write $\psi_\epsilon(B,A)(s)
 =\sum_{r=0}^\infty\psi_r(s,\epsilon)\lambda^r$ and its inverse
 $\psi_\epsilon(B,A)(s)^{-1}$ as 
 $\sum_{r=0}^\infty\hat{\psi}_r(s,\epsilon)\lambda^r$ where of course $\psi_0(s,\epsilon)=1
 =\hat{\psi}_0(s,\epsilon)$. Fix a norm $||~||$ on the complex vector space $\mathcal{A}$. Since 
 $s=|s|\leqslant 1$ and for each non-negative integer $r$ there are positive real constants $C_r$ and $\hat{C}_r$ with $||\psi_r(s,\epsilon)||\leqslant C_r$ and $||\hat{\psi}_r(s,\epsilon)||\leqslant \hat{C}_r$
 independent on $s,\epsilon$ thanks
 to Lemma \ref{LDrinfeldAssFactorizationHalfAss} we get
 for each non-negative integer $r$
 \begin{equation*}
 \begin{split}
 \left|\left|\tilde{Z}(s,\delta)_r\right|\right| &= \left|\left|\left(
  e^{\lambda\ln(\epsilon)s\mathrm{ad}_{B}}
  	\left(\sum_{u,v=0}^\infty\psi_u(s,\epsilon)
  	Z(s,\delta)
  	\hat{\psi}_v(s,\epsilon)\lambda^{u+v}\right)
  	\right)_r\right|\right| \\
  	&= \left|\left|\sum_{v,w=0}^r
    \sum_{1\leqslant i<j\leqslant 4}f_{ij}(s,\delta)
    \frac{(\ln(\epsilon))^w s^w}{w!}f_{ij}(s,\delta)
      \left(\mathrm{ad}_{B}\right)^{\circ w}
  \left(\psi_{r-v-w}(s,\epsilon)A_{ij}
  \hat{\psi}_v(s,\epsilon)\right)
     \right|\right| \\
     & \leqslant 
  \sum_{w=0}^r
  	\sum_{1\leqslant i<j\leqslant 4}\left|f_{ij}(s,\delta)\right|
  	\frac{\ell^w|\ln(\delta)|^w}{w!}
  	\left|\left|\sum_{v=0}^r
  	\left(\mathrm{ad}_{B}\right)^{\circ w}
  	\left(\psi_{r-v-w}(s,\epsilon)A_{ij}
  	\hat{\psi}_v(s,\epsilon)\right)
  	\right|\right| \\   
  	 &\stackrel{(\ref{EqPentEqfijUpperBound})}
    {\leqslant} 
    \sum_{w=0}^r
    \sum_{1\leqslant i<j\leqslant 4} C_{ij} \delta^{\beta_{ij}}
    |\ln(\delta)|^wC'_w	
    ~\stackrel{(\ref{EqFilNormsDeltaTimesLnDeltaIneq})}
    		{\leqslant}~C_r''\delta^{\beta_r}     
 \end{split}
 \end{equation*}
 where the non-negative real constant $C'_w$ 
 ($0\leqslant w\leqslant r$) is an upper bound for the finite sum over $v$ of bounded algebra
 elements, $0<\beta_r$ is the minimum of all
 $\beta_{ij}/2$, $1\leqslant i<j\leqslant 4$, coming from inequality 
 (\ref{EqFilNormsDeltaTimesLnDeltaIneq}), and 
 the non-negative real number $C_r''$ is the maximum of all appearing non-egative multiplicative upper bounds. 
 This proves the last inequality in 
 (\ref{LEqFilNormsDefRBH}) and, according to
 Lemma \ref{LNormsLimIntegralBounds}, the fact that
 each $\delta\mapsto H^{(i,u,\delta)}_{10}$ is a harmless
 group term.\\
 In the following we prove the criterion (\ref{EqPentEqfijUpperBound}) for each path where the following
 elementary inequality will occur quite often:
 \begin{equation}\label{EqPentEqExponentialInequality}
   \forall~s\in[0,1]~\forall~\gamma\in J=~]0,1[:~~
     1\leqslant e^{-\ln(\gamma)s}
      \leqslant \frac{1}{\gamma}.
 \end{equation}
 Recall that the logarithms $\ln(2\delta)$,
 $\ln(2\hat{\delta})$, and $\ln\left(1-\frac{\delta}{2}\right)$ are non-positive numbers.\\
 
 \noindent $\mathrm{\RoM{1}}$. An elementary computation
 gives the following formulas for $\Gamma^{(\check{c}_{(\mathrm{\RoM{1}},1,\delta)})}$ and 
 $\Gamma^{(\tilde{c}_{(\mathrm{\RoM{1}},2,\delta)})}$
 showing that they are of the form 
 (\ref{EqPentEqGeneralFormOfGammaOfPath}) with $\epsilon=
 \delta$, with $B=A_{12}$, $A=A_{23}$ for the first path,
 and with $B=A_{23}$, $A=A_{12}$ for the second path:
 \begin{eqnarray*}
 \Gamma^{(\check{c}_{(\mathrm{\RoM{1}},1,\delta)})}(s)
 & = &
 \ln(2\delta)A_{12}
   +
 \frac{-\ln(2\delta)}
 {2e^{-\ln(2\delta)s}-1} A_{23}
 +
 \underbrace{
 \delta \frac{(-\ln(2\delta))}
        {2e^{-\ln(2\delta)s}-\delta}
       }_{=:f_{24}^{(\mathrm{\RoM{1}},1,\delta)}(s)} A_{24}, 
          \\
  \Gamma^{(\tilde{c}_{(\mathrm{\RoM{1}},2,\delta)})}(s)
  & = &
  \ln(2\delta)A_{23}
  +
  \frac{-\ln(2\delta)}
  {2e^{-\ln(2\delta)s}-1} A_{12}
  +
  \underbrace{
  \delta \frac{\ln(2\delta)}
  {2(1-\delta)e^{-\ln(2\delta)s}+\delta}
  }_{=:f_{24}^{(\mathrm{\RoM{1}},2,\delta)}(s)} A_{24}.         
 \end{eqnarray*}
 For both denominators in the expressions for
 $f^{(\mathrm{\RoM{1}},1,\delta)}(s)$
 and $f^{(\mathrm{\RoM{1}},2,\delta)}(s)$ the inequality
 (\ref{EqPentEqExponentialInequality}) gives us 
 the obvious lower bound $2-\delta> 1$ (for $s=0$), hence
 \begin{eqnarray*}
  \left|f_{24}^{(\mathrm{\RoM{1}},1,\delta)}(s)\right|
    & \leqslant & \delta\big(\ln(2)+|\ln(\delta)|\big)
   ~\stackrel{(\ref{EqFilNormsDeltaTimesLnDeltaIneq})}
       {\leqslant}~3\delta^{1/2},\\
   \left|f_{24}^{(\mathrm{\RoM{1}},2,\delta)}(s)\right|
   & \leqslant & \delta\big(\ln(2)+|\ln(\delta)|\big)
   ~\stackrel{(\ref{EqFilNormsDeltaTimesLnDeltaIneq})}
   {\leqslant}~3\delta^{1/2},     
 \end{eqnarray*}
 thanks to $\ln(2)\leqslant 1$ and $\delta\leqslant \delta^{1/2}$
 for all $\delta\in]0,1]$.
 By the criterion (\ref{EqPentEqfijUpperBound}) the terms $\delta\mapsto H^{(\mathrm{\RoM{1}},1,\delta))}$ and 
  $\delta\mapsto H^{(\mathrm{\RoM{1}},2,\delta))}$ in the factorization equation
   (\ref{EqPentEqFactorOfParTrHalfExpPath}) are thus harmless group terms. The factorization equation
   (\ref{EqPentEqFactorOfParTrHalfExpPath}) and (\ref{EqPentEqParTrAffinePEqProdExpP}) prove
  the first equation in (\ref{EqPentEqFiveParTrDeltaPositive}) upon setting
  $H^{(\mathrm{\RoM{1}},\delta))}\coloneqq
  H^{(\mathrm{\RoM{1}},2,\delta))} \big(H^{(\mathrm{\RoM{1}},1,\delta))}\big)^{-1}$.\\
  
  \noindent $\mathrm{\RoM{2}}$: An elementary, but lengthy computation
  gives the following formula for $\Gamma^{(\tilde{c}_{(\mathrm{\RoM{2}},2,\delta)})}$
  showing that it is of the form 
  (\ref{EqPentEqGeneralFormOfGammaOfPath}) with $\epsilon=
  \delta$, $B=A_{24}+A_{34}$, and $A=A_{12}+A_{13}$:
  \begin{eqnarray*}
   \Gamma^{(\tilde{c}_{(\mathrm{\RoM{2}},2,\delta)})}
   & = &
   \ln(2\delta)\left(A_{24}+A_{34}\right)
   +
   \frac{-\ln(2\delta)}
   {2e^{-\ln(2\delta)s}-1}\left(A_{12}+A_{13}\right) 
     \nonumber \\
   & &
   + ~\underbrace{
     \frac{\ln\left(1-\frac{\delta}{2}\right)
   	    -\delta^2\ln(2\delta)e^{-\ln(2\hat{\delta})s}}
   	      {1+\delta^2e^{-\ln(2\hat{\delta})s}}
   	   }_{=:f_{24}^{(\mathrm{\RoM{2}},2,\delta)}(s)} 
   	      A_{24}~
   	  + ~\underbrace{
   	  \frac{\ln\left(1-\frac{\delta}{2}\right)
   	  	+\delta^2\ln(2\delta)e^{-\ln(2\hat{\delta})s}}
   	  {1-\delta^2e^{-\ln(2\hat{\delta})s}}
   	}_{=:f_{34}^{(\mathrm{\RoM{2}},2,\delta)}(s)}
   	  A_{34} 
   	   \nonumber \\ 
   & &
   + ~\underbrace{
   \frac{
   	  \ln\left(1-\frac{\delta}{2}\right) 
   	  -\ln(2\hat{\delta})2e^{-\ln(2\delta)s}
   	  +\ln(2\delta)(2-\delta^2)e^{-\ln(2\hat{\delta})s}
   	   }{
   	   \left((2-\delta^2)e^{-\ln(2\hat{\delta})s}-1\right)
   	   \left(2e^{-\ln(2\delta)s}-1\right)
   	   }
   	}_{=:f_{12}^{(\mathrm{\RoM{2}},2,\delta)}(s)} 
   	   ~A_{12}
   	   \nonumber \\ 
   	 & &
   	 + ~\underbrace{
   	 	\frac{
   	 		\ln\left(1-\frac{\delta}{2}\right) 
   	 		-\ln(2\hat{\delta})2e^{-\ln(2\delta)s}
   	 		+\ln(2\delta)(2+\delta^2)e^{-\ln(2\hat{\delta})s}
   	 	}{
   	 	\left((2+\delta^2)e^{-\ln(2\hat{\delta})s}-1\right)
   	 	\left(2e^{-\ln(2\delta)s}-1\right)
   	 }
 }_{=:f_{13}^{(\mathrm{\RoM{2}},2,\delta)}(s)} 
 ~A_{13}   
  \end{eqnarray*}
  We shall now prove the upper bound (\ref{EqPentEqfijUpperBound}) for the four functions
  $f_{24}^{(\mathrm{\RoM{2}},2,\delta)}$,
  $f_{34}^{(\mathrm{\RoM{2}},2,\delta)}$,
  $f_{12}^{(\mathrm{\RoM{2}},2,\delta)}$, and
  $f_{13}^{(\mathrm{\RoM{2}},2,\delta)}$. Note first the
  following elementary inequality for all
   $0<\delta\leqslant 1/4$
  \begin{equation}
  \label{EqPentEqLnOneMinusDeltaHalfSmalDelta}
  \left|\ln\left(1-\frac{\delta}{2}\right)\right|
   =-\ln\left(1-\frac{\delta}{2}\right) \leqslant \delta.
  \end{equation}
  Indeed, for all $0<x\leqslant 1$ we have 
  $\frac{1}{x}\leqslant \frac{1}{x^2}$, hence
  \[
   -\ln\left(1-\frac{\delta}{2}\right)
   ~=~\int_{1-\frac{\delta}{2}}^1 \frac{1}{x}\mathsf{d}x
   ~\leqslant~
   \int_{1-\frac{\delta}{2}}^1 \frac{1}{x^2}\mathsf{d}x
   ~=~\frac{\delta}{2-\delta}~\leqslant~ \delta.
  \]
  For $f_{24}^{(\mathrm{\RoM{2}},2,\delta)}$ and
  $f_{34}^{(\mathrm{\RoM{2}},2,\delta)}$ we can bound
  both denominators by $1/2$ from below thanks to the lower bound $1$ in (\ref{EqPentEqExponentialInequality}).
  In the numerators the exponential function
  $e^{-\ln(2\hat{\delta})s}$ has an upper bound
  $\frac{1}{2\hat{\delta}}=
  \frac{1}{2\delta\left(1-\frac{\delta}{2}\right)}
  \leqslant \frac{1}{\delta}$ by (\ref{EqPentEqExponentialInequality}). Hence, both
  functions have the following upper bounds
  (where we also use the inequality
   (\ref{EqPentEqLnOneMinusDeltaHalfSmalDelta})):
   for all $s\in[0,1]$ and $\delta\in~]0,1/4]$
  \begin{equation*}
   \left\{
    \begin{array}{c}
     \left|f_{24}^{(\mathrm{\RoM{2}},2,\delta)}(s)\right| \\
     \left|f_{34}^{(\mathrm{\RoM{2}},2,\delta)}(s)
      \right|
    \end{array} \right.
    ~\leqslant~ 2\delta +2\ln(2\delta)\delta^2\frac{1}{\delta}
    \stackrel{(\ref{EqFilNormsDeltaTimesLnDeltaIneq})}
    {\leqslant} 8\delta^{\frac{1}{2}},  
  \end{equation*}
  where the inequalities $\delta\leqslant\delta^{\frac{1}{2}}$
  and $\ln(2) \leqslant 1$ have been used. It follows that
  the criterion (\ref{EqPentEqfijUpperBound})
  holds for 
   $f_{24}^{(\mathrm{\RoM{2}},2,\delta)}$ and
   $f_{34}^{(\mathrm{\RoM{2}},2,\delta)}$.\\
   For $f_{12}^{(\mathrm{\RoM{2}},2,\delta)}$ and
   $f_{13}^{(\mathrm{\RoM{2}},2,\delta)}$ note first
   that their numerators can be expressed in the following form
   where we have extracted a factor 
   $\frac{1}{2}e^{-\ln(2\delta)s}$:
   \[
    \frac{1}{2}e^{-\ln(2\delta)s}
    \left(
    -2\ln\left(1-\frac{\delta}{2}\right)
       \left(2-e^{\ln(2\delta)s}\right)
       +4\ln(2\delta)\left(
       e^{-\ln\left(1-\frac{\delta}{2}\right)s}-1
       \right)
       -2g\ln(2\delta)\delta^2
       e^{-\ln\left(1-\frac{\delta}{2}\right)s}
    \right).
   \]
   with $g\in\{-1,1\}$. On the other hand, in the denominators of $f_{12}^{(\mathrm{\RoM{2}},2,\delta)}$ and
   $f_{13}^{(\mathrm{\RoM{2}},2,\delta)}$
   we can bound
   both left factors from below by
   $\frac{1}{2}$ thanks to the lower bound $1$ in (\ref{EqPentEqExponentialInequality}), and both right factors from below
   by $e^{-\ln(2\delta)s}$. It follows that
   for all $s\in[0,1]$ and $\delta\in~]0,1/4]$
   \begin{eqnarray*}
    \left\{
    \begin{array}{c}
    \left|f_{12}^{(\mathrm{\RoM{2}},2,\delta)}(s)\right| \\
    \left|f_{13}^{(\mathrm{\RoM{2}},2,\delta)}(s)
    \right|
    \end{array} \right.
     & \leqslant &
      2\left|\ln\left(1-\frac{\delta}{2}\right)\right|
            \left(2-e^{\ln(2\delta)s}\right)
            + 4|\ln(2\delta)|
            \left(
            e^{-\ln\left(1-\frac{\delta}{2}\right)s}-1
            \right)
            \nonumber \\
            & &
            +
            2|\ln(2\delta)|\delta^2
            e^{-\ln\left(1-\frac{\delta}{2}\right)s}
            \nonumber \\
       & \leqslant & 4\delta 
               +4\big(\ln(2)+|\ln(\delta)|\big)
               \left(
               e^{-\ln\left(1-\frac{\delta}{2}\right)s}-1
               \right)
               +4\big(\ln(2)+|\ln(\delta)|\big)\delta^2. \nonumber \\
   \end{eqnarray*}
   Here we have used the inequality
    (\ref{EqPentEqLnOneMinusDeltaHalfSmalDelta}), the fact
    that $e^{\ln(2\delta)s}\geqslant 2\delta$ and
    $e^{-\ln\left(1-\frac{\delta}{2}\right)s}\leqslant
    1/(1-\delta/2)\leqslant 2$ thanks to the upper bound in (\ref{EqPentEqExponentialInequality}). Finally, we get,
    again by (\ref{EqPentEqExponentialInequality}) for
    $\gamma= \left(1-\frac{\delta}{2}\right)$,
    \[
      e^{-\ln\left(1-\frac{\delta}{2}\right)s}-1
      \leqslant \frac{1}{1-\frac{\delta}{2}} -1
      =\frac{\delta}{2-\delta}\leqslant \delta,
    \]
    and again by the inequality (\ref{EqFilNormsDeltaTimesLnDeltaIneq}) we get the final upper bound
    \[
    \left|f_{12}^{(\mathrm{\RoM{2}},2,\delta)}(s)\right| 
     \leqslant
     24\delta^{\frac{1}{2}}
     ~~~\mathrm{and}~~~
    \left|f_{13}^{(\mathrm{\RoM{2}},2,\delta)}(s)
    \right|
    \leqslant
      24\delta^{\frac{1}{2}}
    \]
    since $\ln(2)\leqslant 1$ and $\delta\leqslant\delta^{\frac{1}{2}}$.
    It follows that
    the criterion (\ref{EqPentEqfijUpperBound})
    holds for 
    $f_{12}^{(\mathrm{\RoM{2}},2,\delta)}$ and
    $f_{13}^{(\mathrm{\RoM{2}},2,\delta)}$.\\
    These upper bounds show that the parallel transport
    along the path 
    $\tilde{c}_{(\mathrm{\RoM{2}},2,\delta)}$ factorizes
    according to
    (\ref{EqPentEqFactorOfParTrHalfExpPath}):
    \begin{equation}
    \label{EqPentEqFactorOfParTrPathII2}
    W^{(c_{(\mathrm{\RoM{2}},2,\delta)})}_{10}~=~
    	e^{\lambda\ln(\delta)(A_{24}+A_{34})}~
    	\psi_\delta(A_{24}+A_{34},A_{12}+A_{13})
    ~
    H^{(\mathrm{\RoM{2}},2,\delta)},
    \end{equation}
    where $\delta \mapsto H^{(\mathrm{\RoM{2}},2,\delta)}$
    is a harmless term. Using
    (\ref{EqPentEqThetaPullBackOfGamma}) and the second equation of
    (\ref{EqPentEqAffinePathsOfPentagon})
    we can conclude that the parallel transport $W^{(\check{c}_{(\mathrm{\RoM{2}},1,\delta)})}_{10}$
    is given by formula (\ref{EqPentEqFactorOfParTrPathII2}) with the index
    change induced by the inversion permutation $\sigma$, see (\ref{EqPentEqThetaPullBackOfGamma}). Passing to the inverse of $W^{(\check{c}_{(\mathrm{\RoM{2}},1,\delta)})}_{10}$
    and using formula (\ref{EqPentEqParTrAffinePEqProdExpP}) we get the proof of the second equation of
    (\ref{EqPentEqFiveParTrDeltaPositive}) upon setting
    $H^{(\mathrm{\RoM{2}},\delta))}\coloneqq
    H^{(\mathrm{\RoM{2}},2,\delta))} \big(H^{(\mathrm{\RoM{2}},1,\delta))}\big)^{-1}$.\\
    
    \noindent $\mathbf{\mathrm{\RoM{3}}}$: Due to the symmetry
    $c_{(\mathrm{\RoM{3}},\delta)} =\Theta\circ 
    c_{(\mathrm{\RoM{1}},\delta)}\circ \iota$ we get the third formula of (\ref{EqPentEqFiveParTrDeltaPositive})
    by taking the first one, applying the inversion
    permutation $\sigma$, see (\ref{EqPentEqThetaPullBackOfGamma}), and passing
     to the inverse.\\
     
     \noindent $\mathrm{\RoM{4}}$:
     An elementary computation
     gives the following formulas for $\Gamma^{(\check{c}_{(\mathrm{\RoM{4}},1,\delta)})}$ and 
     $\Gamma^{(\tilde{c}_{(\mathrm{\RoM{4}},2,\delta)})}$
     showing that they are of the form 
     (\ref{EqPentEqGeneralFormOfGammaOfPath}) with $\epsilon=
     \delta$, with $B=A_{13}+A_{23}$, $A=A_{34}$ for the first path,
     and with $\epsilon=\delta^2$ (!), $B=A_{34}$, $A=A_{13}+A_{23}$ for the second path:
     \begin{eqnarray*}
     \Gamma^{(\check{c}_{(\mathrm{\RoM{4}},1,\delta)})}(s)
     & = &
     \ln(2\delta)\left(A_{13}+A_{23}\right)
     +
     \frac{-\ln(2\delta)}
     {2e^{-\ln(2\delta)s}-1} A_{34}
     +
     \underbrace{
     	\frac{2\ln(2\delta)\delta^2e^{-\ln(2\delta)s}}
     	{1-2\delta^2e^{-\ln(2\delta)s}}
     }_{=:f_{23}^{(\mathrm{\RoM{4}},1,\delta)}(s)} A_{23}, 
     \\
     \Gamma^{(\tilde{c}_{(\mathrm{\RoM{4}},2,\delta)})}(s)
     & = &
     \ln(2\delta^2)A_{34}
     +
     \frac{-\ln(2\delta^2)}
     {2e^{-\ln(2\delta^2)s}-1}\left( A_{13}+A_{23}\right)
     \nonumber \\
     & & -
     \underbrace{
     	 \frac{2\ln(2\delta^2)\delta^2e^{-\ln(2\delta^2)s}}
     	{
     	\left(2e^{-\ln(2\delta^2)s}-1\right)
     	\left(	2(1-\delta^2)
     	  e^{-\ln(2\delta^2)s}-1\right)}
     }_{=:f_{23}^{(\mathrm{\RoM{4}},2,\delta)}(s)} A_{23}.         
     \end{eqnarray*}
     The denominator of $\left|f_{23}^{(\mathrm{\RoM{4}},1,\delta)}(s)
    \right|$ can be bounded from below by $1/2$ upon using
    the upper bound in inequality
    (\ref{EqPentEqExponentialInequality}) for
    $\gamma=2\delta$. In the numerator we get the upper bound $\frac{1}{2\delta}$ for the exponential function,
    again thanks to (\ref{EqPentEqExponentialInequality}),
    hence
    \begin{equation}\label{EqPentEqIneqFIV1}
     \left|f_{23}^{(\mathrm{\RoM{4}},1,\delta)}(s)
    \right|\leqslant 
    2(\ln(2)+|\ln(\delta)|)\delta
     \stackrel{(\ref{EqFilNormsDeltaTimesLnDeltaIneq})}
    {\leqslant} 6\delta^{\frac{1}{2}},   
    \end{equation}
    hence the criterion 
    (\ref{EqPentEqfijUpperBound}) holds for
    $f_{23}^{(\mathrm{\RoM{4}},1,\delta)}(s)$.\\
    Next, the right factor in the denominator of
    $\left|f_{23}^{(\mathrm{\RoM{4}},2,\delta)}(s)
    \right|$ can be bounded from below by $1/2$ upon using
    the lower bound in inequality
    (\ref{EqPentEqExponentialInequality}) for
    $\gamma=2\delta^2$. We can bound the left factor
    in that denominator from below by 
    $e^{-\ln(2\delta^2)s}$, hence
    \begin{equation}\label{EqPentEqIneqFIV2}
     \left|f_{23}^{(\mathrm{\RoM{4}},2,\delta)}(s)
     \right| \leqslant
      4\big(\ln(2)+2|\ln(\delta)|\big)\delta^2
      \stackrel{(\ref{EqFilNormsDeltaTimesLnDeltaIneq})}
      {\leqslant}
      12\delta,
    \end{equation}
    hence the criterion 
    (\ref{EqPentEqfijUpperBound}) holds for
    $f_{23}^{(\mathrm{\RoM{4}},2,\delta)}(s)$.
    Both upper bounds (\ref{EqPentEqIneqFIV1}) and
    (\ref{EqPentEqIneqFIV2}) prove that both
    parallel transports
    $W^{(\check{c}_{(\mathrm{\RoM{4}},1,\delta)})}_{10}$
    and
    $W^{(\tilde{c}_{(\mathrm{\RoM{4}},2,\delta)})}_{10}$
    factorize in the way described in (\ref{EqPentEqFactorOfParTrHalfExpPath})
    with harmless group terms $\delta\mapsto H^{(\mathrm{\RoM{4}},1,\delta)}$ and
    $\delta\mapsto H^{(\mathrm{\RoM{4}},2,\delta)}$,
    respectively. This proves the fourth parallel transport
    equation in (\ref{EqPentEqFiveParTrDeltaPositive}) upon setting
    $H^{(\mathrm{\RoM{4}},\delta))}\coloneqq
    H^{(\mathrm{\RoM{4}},2,\delta))} \big(H^{(\mathrm{\RoM{4}},1,\delta))}\big)^{-1}$.
     \\

     \noindent $\mathrm{\RoM{5}}$: Due to the symmetry
     $c_{(\mathrm{\RoM{4}},\delta)} =\Theta\circ 
     c_{(\mathrm{\RoM{1}},\delta)}\circ \iota$ we get the fifth formula of (\ref{EqPentEqFiveParTrDeltaPositive})
     by taking the fourth one, applying the inversion
     permutation $\sigma$, see
     (\ref{EqPentEqThetaPullBackOfGamma}), and passing
     to the inverse.
  \end{proof}
  By means of these informations we can prove the Pentagon
  Equation:
  \begin{theorem}\label{TPentEqFinalProofOfPentEq}
  	The Pentagon Equation (\ref{EqPentDefPentagonEquation})
  	for the Drinfel'd associator holds.
  \end{theorem}
  \begin{proof}
  	According to Corollary 
  	\ref{CFormFlatContractibleLoops} we have the following
  	equation of parallel transports along the paths
  	(\ref{EqPentEqAffinePathsOfPentagon}) because $U'$ is star-shaped around $\big(\frac{1}{2},\frac{2}{3}\big)$, and the
  	two composed paths 
  	$c_{(\mathrm{\RoM{5}},\delta)}*
  	c_{(\mathrm{\RoM{4}},\delta)}$ and
  	$c_{(\mathrm{\RoM{3}},\delta)}*
  	   \big(c_{(\mathrm{\RoM{2}},\delta)}
  	*c_{(\mathrm{\RoM{1}},\delta)}\big)$ are both continuous and piecewise smooth and have the same
  	initial point $(\delta^2,\delta)$ and final point
  	$(1-\delta,1-\delta^2)$:
  	\begin{equation}\label{EqPentEqIdentityOfFiveParTransp}
  	W^{ (c_{(\mathrm{\RoM{5}},\delta)}) }
  	 ~W^{ (c_{(\mathrm{\RoM{4}},\delta)}) }
  	    ~=~
  	 W^{ (c_{(\mathrm{\RoM{3}},\delta)}) }
  	 ~W^{ (c_{(\mathrm{\RoM{2}},\delta)}) }
  	 ~W^{ (c_{(\mathrm{\RoM{1}},\delta)}) }.
  	\end{equation}
  	In view of the length of the formulas
  	(\ref{EqPentEqFiveParTrDeltaPositive}) of the preceding
  	Lemma \ref{TPentEqFiveParTrDeltaPositive} we define
  	the following abbreviations where $\epsilon,\epsilon'$
  	are monomials in $\delta$ (in practice $\delta$ or 
  	$\delta^2$), $A,B$ are certain linear combinations
  	of the elements $A_{ij}=A_{ji}\in\mathcal{A}$ for 
  	$1\leqslant i<j\leqslant 4$, and $i$ is an element
  	of
  	 $\{\mathrm{\RoM{1}},\mathrm{\RoM{2}},\mathrm{\RoM{3}},
  	\mathrm{\RoM{4}},\mathrm{\RoM{5}}\}$:
  	\begin{equation}\label{EqPentEqPhiEEPrimeBAHDElta}
  	  \Phi_{\epsilon,
  	  \epsilon'}\left(A,B,H^{(i,\delta)}\right)~\coloneqq~
  	  \psi_{\epsilon'}(B,A)~H^{(i,\delta)}
  	   ~\psi_{\epsilon}(A,B)^{-1}.
  	\end{equation}
  	We recall the relevant commutation relations for
  	the elements $A_{ij}$ coming from the conditions
  	(\ref{eq:DK-definition-two}) and
  	(\ref{eq:DK-definition-three}):
  	\begin{eqnarray}
  	\left[A_{12},A_{34}\right] &   =0, ~~&   
  	  \label{EqPentEqComm12With34}\\
  	\left[A_{12},A_{13}+A_{23}\right] &  =0= &  
  	\left[A_{12},A_{12}+A_{13}+A_{23}\right], 
  	   \label{EqPentEqComm12With13Plus23}\\
  	\left[A_{23},A_{12}+A_{13}\right] &  =0= &  
  	\left[A_{23},A_{12}+A_{13}+A_{23}\right], 
  	\label{EqPentEqComm23With12Plus13}\\
  	\left[A_{23},A_{24}+A_{34}\right] &  =0= &  
  	\left[A_{23},A_{23}+A_{24}+A_{34}\right], 
  	  \label{EqPentEqComm23With24Plus34}\\
  	\left[A_{34},A_{23}+A_{24}\right] &  =0= &  
  	\left[A_{34},A_{23}+A_{24}+A_{34}\right] .
  	  \label{EqPentEqComm34With23Plus24}
  	\end{eqnarray}
  Moreover, recall that if $A\in\mathcal{A}$ commutes with $B_1,\ldots,B_N\in\mathcal{A}$ then the formal
  exponential $e^{\lambda \gamma A}$ ($\gamma\in\mathbb{C}$) commutes with
  any formal series whose  coefficients consist of noncommutative polynomials in 
   $B_1,\ldots,B_N\in\mathcal{A}$, hence in particular
   \begin{equation}
   \label{EqPentEqCommAWithBZeroThenCommExp}
    [A,B]=0~~~\mathrm{implies}~~
         e^{\lambda \gamma A}e^{\lambda \gamma' B}
         =e^{\lambda (\gamma A+\gamma' B)}
         =e^{\lambda \gamma' B}e^{\lambda \gamma A}.
   \end{equation}
   As in the proof of the Hexagon Equation
   (\ref{EqHexDefHexagonEqn}) we denote the conjugation
   $L_\delta H^{(i,\delta)}L_\delta^{-1}$
   of a harmless term $\delta\mapsto H^{(i,\delta)}$ by an at most logarithmically divergent term 
   $\delta\mapsto L_\delta$ (as elements in
   $\big(\mathrm{Fun}\big(]0,1/4],\mathbb{C}\big)\otimes \mathcal{A}\big)[[\lambda]]$) by
   $\tilde{H}^{(i,\delta)}$. Note also that
   \begin{equation}\label{EqPentEqConjOfPhiH}
      L_\delta~\Phi_{\epsilon, \epsilon'}\left(A,B,H^{(i,\delta)}\right)~
      L_\delta^{-1}
      ~=~\Phi_{\epsilon, \epsilon'}\left(L_\delta A 
       L_\delta^{-1},L_\delta B 
       L_\delta^{-1},\tilde{H}^{(i,\delta)}\right).
   \end{equation}\\
   	We compute the left hand side of 
   	(\ref{EqPentEqIdentityOfFiveParTransp}) and try
   	to `push' the `singular terms' of the form
   	$e^{\lambda \ln(\epsilon) A}$ (with $\epsilon=\delta$
   	or $\epsilon=\delta^2$) to the left and to the right:
   	here (\ref{EqPentEqCommAWithBZeroThenCommExp})
   	will be used:  
   	
   \begin{equation}
   \label{EqPentEqParTransIVAndV}	
   \begin{split}
   W_{10}^{ (c_{(\mathrm{\RoM{5}},\delta)}) }
   	~W_{10}^{ (c_{(\mathrm{\RoM{4}},\delta)}) }  
   	  =\  &e^{\lambda\ln(\delta)(A_{23}+A_{24})}
   ~\Phi_{\delta^2, \delta
   }\left(A_{12},A_{23}+A_{24},
   H^{(\mathrm{\RoM{5}},\delta)}\right)
   ~e^{-\lambda\ln(\delta^2)A_{12}}	 \\
   & e^{\lambda\ln(\delta^2)A_{34}}
      ~\Phi_{\delta,\delta^2}\left(A_{13}+A_{23},A_{34},
      ~H^{(\mathrm{\RoM{4}},\delta)}\right)
      ~e^{-\lambda\ln(\delta)(A_{13}+A_{23})} \\
      =\ & e^{\lambda\ln(\delta)(A_{23}+A_{24})}
     ~\Phi_{\delta^2, \delta}\left(A_{12},A_{23}+A_{24},
     H^{(\mathrm{\RoM{5}},\delta)}\right)
     ~e^{\lambda\ln(\delta^2)A_{34}}	  \\
     & e^{\lambda\ln(\delta^2)A_{34}}
      ~\Phi_{\delta,\delta^2}\left(A_{13}+A_{23},A_{34},
      ~H^{(\mathrm{\RoM{4}},\delta)}\right)
      ~e^{-\lambda\ln(\delta)(A_{13}+A_{23})} \\
      = \ & e^{\lambda\ln(\delta)(A_{23}+A_{24}+2A_{34})}
      ~\Phi_{\delta^2, \delta}\left(A_{12},A_{23}+A_{24},
      \tilde{H}^{(\mathrm{\RoM{5}},\delta)}\right)	   \\
     & \Phi_{\delta,\delta^2}\left(A_{13}+A_{23},A_{34},
      ~\tilde{H}^{(\mathrm{\RoM{4}},\delta)}\right)
      ~e^{-\lambda\ln(\delta)(2A_{12}+A_{13}+A_{23})} ,
   \end{split}
   \end{equation}
   	where the first equality follows by (\ref{EqPentEqFiveParTrDeltaPositive}), the second by (\ref{EqPentEqComm12With34}), and the third by (\ref{EqPentEqComm12With34}),(\ref{EqPentEqComm34With23Plus24}),(\ref{EqPentEqComm12With13Plus23}),
(\ref{EqPentEqConjOfPhiH}),(\ref{EqPentEqCommAWithBZeroThenCommExp}).
   	Next, we compute the right hand side of (\ref{EqPentEqIdentityOfFiveParTransp}) in a similar way: 
   	\begin{eqnarray}
   	\lefteqn{
   	W_{10}^{ (c_{(\mathrm{\RoM{3}},\delta)}) }~
   	W_{10}^{ (c_{(\mathrm{\RoM{2}},\delta)}) }~
   	W_{10}^{ (c_{(\mathrm{\RoM{1}},\delta)}) }
   	} \nonumber \\
   	&\stackrel{(\ref{EqPentEqFiveParTrDeltaPositive})}
   	{=}&
   	 e^{\lambda\ln(\delta)A_{34}}
   	 ~\Phi_{\delta,\delta}\left(A_{23},A_{34},
   	 ~H^{(\mathrm{\RoM{3}},\delta)}\right)
   	 ~e^{-\lambda\ln(\delta)A_{23}}
   	  \nonumber \\
   	 &  &
   	 e^{\lambda\ln(\delta)(A_{24}+A_{34})}
   	 ~\Phi_{\delta,\delta}\left(A_{12}+A_{13},A_{24}+A_{34},
   	   H^{(\mathrm{\RoM{2}},\delta)}\right)
   	    ~e^{-\lambda\ln(\delta)(A_{12}+A_{13})}
   	 \nonumber \\
   	 & &
   	 e^{\lambda\ln(\delta)A_{23}}
   	 ~\Phi_{\delta,\delta}\left(A_{12},A_{23},
   	 H^{(\mathrm{\RoM{1}},\delta)}\right)
   	 ~e^{-\lambda\ln(\delta)A_{12}}
   	    \nonumber \\
   	 &\stackrel{(\ref{EqPentEqComm23With24Plus34}),
   	 	(\ref{EqPentEqComm23With12Plus13})}
   	 {=}&
   	 e^{\lambda\ln(\delta)A_{34}}
   	 ~\Phi_{\delta,\delta}\left(A_{23},A_{34},
   	 ~H^{(\mathrm{\RoM{3}},\delta)}\right)
   	 ~e^{\lambda\ln(\delta)(A_{24}+A_{23}+A_{34})}
   	  \nonumber \\
   	  & &
   	  e^{-\lambda\ln(\delta)2A_{23}}
   	  ~\Phi_{\delta,\delta}\left(A_{12}+A_{13},A_{24}+A_{34},
   	  H^{(\mathrm{\RoM{2}},\delta)}\right)
   	  ~e^{\lambda\ln(\delta)2A_{23}}
   	  \nonumber \\
   	  & &
   	  e^{-\lambda\ln(\delta)(A_{12}+A_{13}+A_{23})}
   	  ~\Phi_{\delta,\delta}\left(A_{12},A_{23},
   	  H^{(\mathrm{\RoM{1}},\delta)}\right)
   	  ~e^{-\lambda\ln(\delta)A_{12}}
   	  \nonumber \\
   	 &\stackrel{(\ref{EqPentEqComm23With24Plus34}),
   	 	(\ref{EqPentEqComm34With23Plus24}),
   	 	(\ref{EqPentEqComm12With13Plus23}),
   	 	(\ref{EqPentEqComm23With12Plus13})
   	 	(\ref{EqPentEqConjOfPhiH})}{=} &
   	  e^{\lambda\ln(\delta)(A_{23}+A_{24}+2A_{34})}
   	  ~\Phi_{\delta,\delta}\left(A_{23},A_{34},
   	  ~\tilde{H}^{(\mathrm{\RoM{3}},\delta)}\right)
   	  \nonumber \\
   	 & & \Phi_{\delta,\delta}\left(A_{12}+A_{13},A_{24}+A_{34},
   	  \tilde{H}^{(\mathrm{\RoM{2}},\delta)}\right)
   	  \nonumber \\
   	 & & \Phi_{\delta,\delta}\left(A_{12},A_{23},
   	  \tilde{H}^{(\mathrm{\RoM{1}},\delta)}\right)
   	  ~e^{-\lambda\ln(\delta)(2A_{12}+A_{13}+A_{23})}.
   	  \label{EqPentEqParTransIAndIIAndIII}
   	\end{eqnarray}
   	A comparison of the preceding equations
   	(\ref{EqPentEqParTransIVAndV})
   	and (\ref{EqPentEqParTransIAndIIAndIII}) 
   	immediately shows that
   	the singular terms
   	 $e^{\lambda\ln(\delta)(A_{23}+A_{24}+2A_{34})}$
   	 and $e^{-\lambda\ln(\delta)(2A_{12}+A_{13}+A_{23})}$
   	 cancel out in (\ref{EqPentEqIdentityOfFiveParTransp}),
   	 leaving only products of terms of type
   	 (\ref{EqPentEqPhiEEPrimeBAHDElta}) which tend to
   	 the desired product of Drinfel'd associators
   	 yielding the Pentagon Equation
   	 (\ref{EqPentDefPentagonEquation}) in the limit
   	 $\delta\to 0$
   	  thanks to the limit rules (\ref{EqNormsLimLimitProdAreProdOfLimits}),
   	  the definition of the Drinfel'd associator
   	  (\ref{EqDrinfeldAssDefAss}), and the fact that harmless
   	  group terms tend to $1$ for $\delta\to 0$ (see
   	  statement $v.)$ of Proposition \ref{PFilNormsRBHProperties}). 
  \end{proof}

\appendix

\section{Some Proofs of Theorems in Section 1.}

\subsection{Proof of Proposition
	 \ref{PFormConnComplexPullBackEqualReal}}
	
	\begin{proof}
		We compute using (\ref{EqFormFlatRealFormCompRatConn}):
		\begin{eqnarray*}
			\left(\check{\Theta}^*\check{\Gamma}\right)(x',y')
			& = & \sum_{j=1}^{N'}\sum_{i=1}^{N}\Bigg(
			\check{\Gamma}^{[1]}_i\left(\Theta^{(1)}(x',y'),
			\Theta^{(2)}(x',y')\right)
			\dfrac{\partial\Theta_i^{(1)}}{\partial x'_j}(x',y')
			\mathsf{d}x'_j \\
			&   & \qquad \qquad +\check{\Gamma}^{[1]}_i
			\left(\Theta^{(1)}(x',y'),\Theta^{(2)}(x',y')\right)
			\dfrac{\partial\Theta_i^{(1)}}{\partial y'_j}(x',y')
			\mathsf{d}y'_j \\
			&   & \qquad \qquad +\check{\Gamma}^{[2]}_i
			\left(\Theta^{(1)}(x',y'),\Theta^{(2)}(x',y')\right)
			\dfrac{\partial\Theta_i^{(2)}}{\partial x'_j}(x',y')
			\mathsf{d}x'_j \\
			&   & \qquad \qquad +\check{\Gamma}^{[2]}_i
			\left(\Theta^{(1)}(x',y'),\Theta^{(2)}(x',y')\right)
			\dfrac{\partial\Theta_i^{(2)}}{\partial y'_j}(x',y')
			\mathsf{d}y'_j \Bigg),
		\end{eqnarray*}	
		hence with (\ref{EqFormFlatRealFormCompRatConn})
		and (\ref{EqFormConnComplexRealPhi}) we get

		\begin{eqnarray*}
			\left(\check{\Theta}^*\check{\Gamma}\right)(x',y')
			& 
			\stackrel{(\ref{EqFormConnComplexFunctionRealImPart})}
			{=}
			& \sum_{j=1}^{N'}\sum_{i=1}^{N}
			\Gamma_i\big(\Theta(z')\big)
			\left(\dfrac{\partial \check{\Theta_i}}{\partial x'_j}(x',y')
			\mathsf{d}x'_j
			+\dfrac{\partial \check{\Theta_i}}{\partial y'_j}(x',y')
			\mathsf{d}y'_j
			\right)\\
			& \stackrel{(\ref{EqFormFlatHolomorphicCondandDer})}{=} &
			\sum_{j=1}^{N'}\sum_{i=1}^{N}
			\big(\Gamma_i\circ\Theta\big)^\vee (x',y')
			\left(\dfrac{\partial \Theta_i}
			{\partial z'_j}\right)^\vee(x',y')
			\mathsf{d}z'_j
			= \left(\sum_{j=1}^{N'}\left(\Theta^*\Gamma\right)_j
			\mathsf{d}z'_j\right)^\vee(x',y'),
		\end{eqnarray*}
		which proves the Proposition.
	\end{proof}
	
\subsection{Proof of Proposition
	\ref{PFlatConnPullBackOfFlatIsFlat}}

\begin{proof}
	We get
	\begin{eqnarray*}
		\dfrac{\partial \Gamma'_u}{\partial x'_v} - \dfrac{\partial \Gamma'_v}{\partial x'_u} + \lambda \Big(\Gamma'_u \Gamma'_v - \Gamma'_v \Gamma'_u\Big)
		& = & 
		\sum_{i=1}^N \dfrac{\partial \left((\Gamma_i\circ \Theta)
			\dfrac{\partial \Theta_i}{\partial x'_u}\right) }
		{\partial x'_v} 
		- 
		\sum_{j=1}^N \dfrac{\partial \left((\Gamma_j\circ \Theta)
			\dfrac{\partial \Theta_j}{\partial x'_v}\right) }
		{\partial x'_u}  \\ 
		& &  + \lambda\sum_{i,j=1}^N 
		\Big((\Gamma_i\circ\Theta)(\Gamma_j\circ\Theta)
		-(\Gamma_j\circ\Theta)(\Gamma_i\circ\Theta)\Big)
		\dfrac{\partial \Theta_i}{\partial x'_u} 
		\dfrac{\partial \Theta_j}{\partial x'_v}\\
		& = & \sum_{i=1}^N \left(\dfrac{\partial^2 \Theta_i}
		{\partial x'_v\partial x'_u}
		-
		\dfrac{\partial^2 \Theta_i}
		{\partial x'_u\partial x'_v}\right)
		(\Gamma_i\circ \Theta) \\
		& & +\sum_{i,j=1}^N\left(\left(
		\dfrac{\partial \Gamma_i}{\partial x_j} - \dfrac{\partial \Gamma_j}{\partial x_i} + \lambda \Big(\Gamma_i \Gamma_j - \Gamma_j \Gamma_i\Big)\right)\circ \Phi\right)
		\dfrac{\partial \Theta_i}{\partial x'_u} 
		\dfrac{\partial \Theta_j}{\partial x'_v}\\
		& = & 0+0=0,
	\end{eqnarray*}
	thanks to the chain rule, Schwartz's rule and to the
	flatness of $\Gamma$ whence $\Gamma'$ is flat.
\end{proof}

 \subsection{Proof of Theorem \ref{TFormFlatPathIndep}}

\begin{proof}
	Consider the smooth map 
	$\Gamma^{(F)}\in \Big(\mathcal{C}^\infty(\mathcal{O},\mathbb{C})\otimes
	\mathcal{A}\Big)[[\lambda]]$ given by
	\begin{equation}
	\label{eq:gammact}
	\Gamma^{(F)}(s,t) \coloneqq \sum_{i=1}^N \Gamma_i \big(F(s,t)\big) \dfrac{\partial F_i}{\partial s}(s,t).
	\end{equation}
	Moreover, for each $(s,t)\in\mathcal{O}$ let $W^{(F)}_{\cdot a}(s,t)$ denote the parallel transport from $p$ to $F(s,t)$ along the smooth path
	$s\to F(s,t)$, i.e.~$W^{(F)}_{\cdot a}$ (we suppress the symbol $\Gamma$ attached to $W$ in this proof) satisfies the differential equation
	\begin{equation}
	\label{eq:differential-equation}
	\dfrac{\partial W_{\cdot a}^{(F)} }{\partial s} = \lambda\Gamma^{(F)} W_{\cdot a}^{(F)}.
	\end{equation}
	Since $\Gamma$ and $F$ are smooth, it follows that $(s,t) \mapsto W_{\cdot a}^{(F)}$ is smooth, and so it is an element of 
	$ \big( C^\infty(\mathcal{O}, \mathbb{C}) \otimes A \big) [[\lambda]]$: indeed since $W_{\cdot a}^{(F)}$ is made out of iterated integrals (in the $s$-direction) the claim follows from the usual rule of differentiation of integrals
	depending on a parameter:
	\[ 
	\dfrac{\partial}{\partial t} \int_a^s f(t,s') \mathsf{d} s' 
	= 
	\int_a^s \dfrac{\partial f}{\partial t} (t,s') \mathsf{d} s'.
	\]
	Differentiating equation \eqref{eq:differential-equation} with respect to $t$, and using equation \eqref{eq:gammact}, we get
	--upon using the Schwartz rule that all partial derivatives commute--
	\begin{equation*}
	\footnotesize
	\begin{split}
	\dfrac{\partial^2 W_{\cdot a}^{(F)}}
	{\partial t \partial s}
	&=\lambda \dfrac{\partial}{\partial t} 
	\bigg( \sum_{i=1}^N \big(\Gamma_i\circ F\big) \dfrac{\partial F_i}{\partial s}
	W_{\cdot a}^{(F)}\bigg) \\
	&= \lambda\sum_{i=1}^N \sum_{j=1}^N 
	\left(\dfrac{\partial \Gamma_i}{\partial x_j} 
	\circ F \right)
	\dfrac{\partial F_j}{\partial t} \dfrac{\partial F_i}{\partial s} 
	W_{\cdot a}^{(F)} 
	+ \lambda\sum_{i=1}^N \left(\Gamma_i\circ F\right)
	\dfrac{\partial^2 F_i}{\partial t \partial s}
	W_{\cdot a}^{(F)}
	+ 
	\lambda\sum_{i=1}^N \left(\Gamma_i\circ F\right) \dfrac{\partial F_i}{\partial s} 
	\dfrac{\partial W_{\cdot a}^{(F)}}{\partial t} \\
	&
	\stackrel{(\ref{EqFlatDefFlatConnection})}{=} 
	\underline{
		\lambda\sum_{i=1}^N \sum_{j=1}^N 
		\left(\dfrac{\partial \Gamma_j}{\partial x_i}
		\circ F\right)
		\dfrac{\partial F_j}{\partial t} 
		\dfrac{\partial F_i}{\partial s} 
		W_{\cdot a}^{(F)}
	} \\
	& ~~~~~~	  -\lambda^2 
	\sum_{i=1}^N \sum_{j=1}^N 
	\left(\Gamma_i\circ F\right)
	\left(\Gamma_j\circ F\right)
	\dfrac{\partial F_j}{\partial t} 
	\dfrac{\partial F_i}{\partial s} 
	W_{\cdot a}^{(F)}
	+\underline{
		\lambda^2
		\sum_{i=1}^N \sum_{j=1}^N 
		\left(\Gamma_j\circ F\right)
		\left(\Gamma_i\circ F\right)
		\dfrac{\partial F_j}{\partial t} 
		\dfrac{\partial F_i}{\partial s} 
		W_{\cdot a}^{(F)}
	}
	\\
	& ~~~~~~ 	
	+\underline{\lambda 
		\sum_{i=1}^N \left(\Gamma_i\circ F\right)
		\dfrac{\partial^2 F_i}{\partial t \partial s}
		W_{\cdot a}^{(F)}
	}
	+ 
	\lambda\sum_{i=1}^N \left(\Gamma_i\circ F\right) \dfrac{\partial F_i}{\partial s} 
	\dfrac{\partial W_{\cdot a}^{(F)}}{\partial t}
	\\
	&
	\stackrel{(\ref{eq:differential-equation})}{=}	
	\underline{
		\lambda\dfrac{\partial}{\partial s} \left( 
		\sum_{i=1}^N \left(\Gamma_i\circ F\right) \dfrac{\partial F_i}{\partial t}
		W_{\cdot a}^{(F)}
		\right)
	} 
	+ \lambda\Gamma^{(F)}
	\left( \dfrac{\partial 
		W^{(F)}_{\cdot a}}{\partial t} 
	- \lambda \sum_{j=1}^N 
	\left(\Gamma_j \circ F\right) 
	\dfrac{\partial F_j}{\partial t} 
	W_{\cdot a}^{(F)}\right).
	\end{split}
	\end{equation*}
	Hence, setting 
	\[ 
	H  \coloneqq 
	\dfrac{\partial 
		W^{(F)}_{\cdot a}}{\partial t} 
	- \lambda \sum_{j=1}^N 
	\left(\Gamma_j \circ F\right) 
	\dfrac{\partial F_j}{\partial t} 
	W_{\cdot a}^{(F)}~
	\in ~\Big(\mathcal{C}^\infty(\mathcal{O},\mathbb{C})\otimes
	\mathcal{A}\Big)[[\lambda]],
	\]
	the preceding equation gives us the formal linear ODE
	\begin{equation} \label{EqFormFlatODEForH}
	\dfrac{\partial H }{\partial s} = 	\lambda \Gamma^{(F)}  H
	\end{equation}
	with initial condition at $s=a$ for each 
	$t\in \mathcal{O}''$ with $\mathcal{O}''=\{t\in\mathbb{R}~|~(a,t)\in \mathcal{O}\}$, note that $[0,1]\subset 
	\mathcal{O}''$:
	\begin{equation*}
	\begin{split}
	H(a,t) &= \dfrac{\partial W_{a a}^{(F)}(t) }
	{\partial t} 
	- \lambda \sum_{j=1}^N \Gamma_j\big(F(a,t)\big)
	\dfrac{\partial F_j}{\partial t} (a,t)  W_{a a}^{(F)}(t) 
	\stackrel{(\ref{EqFormFlatBoundaryCondHomotF})}{=}	 
	\dfrac{\partial 1 }{\partial t}(t)   
	- \lambda \sum_{j=1}^N \Gamma_j(p) 
	\dfrac{\partial p }{\partial t}(t)  = 0.
	\end{split}
	\end{equation*}
	Hence the formal linear ODE (\ref{EqFormFlatODEForH}) has the unique solution
	\[ 
	H(s,t) = 0 \qquad \forall (s,t) \in 
	\mathcal{O}\supset [a,b] \times [0,1]. 
	\]
	It follows by the definition of $H$ that there is the following formal linear ODE with respect to $t$: 
	\[ 
	\dfrac{W_{\cdot a}^{(F)}}{\partial t} (s,t)
	= \lambda \sum_{j=1}^N  \Gamma_j\big(F(s,t)\big) \dfrac{\partial F_j}{\partial t}(s,t) 
	W_{\cdot a}^{(F)}(s,t),
	\] 
	and since there is no derivative with respect to $s$ in this equation, we can set $s=b$ and get 
	\[ 
	\dfrac{W_{b a}^{(F)}}{\partial t}(t) 
	= 
	\lambda \sum_{j=1}^N  \Gamma_j\big(F(b,t)\big) \dfrac{\partial F_j}{\partial t}(b,t) 
	W_{b a}^{(F)}(t) \stackrel{(\ref{EqFormFlatBoundaryCondHomotF})}{=}
	\lambda \sum_{j=1}^N  \Gamma_j\big(q\big) \dfrac{\partial q}{\partial t}(t) 
	W_{b a}^{(F)}(t)
	= 0.
	\] 
	It follows that the parallel transport 
	$W_{b a }^{(F)}$ does not depend on $t$, hence in particular 
	\[ 
	W_{b a}^{(c_0)} =W_{ba}^{(F)}(t=0)= 
	W_{ba}^{(F)}(t=1) = W_{b a }^{(c_1)}.
	\]
\end{proof}

\subsection{Proof of 
	   Corollary \ref{CFormFlatContractibleLoops}}
We first need 
the
following well-known Smoothing Lemma which is a technical tool allowing for smoothing reparametrizations: it will only be needed in the proof of Corollary \ref{CFormFlatContractibleLoops}:
\begin{lemma}[\textbf{Smoothing Lemma}] 
	\label{LFlatFormConnSmoothing}
	Let $a<b$ be real numbers,
	$\{a,b\}\subset D\subset [a,b]$ a finite subset
	whose elements are given by $a_0=a<a_1<\cdots<a_m<
	a_{m+1}=b$, and
	let $c:[a,b]\to U\subset \mathbb{R}^N$ be a continuous piecewise smooth path.
	Then there exists a smooth map 
	$\theta:\mathbb{R}\to [a,b]$ such that its restriction to $[a,b]$ is strictly monotonous
	and surjective (hence has a continuous inverse on $[a,b]$),
	it induces the identity map on $D$, and whose higher derivatives all vanish at the points of
	$D$.
	Moreover, the composition $c\circ\theta:\mathbb{R}\to [a,b]$ (in the sense of (\ref{EqPiecewiseDefComposition})) is an everywhere well-defined smooth
       map all of whose higher derivatives (for $r\geqslant 1$) vanish at all points of $D$.
       For all $s\leqslant a$ the map $c\circ\theta$ takes the constant value $c(a)$, and for
       all $s\geqslant b$ it takes the constant value $c(b)$. In particular, for every positive
     $\epsilon$ the restriction of
      $c\circ\theta$ to $[a-\epsilon,b+\epsilon]$ yields a reparametrized path which is
     smooth.
\end{lemma}
\begin{proof}
	Let $\rho:\mathbb{R}\to \mathbb{R}$ be the following function,
	well-known from analysis and differential geometry:
	\begin{equation*}
	\rho(s)\coloneqq \left\{\begin{array}{cl}
	e^{-{1/s}} & \mathrm{if}~s>0, \\
	0       & \mathrm{if}~s\leqslant 0.
	\end{array}\right.
	\end{equation*}
	It is well-known and not hard to see that $\rho$
	is $\mathcal{C}^\infty$, has all of its higher derivatives equal to zero at $0$ and only strictly positive values for $s>0$. Define the function
	$\vartheta:\mathbb{R}\to \mathbb{R}$ by
	\begin{equation*}
	\vartheta(s)\coloneqq
	\mathlarger{\sum}_{i=0}^m \big(a_{i+1}-a_i\big)
	\frac{\rho(s-a_i)\rho(a_{i+1}-s)}
	{\int_{a_i}^{a_{i+1}}\rho(s'-a_i)\rho(a_{i+1}-s')
		\mathsf{d}s'}
	\end{equation*}
	and $\theta:\mathbb{R}\to \mathbb{R}$ by the primitive of $\vartheta$:
	\begin{equation*}
	\theta(s)\coloneqq a+\int_a^s\vartheta(s')\mathsf{d}s'.
	\end{equation*}
	Then all the properties of $\theta$ follow from the fact that $a_{i+1}-a_{i}>0$ and that the smooth
	function $s\mapsto \rho(s-a_i)\rho(a_{i+1}-s)$
	is strictly positive on $]a_i,a_{i+1}[$ and
	zero outside $]a_i,a_{i+1}[$. Moreover, it is clear
	that the composition $c\circ \theta$ is smooth on
	$[a,b]\setminus D$ as a composition of smooth maps.
	By the iterated chain rule it follows that all
	the higher derivatives of $c\circ \theta$ tend to zero at the points of $D$ since all the higher derivatives of
	$\theta$ go to zero at these points whereas all the higher left-side and right-side derivatives of $c$
	remain bounded.
\end{proof} 

\begin{proof} (of 
	Corollary \ref{CFormFlatContractibleLoops}):\\
	Let $c:[0,1]\to U'\subset U$ be the continuous piecewise smooth loop with $c(0)=c(1)=p$ defined
	by the composition $(c_2\circ \iota)*c_1$
	where $\iota:[0,1]\to[0,1]$ is the interval
	inversion $\iota(s)=1-s$. Furthermore,
	let $d:[0,1]\to U'$ be the affine path joining
	$\varpi$ with $p$, i.e.~$d(s)=(1-s)\varpi+sp$. Let
	$\check{c}:[0,1]\to U'$ be the piecewise smooth path
	$\check{c}\coloneqq \big(d\circ \iota\big)*(c*d)$
	obtained by composition of piecewise smooth paths.  Clearly, $\check{c}$ is a continuous piecewise
	smooth loop based at $\varpi$.
	Choose a smooth reparametrization $\theta$ of the
	path $\check{c}$ in the sense of the preceding Lemma 
	\ref{LFlatFormConnSmoothing}. Recall that
	$\theta$ is a smooth map $\mathbb{R}\to [0,1]$
	with $\theta(s)=0$ for all $s\leqslant 0$ and
	$\theta(s)=1$ for all $s\geqslant b$. Thanks to
	(\ref{EqFormConnWReparamInv}) and to the fact that
	$\theta(0)=0$, $\theta(1)=1$ we have the following equality of parallel transports
	\begin{equation}\label{EqFlatConnConjugationOfLoop}
	{}^\Gamma W_{10}^{(\check{c}\circ\theta)}
	~=~{}^\Gamma W_{10}^{(\check{c})}
	~=~{}^\Gamma W_{10}^{((d\circ \iota)*(c*d))}
	~\stackrel{(\ref{EqFormConnComposOfPathsParTransp})}{=}~
	{ {}^\Gamma W_{10}^{(d)}}^{-1}
	~{}^\Gamma W_{10}^{(c)}
	~{}^\Gamma W_{10}^{(d)}.
	\end{equation}
	Next, the map $\tilde{F}:\mathbb{R}^2\to \mathbb{R}^N$
	defined by
	\begin{equation*}
	\tilde{F}(s,t)= (1-t)\check{c}\big(\theta(s)\big)+tq    	
	\end{equation*}
	is clearly smooth, hence in particular continuous, whence the inverse image $\mathcal{O}
	:= \tilde{F}^{-1}(U')$ is an open subset of 
	$\mathbb{R}^2$ which contains the rectangle
	$[0,1]\times [0,1]$ thanks to the hypothesis that
	$U'$ is star-shaped around $\varpi$ and that all the points of the loop $c$ and hence of $\check{c}$ are in $U'$. By compactness of
	$[0,1]$ there is a strictly positive real number $\epsilon$ such that
	the open rectangle $]-\epsilon,1+\epsilon[ ~\times~ 
	]-\epsilon,1+\epsilon[$ is still contained in
	$\mathcal{O}$ thanks to the Heine-Borel Theorem.  If you think that this is no longer undergraduate analysis there is the following first year argument: if for each non-negative integer $n$ there was $t_n\in [0,1]$ such that $\big(-\frac{1}{n+1},t_n\big)\notin \mathcal{O}$ there would be subsequence
	$\left(t_{n_k}\right)_{k\in\mathbb{N}}$ converging to $\tau\in[0,1]$ by the Bolzano-Weierstrass Theorem.
	However, the limit point $(0,\tau)$ in $\mathbb{R}^2$ is contained in the open subset $\mathcal{O}$, hence
	there is a strictly positive real number $\delta$ with
	$(0,\tau)\in [0,0-\delta[~\times~]\tau-\delta,\tau+\delta[~\subset
	\mathcal{O}$. But then nearby elements of the subsequence $\big(-\frac{1}{n_k+1},t_{n_k}\big)$ would also be in $\mathcal{O}$
	contrary to the hypothesis. By a reasonable iteration of this argument the statement is proved. Let $F:\mathcal{O}\to U'\subset \mathcal{U}\subset \mathbb{R}^N$ denote the restriction of $\tilde{F}$ to $\mathcal{O}$. Then
	$F$ clearly satisfies all the hypotheses
	of Theorem \ref{TFormFlatPathIndep}: for all
	$s\in ~]-\epsilon,1+\epsilon[$ we have 
	$F(s,0)=\check{c}\big(\theta(s)\big)=:c_{(0)}(s)$ and
	$F(s,1)=\varpi =:c_{(1)}(s)$ (the constant loop at $\varpi$), and
	of course for all $t\in~ ]-\epsilon,1+\epsilon[$ we get $F(0,t)=\varpi=F(1,t)$.\\
	By Theorem \ref{TFormFlatPathIndep} we get
	\[
	{}^\Gamma W_{10}^{(\check{c}\circ \theta)}
	={}^\Gamma W_{10}^{(c_{(1)})}=1
	\]
	since $c_{(1)}$ is the constant loop whence
	$\Gamma^{(c_{(1)})}=0$. Thanks to equation
	(\ref{EqFlatConnConjugationOfLoop}) we get
	\[
	1={}^\Gamma W_{10}^{(c)} = 
	\left({}^\Gamma W_{10}^{(c_2)}\right)^{-1}
	{}^\Gamma W_{10}^{(c_1)}
	\]
	which proves the statement. The case of a 
	continuous piecewise smooth loop
	$c_3$ 
	based at $p$ is a particular case of the preceding
	statement upon choosing
	the constant loop $c_4$ at $p$ as a second path.
\end{proof}

\vskip24pt

{\bigskip{\footnotesize%
 (M.~Bordemann) \textsc{ D\'{e}partement de Math\'{e}matiques, Laboratoire IRIMAS,
   Universit\'{e} de Haute Alsace,
    18, rue des Fr\`{e}res Lumi\`{e}re,
    68093 Mulhouse, France.} \\ 
  \textit{E-mail address:}  \texttt{Martin.Bordemann@uha.fr}, \par
  \addvspace{\medskipamount}

  (A.~Rivezzi) \textsc{Dipartimento di Matematica e Applicazioni, Universit\`{a} di Milano-Bicocca,
        Via R.Cozzi 55, 20125 Milano, Italy.}

 Current address: \textsc{ D\'{e}partement de Math\'{e}matiques, Laboratoire IRIMAS,
   Universit\'{e} de Haute Alsace,
    18, rue des Fr\`{e}res Lumi\`{e}re,
    68093 Mulhouse, France.}\\
  \textit{E-mail address:} \texttt{a.rivezzi@campus.unimib.it}, \par
  
   \addvspace{\medskipamount}
   (Th.~Weigel) \textsc{Dipartimento di Matematica e Applicazioni, Universit\`{a} di Milano-Bicocca,
        Via R.Cozzi 55, 20125 Milano, Italy}
        \textit{E-mail address:} \texttt{thomas.weigel@unimib.it}.
   
}}

\end{document}